\documentclass[11pt,reqno]{amsart}

\usepackage{fullpage}

\usepackage[margin=1in]{geometry}
\usepackage{amsmath}
\usepackage{amsfonts}
\usepackage{amssymb}
\usepackage{amsthm,mathtools,mathabx,accents,}
\usepackage{newlfont}
\usepackage{graphicx}
\usepackage{xcolor,enumerate,mathrsfs, scalerel}

\newtheorem{thm}{Theorem}[section]
\newtheorem{cor}[thm]{Corollary}
\newtheorem{lem}[thm]{Lemma}
\newtheorem{prop}[thm]{Proposition}

\theoremstyle{definition}
\newtheorem{defn}[thm]{Definition}
\newtheorem{rem}[thm]{Remark}
\numberwithin{equation}{section}

\newcommand{\norm}[1]{\Vert#1\Vert}
\newcommand{\Norm}[1]{\left\Vert#1\right\Vert}
\newcommand{\na}{\nabla}
\newcommand{\pa}{\partial}
\newcommand{\lec}{\lesssim}
\newcommand{\td}{\tilde}

\renewcommand{\div}{\operatorname{div}}

\newcommand\al{\alpha}
\newcommand\be{\beta}
\newcommand\de{\delta}
\newcommand\De{\Delta}
\newcommand{\ga}{\gamma}
\newcommand\Ga{\Gamma}

\newcommand\e {\varepsilon}
\newcommand\ka{\kappa}
\newcommand{\la}{\lambda}
\newcommand{\La}{\Lambda}
\newcommand\ph{\varphi}
\newcommand{\ta}{\tau}
\renewcommand{\th}{\theta}

\newcommand{\T}{\mathbb{T}}
\newcommand{\R}{\mathbb{R}}
\newcommand{\Z}{\mathbb{Z}}
\newcommand{\N}{\mathbb{N}}

\newcommand{\cF}{\mathcal{F}}

\newcommand{\cP}{\mathcal{P}}
\newcommand{\cR}{\mathcal{R}}

\newcommand{\crF}{\mathscr{F}}

\newcommand{\supp}{\operatorname{supp}}
\newcommand{\tr} {\mathop{\mathrm{tr}}}
\newcommand{\I}{\textrm{Id}}
\newcommand{\tri}{\triangle}

\newcommand{\idv}[1]{\mathcal{R}\left( #1\right)}

\renewcommand{\dot}[1]{\accentset{\circ}#1}

\DeclarePairedDelimiter{\ceil}{\lceil}{\rceil}

\usepackage[usestackEOL]{stackengine}
\def\dint{\,\ThisStyle{\ensurestackMath{%
  \stackinset{c}{.2\LMpt}{c}{.5\LMpt}{\SavedStyle-}{\SavedStyle\phantom{\int}}}%
  \setbox0=\hbox{$\SavedStyle\int\,$}\kern-\wd0}\int}

\begin{document}

\title{On Non-uniqueness of 
continuous entropy solutions to the isentropic compressible Euler equations}

\author{Vikram Giri, Hyunju Kwon}
\address{\parbox{\linewidth}{
Vikram Giri\\
Department of Mathematics, Princeton University\\
Fine Hall, Princeton, NJ 08544, USA\\
E-mail address: vgiri@math.princeton.edu\\
\\
Hyunju Kwon \\
School of Mathematics, Institute for Advanced Study\\
1 Einstein Dr., Princeton NJ 08540, USA\\
E-mail address: hkwon@ias.edu
 }
 }

\begin{abstract} 
We consider the Cauchy problem for the isentropic compressible Euler equations in a three-dimensional periodic domain under general pressure laws. For any smooth initial density away from the vacuum, we construct infinitely many entropy solutions with no presence of shock. In particular, the constructed density is smooth and the momentum is $\al$-H\"older continuous for $\al<1/7$. Also, we provide a continuous entropy solution satisfying the entropy inequality strictly.      
\end{abstract}

\maketitle

\section{Introduction} 
We consider the Euler equations for an isentropic compressible fluid on the spatially periodic domain $[0,T] \times \T^3$, $\T^3=[-\pi,\pi]^3$ and $T\in(0,\infty)$,
\begin{equation}\label{eqn.CE0}\begin{dcases}
\pa_t  \varrho + \div (\varrho u) =0 \\
\pa_t (\varrho u) + \div \left(\varrho u\otimes u\right) + \na [p(\varrho)] = 0.
\end{dcases}\end{equation}
The two unknowns are the mass density $\varrho:[0,T] \times \T^3\to [0,\infty)$ of the fluid (or gas) and its velocity $u:[0,T] \times \T^3\to \R^3$. In this paper, we consider mass densities bounded below by a positive constant (i.e. no formation of vacuum).
The equations express the conservation laws of mass and linear momentum, respectively, and are called the continuity equation and the momentum equation. The pressure $p$ is given as a function of the density, $p=p(\varrho)$. A typical example of the pressure is given by the polytropic pressure law $p(\varrho) = \ka \varrho^\ga$ for $\ka>0$ and $\ga>1$, but various pressure laws appear in the study of real gases and complex fluids, see \cite{CoFr,Whitham}. In particular, the equations \eqref{eqn.CE0} are a hyperbolic system when $p'(\varrho)>0$. 

For general hyperbolic system, the unique existence of a smooth solution for short time is well-known \cite{Majda, Kato}. As we see from the one-dimensional example, the Burgers equations, however, a smooth solution develops discontinuity in finite time. In an attempt to continue the solution after the singularity occurs, {\it weak solutions} {in $L^\infty_{t,x}$} have been considered, but they are known to be non-unique. This leads the Lax entropy inequalities as a selection principle and they play a successful role in the Burgers equations. Similarly, {the isentropic Euler equations \eqref{eqn.CE0} have many smooth solutions with finite-time blow-up. To single out physically relevant solutions among bounded weak solutions,} 
{\it the entropy inequality} is imposed as
\begin{align}
    \pa_t \left( \frac{\varrho|u|^2}{2} +  \varrho e( \varrho)\right) +
\div\left( \left(
\frac{\varrho|u|^2}{2} +  \varrho e( \varrho) + p( \varrho)\right)u\right)\leq 0 \label{LEI0}.
\end{align}
Here, $e:\R \to \R$ denotes the specific internal energy, which is related to pressure $p( \varrho)$ through $ \varrho^2 e'( \varrho)=p( \varrho)$. We often use $P(\varrho) := \varrho e(\varrho)$ instead, called pressure potential, which satisfies $\varrho P'(\varrho) = P(\varrho) + p(\varrho)$. Since the equality demonstrates the energy conservation law, \eqref{LEI0} is also called the {\it energy inequality}. Indeed, ${\varrho|u|^2}/{2} +  \varrho e( \varrho)$ represents the total energy density, the sum of the kinetic energy density and the internal energy density. {We call weak solutions $(\varrho, u)$ {in $L^\infty_{t,x}$} which solve \eqref{eqn.CE0} and satisfy \eqref{LEI0} in distribution sense as {\it entropy solutions}.} See \cite{Dafermos, Lax} for further discussion for general hyperbolic systems.

Entropy solutions have a weak-strong uniqueness principle {when $p'(\varrho)>0$;} if there exists a short-time classical solution (strong solution) to \eqref{eqn.CE0} then any entropy solutions (weak solutions) with the same initial data must coincide with it, \cite{Dafermos, Weid2018, GJK20}. However, it turns out that the entropy inequality \eqref{LEI0} cannot serve as a selection principle\footnote{For discussions on different-type of admissible conditions such as mixing entropy or maximal entropy production, see \cite{CG2019, CGL2020} for example.}. After conjectured by Elling \cite{Elling}, De Lellis and Sz\'ekelyhidi \cite{DLeSz2010} established the first non-uniqueness result, finding a bounded initial data which has infinitely many entropy solutions on $(0,\infty)\times \R^d$, $d\geq 2$.
This was extended in the works of Chiodaroli \cite{Ch2014} allowing regular initial density $\varrho_0$. 
In the special case of $d=2$ and $p(\varrho)=\varrho^2$, the non-uniqueness results were obtained 
for some Lipschitz initial data in \cite{ChDLKr2015} and even with smooth initial data in \cite{ChKr2021}. Indeed, the constructed solution coincides with the classical solution for finite time but collapsed into a perturbed Riemann state. Also, in the case of $d=2$ and under the polytropic pressure law, Markfelder and Klingenberg \cite{MaKl2018} showed
the non-uniqueness that encompasses a large variety of shocks and rarefaction waves. For a more comprehensive survey of these results, we refer the reader to \cite{Ma2020}. Recently, in the case of $p(\varrho)= \varrho^\gamma$, $1<\ga< 1+ \frac 2d$, $d=2,3$, Chen, Vasseur and Yu \cite{CVY21} obtained a dense subset of the energy space such that any initial data in the set generates infinitely many weak solutions on $(0,\infty)\times \T^d$ with no increment of the average of the total energy on the torus.

All the previous non-uniqueness results for entropy solutions produce entropy solutions. The paper, on the other hand, provides infinite many solutions even without the presence of discontinuities. The constructed solutions furthermore satisfy the energy equation.
The theorem is stated for \eqref{eqn.CE0} formulated in terms of the density and the linear momentum $m= \varrho u$,
\begin{equation}\label{eqn.CE}
\begin{dcases}
\pa_t  \varrho + \div m =0 \\
\pa_t m + \div \left(\frac{m\otimes m}{ \varrho}\right) + \na p( \varrho) = 0.
\end{dcases}
\end{equation}
and the corresponding energy inequality,
\begin{align}
    \pa_t \left( \frac{|m|^2}{2 \varrho} +  \varrho e( \varrho)\right) +
\div\left(\frac{m}{ \varrho} \left(
\frac{|m|^2}{2 \varrho} +  \varrho e( \varrho) + p( \varrho)\right)\right)\leq 0 \, . \label{LEI}
\end{align}

\begin{thm}\label{thm}
For any $0\le \be< 1/7$, initial density $\varrho_0\in C^\infty(\T^3)$ with $\varrho_0\geq \e_0$ for some positive constant $\e_0$, and { pressure $p\in C^\infty([\e_0,\infty))$}, we can find
infinitely many distinct entropy solutions, $\varrho \in C^\infty([0,T]\times \T^3)$ and $m\in C^\be([0, T]\times \T^3)$, to the isentropic compressible Euler equations \eqref{eqn.CE} emanating from the same initial data and satisfying the energy equation
\begin{align}\label{LEE}
    \pa_t \left( \frac{|m|^2}{2 \varrho} +  \varrho e( \varrho)\right) +
\div\left(\frac{m}{ \varrho} \left(
\frac{|m|^2}{2 \varrho} +  \varrho e( \varrho) + p( \varrho)\right)\right)= 0 \, 
\end{align}
in distribution sense.
\end{thm}

\noindent With a suitable modification of our arguments, one can also produce an analogous example of infinitely many entropy solutions with the same initial data that produce entropy. 

\medskip

Another aspect of the entropy inequality \eqref{LEI} is the energy balance law. Any smooth solutions to \eqref{eqn.CE} conserve the energy locally for any time, satisfying \eqref{LEE}. The conservation is still valid when a weak solution $(\varrho, m)$ satisfies $\varrho \in L^3(0,T; B^{\alpha}_{3,\infty}(\T^3))\cap L^\infty([0,T]\times \T^3)$, $\varrho\geq \e_0>0$, $m\in L^3(0,T; B^{\alpha}_{3,\infty}(\T^3))$ for $\al>1/3$ and $p\in C^2([\e_0, \norm{\varrho}_{L^\infty}])$, see \cite{GMSG2018, FGSGW2017} (we also refer to \cite{DE2018} for the conservation in the Euler system). 
The result is sharp because shock solutions \cite{FGSGW2017}, yet even in the absence of shock, an entropy solution can dissipate the total energy.

\begin{thm}\label{thm.onsager}
For any $0\leq \beta < 1/7$, $ \varrho_0\in C^\infty(\T^3)$ with $ \varrho_0\geq \e_0$ for some positive constant $\e_0$, and { $p\in C^\infty([\e_0,\infty))$}, there is an entropy solution $(\varrho, m) \in C^\infty([0,T]\times \T^3)\times C^\be([0, T]\times \T^3)$ to the isentropic compressible Euler equations \eqref{eqn.CE} such that $(\varrho, m)$ satisfies the entropy inequality \eqref{LEI} strictly in distribution sense. 
\end{thm}
\begin{rem} The constructed solution satisfies the total energy/entropy dissipation,
\begin{equation}\label{eqn.Onsager}
    \int_{\T^3} \frac{|m(t,x)|^2}{2 \varrho(t,x)} +  \varrho(t,x) e( \varrho(t,x)) \, dx < \int_{\T^3} \frac{|m_0|^2}{2 \varrho_0} + \varrho_0 e(\varrho_0) \, dx\, ,
\end{equation}
where $(\varrho_0, m_0)$ denotes the initial data of the solution $(\varrho,m)$.
\end{rem}

The proof is relying on the convex integration scheme starting from De Lellis and Sz\'ekelyhidi \cite{DLSz2012,DLSz2013}, used with great success in proving the longstanding open Onsager Conjecture \cite{On1949} for {\it incompressible} Euler equations. Indeed, the full conjecture is established by Isett \cite{Is2016} (see also \cite{BDLSV2020}), after a series of developments \cite{DLSz2013,DlSzJEMS,Is2013, Bu2014,BuDLeSz2013,Bu2015,BuDLeIsSz2015,BuDLeSz2016,DaSz2016,IsOh2016}. In the effort of a rigorous mathematical validation of the classical Kolmogorov's theory of turbulence, a stronger version of Onsager conjecture has been introduced, namely for $\al<1/3$ the existence of $\alpha$-H\"older continuous weak solutions of the incompressible Euler equations satisfying the local energy inequality strictly,
\begin{equation}\label{eqn.ILEE}
\pa_t \frac{|u|^2}2 + \div{\left(u\left(\frac{|u|^2}2 + p\right)\right)} < 0.
\end{equation}
Building upon \cite{DLeSz2010, Is17}, De Lellis and the second author in \cite{DLK20} obtain the result up to the threshold 1/7, which is the same as the threshold exponent $1/7$ in Theorem \ref{thm.onsager}. Indeed,  we adapt the convex integration used in \cite{DLK20} to the compressible case, using the structural similarity between the entropy inequality \eqref{LEI0} and the local energy inequality for the incompressible Euler equations.

\section{Outline of the proof}

We construct {entropy solutions approximations with sequences of {\it dissipative Euler-Reynolds flows}. }

\begin{defn} For a given $ \varrho\in C^\infty({[T_1,T_2]}\times \T^3)$ with $ \varrho\geq \e_0$ for some positive constant $\e_0$, a tuple of smooth tensors $(m, c, R,\ph)$ is a {\it dissipative Euler-Reynolds flow} with {\it global energy loss} $E=E(t,x)$ if it solves the following system
\begin{equation}\begin{split}\label{app.eq}
&\pa_t  \varrho + \div m =0\\
&\pa_t m + \div \left(\frac{m\otimes m}{ \varrho}\right) + \na p( \varrho) = \div( \varrho (R-c\I) )\\
&\pa_t \left( \frac{|m|^2}{2 \varrho} + P( \varrho)\right) +
\div\left(\frac{m}{ \varrho} \left(
\frac{|m|^2}{2 \varrho} +  \varrho P'( \varrho)\right)\right)\\
&\hspace{4cm}=  \varrho \left(\pa_t + \frac{m}{ \varrho}\cdot\na\right) \frac12\tr(R) + \div((R-c\I) m) + \div ( \varrho\ph) + \pa_t E
\end{split}\end{equation}
{in distribution sense,}
where $(\div S)_i = \pa_{j} S_{ij}$ and $\I$ is the identity matrix. To be consistent with the term dissipative, we assume that $\pa_t E\leq 0$. 
\end{defn}

{Compared with a dissipative Euler-Reynolds flow for the incompressible Euler equations introduced in \cite{Is17, DLK20}, the counterpart for the compressible Euler equations has the constant $c$. When the fluid is incompressible ($\varrho={\text{const.}}$), the terms involved with $c$ vanish, while in the compressible case, the constant $c$ plays an important role in cancelling the Reynolds stress from the previous step and estimating the new unsolved current error in the convex integration scheme. }

\subsection{Induction scheme} 
At each $q$th step, we construct a dissipative Euler-Reynold flow $(m_q, c_q, R_q, \ph_q)$ with some fixed global energy loss $E=E(t)$, where $c_q$ is prescribed and $(c_q, R_q,\ph_q)$ converges to $0$ in $\R\times C^0([0,T]\times \T^3)\times C^0([0,T]\times \T^3)$ as $q$ goes to infinity. Then, we see that limit solution will solve \eqref{eqn.CE}, \eqref{LEI} in distribution sense. 

More precisely, for $q\in \N\cup\{0\}$ we introduce the frequency $\la_q$ and the amplitude $\de_q^\frac 12$ of the momentum $m_q$, which have the form
\begin{align*}
\la_q = \ceil{\la_0 ^{(b^q)}}, \quad \de_q = \la_q^{-2\al}\, ,
\end{align*}
where $\alpha$ is a positive parameter smaller than $1$ and $b$ and $\lambda_0$ are real parameters larger than $1$ (however, while $b$ will be typically chosen close to $1$, $\lambda_0$ will be typically chosen very large). In particular, $\lambda_q\delta_q^{\frac{1}{2}}$ is a monotone increasing sequence. We also set $c_q$ as 
\[
c_q = \sum_{j=q+1}^\infty \de_j,
\]
At $(q+1)$th step, we then find a correction $n_{q+1} := m_{q+1} - m_q$ to make the error $(R_q, \ph_q)$ get smaller (in $C^0$ space) as $q$ goes to infinity. 

 In the induction hypothesis, we will assume several estimates on the tuple $(m_q, R_q, \ph_q)$. For technical reasons, the domains of definition of the tuples is changing at each step and it is given by $[-\tau_{q-1}, T+\tau_{q-1}]\times \mathbb T^3$, where $\tau_{-1} = (\la_0\de_0^\frac12)^{-1}$ and for $q\geq 0$ the parameter $\tau_q$ is defined by 
\[
\tau_q = \left(C_\varrho M \la_q^\frac12\la_{q+1}^\frac12 \de_q^\frac14\de_{q+1}^\frac14 \right)^{-1}
\] 
for some constant $C_\varrho$ depending only {on $\varrho$} and a constant $M=M(\varrho, p)$ depending only {on $\varrho$} and $p$; they will be specified later in \eqref{mu.tau} and Proposition \ref{ind.hyp}, respectively. Note the important fact that $\tau_q$ is decreasing in $q$ for the choice of sufficiently large $\la_0$. In order to shorten our formulas, it is convenient to introduce the following notation:
\begin{itemize}
\item $\cal{I} + \sigma$ is the concentric enlarged interval $(a-\sigma, b+\sigma)$ when $\cal{I} = [a,b]$;
\item  $\norm{F_q}_N$ is the $C^0_t C^N_x$ norm of $F_q$ on its domain of definition, namely 
\[
\norm{F_q}_N := \norm{F_q}_{C^0([0, T]+ \tau_{q-1};C^N(\T^3))}\, .
\]
When $F$ is a function of time or space only, we abuse notation and write $\norm{F}_0 =\norm{F}_{C^0(0,T)}$ for a time function $F$ or $\norm{F}_0 =\norm{F}_{C^0(\T^3)}$ for a space function $F$.
\end{itemize}

We are now ready to detail the inductive estimates:
\begin{align}
\norm{m_q}_0 \leq  \underline{M}-\de_q^\frac12, \quad \norm{m_q}_N &\leq M\la_q^N \de_q^\frac 12, \quad N=1,2 \label{est.vp}
\end{align}
and
\begin{align}
\norm{R_q}_N &\leq \la_q^{N-3\ga}\de_{q+1}, \quad \norm{D_{t,q} R_q}_{N-1} \leq \la_q^{N-3\ga} \de_q^\frac 12 \de_{q+1}, \qquad N = 0,1,2\label{est.R}\\
\norm{\ph_q}_N &\leq \la_q^{N-3\ga} \de_{q+1}^\frac 32, \quad
\norm{D_{t,q} \ph_q}_{N-1} \leq \la_q^{N-3\ga}\de_q^\frac12 \de_{q+1}^\frac 32, \qquad N = 0,1,2
 \label{est.ph}
\end{align}
where $D_{t,q} = \pa_t + \frac{m_q}{\varrho}\cdot \na $ and $\ga = (b-1)^2$ and constant $\underline{M}=\underline{M}(\varrho, p)\geq 1$ will be determined in Section \ref{sec 2.2}.

\begin{rem}\label{r:estimates_material_derivative}
Note that in a writing like \eqref{est.R} and \eqref{est.ph}, for $N=0$ we are {\em not claiming any negative Sobolev estimate} on $D_{t,q} F$: the reader should just consider the advective derivative estimate to be an empty statement when $N=0$. The reason for this convention is just to make the notation easier, as we do not have to state in a separate line the estimate for $\|F\|_0$ in many future statements. 
\end{rem} 

 Notice that \eqref{app.eq} is invariant under addition to $E$ of a constant and we adopt the normalization condition $E(0)=0$. We will therefore assume
\begin{equation}\label{e:assumption_on_E}
E(0) =0\, , \quad E' \leq 0.
\end{equation}
Under this setting, the core inductive proposition is given as follows. 
\begin{prop}[Inductive proposition]\label{ind.hyp} Let $\al\in (0,1/7)$ and let $\varrho\in C^\infty({[-\tau_{-1},T+\tau_{-1}]}\times \T^3)$ be a function with $\varrho\geq \e_0$ for some positive constant $\e_0$ and $\frac d{dt}\int_{\T^3} \varrho(t,x) \,d x =  0$ for all $t$. 
There exists a constant $M={M(\varrho)}{>1}$, functions $\bar{b} (\alpha) >1$ and $\Lambda_0=\Lambda_0 (\alpha, b,{M}, \varrho, p) >0$ such that the following property holds. Let $b \in (1, \bar b (\alpha))$ and $\la_0\geq \La_0 $ and let $c_q = \sum_{j=q+1}^\infty \de_j$ for any $q\in \N$. Assume that a tuple of tensors $(m_q,c_q,R_q,\ph_q)$ is a dissipative Euler-Reynolds flow defined on the time interval $[0,T]+ \tau_{q-1}$ satisfying \eqref{est.vp}-\eqref{est.ph} for an energy loss $E$ satisfying \eqref{e:assumption_on_E}. Then, we can find a corrected dissipative Euler-Reynolds flow $(m_{q+1},  c_{q+1},  R_{q+1}, \ph_{q+1})$ with prescribed $c_{q+1}$ on the time interval $[0,T]+\tau_q$ for the same energy loss $E$ which satisfies \eqref{est.vp}-\eqref{est.ph} for $q+1$ and 
\begin{align}\label{cauchy}
\norm{m_{q+1}-m_q}_{C^0([0,T]\times \T^3)}+ \frac 1{\la_{q+1}}\norm{m_{q+1}-m_q}_{C^0([0,T];C^1(\T^3))} \leq  M \de_{q+1}^\frac 12.  
\end{align}
\end{prop}

While the latter proposition would be enough to prove Theorem \ref{thm.onsager}, we will indeed need a technical refinement in order to show Theorem \ref{thm}. In its statement we use the following convention:
\begin{itemize}
\item Given a function $f$ on $[0, T]\times \T^3$, $\supp_t (f)$ denotes its temporal support, 
\[
\supp_t (f) := \{t: \exists \ x \;\mbox{with}\; { f (t,x)} \neq 0\}\, .
\] 
\item Given an interval $\cal{I} = [a,b]$, $|\cal{I}|$ means its length $(b-a)$, and $\cal{I} + \sigma$ denotes the concentric enlarged interval $(a-\sigma, b+\sigma)$. 
\end{itemize}

{
\begin{rem} The assumption
$\frac d{dt}\int_{\T^3} \varrho(t,x) \,d x =  0$ for all $t$ is a compatibility condition of the continuity equation.
\end{rem}
\begin{rem} \label{rem:dep}
The constant $M$ depends on an upper bound of $\norm{\varrho}_{C([0,T]+\tau_{-1}; C^2(\T^3))}$ and $\e_0$ more precisely. Also, the dependence of $\La_0$ on $\varrho$ and $p$ means the dependence on $\e_0$ and upper bounds of $\norm{\varrho}_{C([0,T]+\tau_{-1}; C^{n_0+1}(\T^3))}$, $\norm{\pa_t\varrho}_{C([0,T]+\tau_{-1}; C^{n_0+1}(\T^3))}$ and $\norm{p}_{C^{n_0+1}(\T^3)}$, where $n_0 = \ceil{\frac{2b(2+\al)}{(b-1)(1-\al)}}$. The norm on $p$ can be weakened, but we don't pursue that direction in this paper. 
\end{rem}
}

\begin{prop}[Bifurcating inductive proposition]\label{p:ind_technical}
{Let $\varrho\in C^\infty({[-\tau_{-1},T+\tau_{-1}]}\times \T^3)$ be a function with $\varrho\geq \e_0$ for some positive constant $\e_0$ and $\frac d{dt}\int_{\T^3} \varrho(t,x) \,d x =  0$ for all $t$. }
Let a constant $M$, the functions  $\bar b$ and $\Lambda_0$, the parameters $\alpha$, $b$, $\lambda_0$ and the tuple $(m_q,c_q,R_q,\ph_q)$ be as in the statement of Proposition \ref{ind.hyp}. For any time interval $\cal{I}\subset (0, T)$ with $|\cal{I}|\geq 3\tau_q$
we can produce a first tuple $({m}_{q+1}, c_{q+1},  R_{q+1},  \ph_{q+1})$ and a second one $(\td m_{q+1},  c_{q+1}, \td R_{q+1},  \td \ph_{q+1})$ which share the same initial data, satisfy the same conclusions of Proposition \ref{ind.hyp} and additionally
\begin{align}\label{e:distance_and_support}
\norm{m_{q+1}-\td m_{q+1}}_{C^0([0, T];L^2(\T^3)) }\geq {{\e_0}} \de_{q+1}^\frac 12, \quad
\supp_t(m_{q+1}-\td m_{q+1}) \subset \cal{I}. 
\end{align}
Furthermore, if we are given two tuples $(m_q,c_q,R_q,\ph_q)$ and $(\td m_q, c_q, \td R_q,  \td\ph_q)$ satisfying \eqref{est.vp}-\eqref{est.ph} and
\[
\supp_t(m_q-\td m_q, R_q-\td R_q,  \ph_q-\td \ph_q)\subset \cal{J} 
\]
for some interval $\cal{J}\subset (0, T)$, we can exhibit corrected counterparts $({m}_{q+1},  c_{q+1},  R_{q+1},  \ph_{q+1})$ and $(\td m_{q+1},  c_{q+1}, \td R_{q+1},  \td\ph_{q+1})$ again satisfying the same conclusions of Proposition \ref{ind.hyp} together with the following control on the support of their difference:
\begin{equation}\label{e:second_support}
\supp_t(m_{q+1}-\td m_{q+1}, R_{q+1}-\td R_{q+1},  \ph_{q+1}-\td \ph_{q+1})\subset \cal{J} + (\la_q\de_q^\frac12)^{-1}.
\end{equation}
\end{prop}

\subsection{Construction of a starting tuple with stationary density}\label{sec 2.2}
{For the simplicity, we choose stationary density $\varrho(t,\cdot) = \varrho_0$ for all $t\in\R$. In order to construct a starting momentum, unlike the incompressible case, we need to cancel out the prescribed pressure. 
To this end, we use {\it Mikado} flows as building-blocks,} which are stationary solutions of the {incompressible} Euler equation first introduced in \cite{DaSz2016}. In order to define them, consider a function $\vartheta$ on $\mathbb R^2$ and let $\bar U (x) = e_3 \vartheta (x_1, x_2)$, where $e_3 = (0,0,1)$. Then, we apply a stretching factor $s>0$, a general rotation $O$, and a translation by a vector $\bar x$ to define 
\[
U (x) = s O \bar U (O^{-1} (x-\bar x))\, .
\] 
Observe that the periodization of this function solves the stationary incompressible Euler equations on $\mathbb T^3$ when $O e_3$ belongs to $a \mathbb Q^3$ for some $a>0$. From now on, with a slight abuse of our terminology, a Mikado flow always refers to such periodization. Moreover the vector $f = s O e_3$ will be, without loss of generality assumed to belong to $\mathbb Z^3$ and will be called the {\em direction} of the Mikado flow, while $\bar x$ will be called  its {\em shift}. For each $f$ we will specify an appropriate choice of $\vartheta$, which will be smooth and compactly supported in a disk $B (0, r_0)$ for some small $r_0>0$. Also, $\vartheta$ will not depend on the shift $\bar x$ and we will denote by $U_f$ the corresponding Mikado flows when $\bar x=0$. The Mikado flows then can be written as $U_f = f \psi_f$ for some smooth $\psi_f\in C^\infty_c (\mathbb R^3)$ with $f\cdot \nabla \psi_f =0$. We recall the following elementary lemma, used since the pioneering work \cite{DaSz2016}.
\begin{lem}\label{l:Mikado}
Let $\mathcal{F}$ be a set of vectors in $\Z^3$ with finite cardinality. For each $f\in \mathcal{F}$, let $\bar x(f) \in \mathbb R^3$, $\gamma_f \in \mathbb R$ and $\lambda \in \mathbb N$. If the supports of the maps $U_f (\cdot - \bar x (f))$ are pairwise disjoint, then
\[
\sum_{f\in \mathcal{F}} \gamma_f U_f (\lambda (x-\bar x(f)))
\]
is a stationary solution of the incompressible Euler equations on $\mathbb T^3$.
\end{lem}
\noindent Note that the supports of the functions $U_f (\cdot - \bar x)$ and $\psi_f (\cdot - \bar x)$ are contained in a $r_0$-neighborhood of 
\begin{equation}
l_f+\bar x := \left\{x\in \mathbb T^3: \left(x -\sigma f - \bar x\right) \in 2\pi \mathbb Z^3 \quad \mbox{for some $\sigma \in \mathbb R$}\right\}\, .
\end{equation} 
If $r_0$ is sufficiently small, depending on $f$, the latter is a ``thin tube'' winding around the torus a finite number of time.

Another ingredient for defining the starting errors $R_0$ and $ \ph_0$ is the inverse divergence operator introduced in \cite{DLSz2013}.

\begin{defn}[Inverse divergence operator]\label{idv.defn}
For any $f\in C^\infty(\T^3;\R^3)$, the inverse divergence operator is defined by
\[
(\cal{R}f)_{ij}= \cal{R}_{ijk} f_k 
= -\frac 12 \De^{-2} \pa_{ijk} f_k + \frac 12\De^{-1} \pa_k f_k \de_{ij} - \De^{-1}\pa_i f_j - \De^{-1}\pa_jf_i. 
\]
\end{defn}
\begin{rem} The image of the divergence free operator $\cal{R}f(x)$ is designed to be a trace-free symmetric matrix at each point $x$ and to solve 
\[
\div(\cal{R}f) = f - \langle f \rangle\, .
\]
\end{rem}
\noindent To define $\ph_0$, we abuse the notation and define an inverse divergence operator on $C^\infty(\T^3;\R)$,
\[
(\cR g)_i = \De^{-1} \pa_i g. 
\]
Indeed, it maps a smooth scalar function to a vector-valued function, and solves $\div \cR g =g - \langle g \rangle$. 

\bigskip

We are now ready to find a starting tuple. Set $\varrho(t,x)= \varrho_0(x)$ for all $t\in \R$, so that $\pa_t \varrho=0$ and $\div m_0 =0$. {By the choice of $\varrho$, one can fix $M$, $\al$, $b$, and $\la_0$ by Proposition \ref{ind.hyp} and Proposition \ref{p:ind_technical}}. We also set the preponderant part $\overline m_0$ of $m_0$ as
\begin{align*}
    \overline m_0 = \varrho^\frac12(\overline{C}_{\varrho, p} - p(\varrho) - c_0 \varrho)^\frac12 \sum_{i=1}^3 \psi_i(\overline \la x) e_i, \end{align*}
where $\overline{C}_{\varrho, p}$ is a positive constant defined by $\overline{C}_{\varrho, p} = 2(\norm{p(\varrho)}_0 + c_0\norm{\varrho}_0)$, and $e_i$, $i=1,2,3$, are standard unit vectors whose $i$th component is $1$. Also, $\psi_i = \psi_{e_i}(\cdot - \bar{x}_i)$ is associated functions to each Mikado direction $e_i$, which is compactly supported and satisfies $e_i \cdot \na \psi_i =0$ and
\[
\int_{\T^3} \psi_i d x = \int_{\T^3} \psi_i^3 d x  =0, \quad \dint_{\T^3} \psi_i^2 d x =1.
\]
Suitable shifts $\bar{x}_i$ are chosen to have pairwise disjoint $\supp(\psi_i)$. We remark that $\pa_t \overline{m}_0 =0$ because $\varrho$ is time-independent. Since we have
\begin{align*}
    \frac{\overline m_0 \otimes \overline m_0}{\varrho}
    =( \overline{C}_{\varrho, p} - p(\varrho) -c_0\varrho )  \I
    + ( \overline{C}_{\varrho, p} - p(\varrho) -c_0\varrho  ) \sum_{i=1}^3 (\psi_i^2(\overline \la x) -1) e_i\otimes e_i,
\end{align*}
its divergence satisfies
\begin{align}\label{div.bar.m0}
    \div \left(\frac{\overline m_0 \otimes \overline m_0}{\varrho} \right) + \na (p(\varrho) +c_0\varrho )
    =  -\sum_{i=1}^3 \pa_i( p(\varrho) +c_0\varrho  )(\psi_i^2(\overline \la x) -1) e_i.
\end{align}
The last equation follows from $\div ((\psi_i^2(\overline \la x) -1) e_i\otimes e_i) = 0$ because of $e_i\cdot \na \psi_i =0$. Since $\psi_i$ is a smooth periodic function on $\T^3$ with zero-mean, one can represent it as its Fourier series $\psi_i = \sum_{k\in \Z\setminus\{0\}} \bar{b}_{i,k}  e^{i k\cdot x}$. Note that the divergence-free condition of $\psi_i e_i$ implies $\bar{b}_{i,k}(e_i \cdot k)=0$ for all non-zero $k\in \Z^3$. Using this condition, one can write $\overline m_0$ as
\begin{align*}
    \overline m_0 
    = \sum_{j=1}^3 \sum_{k\in \Z\setminus\{0\}} \varrho^\frac12 (C_{\varrho, p} -p(\varrho) - c_0\varrho)^\frac12 \operatorname{curl}\left( \frac{i \bar{b}_{j,k} k\times e_j}{\overline \la |k|^2} e^{i\overline \la k\cdot x} \right).
\end{align*}
Then, adding a correction
\begin{align*}
    \nu = \sum_{j=1}^3 \sum_{k\in \Z\setminus\{0\}} \na (\varrho^\frac12(C_{\varrho, p} -p(\varrho) - c_0\varrho)^\frac12) \times \left( \frac{i \bar{b}_{j,k} k\times e_j}{\overline \la |k|^2} e^{i\overline \la k\cdot x} \right),
\end{align*}
we get a divergence-free, {mean-zero}, approximate momentum $m_0 = \overline m_0 + \nu$.

Now, we choose a Reynold stress error $R_0$. To make it solve the relaxed momentum equation in \eqref{app.eq}, 
\begin{equation}\begin{split} \label{eq.R0}
    \div (\varrho R_0) 
    &=  \div \left(\frac{ m_0 \otimes  m_0}{\varrho} \right) + \na (p(\varrho) +c_0\varrho )\\
& = -\sum_{i=1}^3 \pa_i( p(\varrho) +c_0\varrho  )(\psi_i^2(\overline \la x) -1) e_i + 
 \div \left(\frac{1}{\varrho} (\overline m_0 \otimes \nu  + \nu \otimes \overline m_0 +\nu \otimes \nu)  \right),
\end{split}\end{equation}
we set
\begin{align*}
    \varrho R_0 &:= \mathcal{R} \left(- \sum_{i=1}^3 \pa_i(  p(\varrho) +c_0\varrho  )(\psi_i^2(\overline \la x) -1) e_i  \right) + \frac{\overline m_0 \otimes \nu  + \nu \otimes \overline m_0 +\nu \otimes \nu}{\varrho} - \frac 23 E(t) \I\\
    &=: \varrho \overline R_0 - \frac 23 E(t) \I.
\end{align*}
We remark that the argument of the inverse divergence operator is mean-zero (see \eqref{eq.R0}), so that $\div(\varrho \overline R_0)$ is the argument itself. 
Also, $E$ is a global energy loss depending only on time, which will be specified later. Note that $\overline R_0$ is independent of time, so that $\pa_t(\varrho \ka_0)= -E'$, where $\ka_0 = \tr(R_0)/2$. 

Lastly, we choose an unsolved current error $\ph_0$. Since both $m_0$ and $\varrho$ are time independent, it is enough to find a solution $\ph_0$ to 
\begin{align*}
    \div (\varrho \ph_0) 
    = \div\left( \frac{m_0}{\varrho} \left( \frac{|m_0|^2}{2\varrho} + \varrho P'(\varrho) 
    \right) \right)
    - \pa_t(\varrho \ka_0) - \div (\ka_0 m_0) - \div (R_0m_0) - E',
\end{align*}
where we used $\varrho \left( \pa_t + \frac{m_0}{\varrho} \cdot \na\right) \ka_0 = \pa_t(\varrho \ka_0) + \div(\ka_0 m_0)$ and $\div (c_0 m_0) = 0$. Since we have $\pa_t (\varrho \ka_0 ) = -E'$, we set
\begin{align*}
    \varrho \ph_0
    &: = \mathcal{R} \left( \sum_{i=1}^3 \pa_i\left( \frac{(C_{\varrho, p}- p(\varrho) -c_0\varrho)^\frac32}{2\varrho^\frac12}\right) \psi^3_i(\overline \la x)e_i\right) + \left(\frac{|m_0|^2m_0}{2\varrho}-\frac{|\overline m_0|^2\overline m_0}{2\varrho}  \right) \\
    &\quad + \mathcal{R}(m_0 \cdot \na P'(\varrho)) - \ka_0 m_0 - R_0 m_0.
\end{align*}
Here we used $e_i \cdot\na \psi^3_i =0$ and hence
\begin{align*}
    \div\left( \frac{\overline m_0}{\varrho} \left( \frac{|\overline m_0|^2}{2\varrho}
    \right) \right)
&= \div  \left(
\frac{(C_{\varrho, p} - p(\varrho) - c_0 \varrho))^\frac32}{2\varrho^\frac12} \sum_{i=1}^3 \psi_i^3(\overline \la x) e_i
\right)\\
& = \sum_{i=1}^3 \pa_i \left( \frac{(C_{\varrho, p}- p(\varrho) -c_0\varrho)^\frac32}{2\varrho^\frac12}\right) \psi^3_i(\overline \la x)e_i.
\end{align*}

Then, $(m_0,c_0, R_0, \ph_0)$ solves \eqref{app.eq}, and can be estimated as follows;
\begin{align*}
&\norm{m_0}_0 \leq \norm{\overline m_0}_0 
+\norm{\nu}_{0}
\leq  \frac 14\underline{M}(\varrho, p ) \left(1+ \overline{\la}^{-1}\right)  \leq\underline M - \de_0^\frac12
\\
&\norm{m_0}_{N} \leq \norm{\overline m_0}_N 
+\norm{\nu}_{N}
\lec_{\varrho, p,} \overline \la^N 
\qquad \forall N = 1,2
\end{align*}
for some {$\underline M(\varrho, p)\geq 2$}, and\footnote{Here and in the rest of the note, given two quantities $A_q$ and $B_q$ depending on the induction parameter $q$ we will use the notation $A\lec B$ meaning that $A\leq CB$ for some constant $C$ which is independent of $q$. In some situations we will need to be more specific and then we will explicitly the dependence of $C$ on the various parameters involved in our arguments.}
\begin{align*}
    &\norm{R_0}_N
    \lec_{\varrho, p} \overline{\la}^{N-1} + \norm{E}_0, 
    \qquad
    \norm{\ph_0}_N
    \lec_{\varrho, p} \overline \la^{N-1} + \norm{E}_0 \overline{\la}^N 
    \\
    &\norm{D_{t,q} R_0}_{N-1}
    \leq \norm{\pa_t R_0}_{N-1} + \Norm{\frac{m_0}{\varrho} \cdot \na R_0}_{N-1}
    \lec_{\varrho, p} \norm{E'}_0 + \overline \la^{N-1}(1+\norm{E}_0)
    \\
    &\norm{D_{t,q} \ph_0}_{N-1}
    \leq \norm{\pa_t \ph_0}_{N-1} + \Norm{\frac{m_0}{\varrho }\cdot \na \ph_0}_{N-1}
    \lec_{\varrho, p} {\norm{E'}_0 \overline{\la}^{N-1}} + \norm{E}_0 \overline{\la}^{N} + \overline{\la}^{N-1} 
\end{align*}
for any $N=0,1,2$ (the cases $\norm{\cdot}_{-1}$ are empty statements). 
Let $C(p, \varrho)\geq 1$ be the maximum of all implicit constants in the above inequalities. For $\bar{b}(\al)$ sufficiently close to $1$ and  sufficiently large $\la_0\geq \La_0(\al, b, M(\varrho), \varrho, p)$,  we can always find a positive integer $\overline \la$ to satisfy
\begin{equation}\label{choice.bar.lam}
2C(\varrho, p) \la_0^{3\ga} \de_1^{-\frac32}\leq \overline \la \leq (2C(\varrho, p))^{-1}\la_0\de_0^\frac12.
\end{equation}

{For the choice of such $\overline \la$ and $E=0$, the constructed initial approximate solution $(m_0, R_0, \ph_0)$ satisfies \eqref{est.vp}-\eqref{est.ph}, and hence serves as a starting tuple for Theorem \ref{thm}.
} On the other hand, one can find a non-trivial $E$  satisfying \eqref{e:assumption_on_E} (moreover, $E'<0$) and
\begin{equation}\label{est.E}
4C(\varrho, p)\norm{E}_0\leq \la_0^{-3\ga}\de_1^\frac32,\quad \mbox{and}\quad 
4C(\varrho, p)\norm{E'}_0 \leq \la_0^{1-3\ga}\de_0^\frac12 {\de_1^\frac32}. 
\end{equation}
For example,
\begin{equation}\label{def.E}
\overline{E}(t) = -\frac{\la_0^{-3\ga} \de_1^\frac32}{8C(\varrho, p)} (1-\exp(-\la_0\de_0^\frac12t)).    
\end{equation}
For such choice of $\overline\la$ and $\overline{E}$, the constructed initial approximate solution $(m_0, R_0, \ph_0)$ again satisfies \eqref{est.vp}-\eqref{est.ph}, {which works as a starting tuple for Theorem \ref{thm.onsager}.}

\subsection{Construction of a starting tuple with time-dependent density}\label{sec:time-dependent.density} 
The construct density for Theorem \ref{thm}-\ref{thm.onsager} can be {\it time-dependent}. In this subsection, we provide a new starting tuple $(\td m_0, \td R_0, \td \ph_0)$ with a time-dependent density $\varrho$ defined by perturbing the stationary density $\varrho_0$. We first fix the numbers $\al$ and $b$ and define the functions $M$ and $\La_0$ by Proposition \ref{ind.hyp} and Proposition \ref{p:ind_technical}, which gives a fixed number $n_0$ (see Remark \ref{rem:dep}). We then write a time-dependent density as
\begin{equation}\label{def.varrho2}
    \varrho(t,x) := \varrho_0(x) + \epsilon \hat \varrho(t,x),
\end{equation}
where $\hat \varrho(t,x)$ can be any smooth function such that $\int_{\T^3}\pa_t\hat\varrho dx=0$ on $[-1,T+1]$ and 
\begin{equation*}
    \norm{\hat \varrho}_{C^0([-1,T+1];C^N(\T^3))}\leq \norm{\varrho_0}_{C^N(\T^3)}, \quad \norm{\pa_t \hat \varrho}_{C^0([-1,T+1];C^N(\T^3))} \leq \norm{\varrho_0}_{C^N(\T^3)}\qquad \forall N\in [0,n_0+1].
\end{equation*}
We consider $\epsilon\in (0,1/2)$ to have $|\varrho|\geq \e_0/2$.
Then, $M(\varrho)$ and $\Lambda_0 (\alpha, b,{M(\varrho)}, \varrho, p)$ are independent of $\epsilon$ and $\hat\varrho$ (see Remark \ref{rem:dep})).    

To define $(\td m_0, \td R_0, \td \ph_0)$, we let $(m_0,R_0,\ph_0)$ be as in the previous subsection \ref{sec 2.2}; if needed, we adjust $\la_0$ to satisfy $\la_0\geq \Lambda_0 (\alpha, b,{M(\varrho)}, \varrho, p)$ and then $\overline \la$ to satisfy \eqref{choice.bar.lam} with the replacement of $C(\varrho,p)$ and $\la_0$ . We set $\td m_0$ first as 
\begin{equation*}
    \td m_0(t,x) := m_0(x) + \hat m_0(t,x)
\end{equation*}
where
\begin{equation*}
    \hat m_0 := - \mathcal{R} \pa_t \varrho = - \epsilon\mathcal{R} \pa_t \hat\varrho.
\end{equation*}
Since the average of $\pa_t \varrho$ on $\T^3$ is zero and $\div m_0 =0$, it solves
$\pa_t \varrho + \div \td m_0 =0$. Also, for sufficiently small $\epsilon$, we have 
\begin{align*}
    \norm{\td m_0}_0 &\leq \norm{m_0}_0 + \norm{\hat m_0}_0 \leq \frac 14\underline{M}( \varrho_0, p) \left(1+ 1/\overline{\la}\right) +
    \epsilon \norm{\pa_t \varrho}_0  \leq\underline M(\varrho, p ) - \de_0^\frac12 \\
    \norm{\td m_0}_N &\leq \norm{m_0}_N + \norm{\hat m_0}_N \lesssim_{\varrho, p} \overline{\la}^N + \epsilon \norm{\pa_t \varrho}_N \lesssim \overline{\la}^N, \hspace{1.5cm} \forall N = 1,2
\end{align*}
for some $\underline M(\varrho, p)\geq 2$. Then, we set the initial Reynolds stress as
\begin{align*}
    \varrho \td R_0 = \varrho_0 R_0 + c_0(\varrho - \varrho_0)\I + \mathcal{R} \left[\pa_t \td m_0 + \div \left(\frac{\td m_0 \otimes \td m_0}{\varrho} - \frac{m_0 \otimes m_0}{\varrho_0}\right) + \na(p(\varrho)-p( \varrho_0))\right] + \zeta\I.
\end{align*}
Since we will choose $\zeta=\zeta(t)$, $(\td m_0, c_0, \td R_0)$ solves the relaxed momentum equation in \eqref{app.eq} with $\varrho$ defined in \eqref{def.varrho2}. The purpose of $\zeta$, on the other hand, is as follows. The relaxed energy equation in \eqref{app.eq} now can be seen as an equation for $\td \ph_0$. Since $\varrho (\pa_t + (\td m_0/\varrho)\cdot \na) \td\kappa_0 = \pa_t(\varrho \td\kappa_0) + \div{(\td m_0 \td\kappa_0)}$, where $\td\kappa_0 = \tr{\td R_0}/2$, we need the following compatibility condition, 
\begin{align}\label{compa}
    \int_{\T^3} \pa_t \left(\frac{|\td m_0|^2}{2\varrho} + P(\varrho)\right) - \pa_t (\varrho \td\kappa_0) - \pa_t E \, dx = 0.
\end{align}
Recalling that \eqref{compa} holds for $(\varrho_0, m_0, R_0)$, we choose $\zeta$ as 
\begin{align*}
    \dint_{\T^3} \frac{|\td m_0|^2}{2\varrho} - \frac{|m_0|^2}{2\varrho_0} + P(\varrho) - P(\varrho_0) - \frac32 c_0 (\varrho - \varrho_0) \, dx =: 
    \frac32 \zeta(t).
\end{align*}
to have \eqref{compa} for $(\varrho, \td m_0, \td R_0)$. Then, we choose $\td\ph_0$ as
\begin{align*}
    \varrho \td\ph_0 &:= \varrho_0 \ph_0 + \mathcal{R}\pa_t \left(\frac{|\td m_0|^2}{2\varrho} - \frac{|m_0|^2}{2\varrho_0} + (P(\varrho) - P( \varrho_0)) - (\varrho \td\kappa_0 -\varrho_0 \kappa_0) \right)\\
    &\quad+ \left(\frac{\td m_0}{\varrho} \left(
    \frac{|\td m_0|^2}{2 \varrho} +  \varrho P'( \varrho)\right) - \frac{ m_0}{\varrho_0} \left(
    \frac{| m_0|^2}{2 \varrho_0} +  \varrho_0 P'( \varrho_0)\right)\right)\\
    &\quad- (\td\kappa_0\td m_0 - \kappa_0 m_0 ) - (\td R_0\td m_0 - R_0 m_0  ) + c_0 (\td m_0 - m_0) 
\end{align*}
so that $(\td m_0, c_0, \td R_0, \td \ph_0)$ satisfies the relaxed energy equation and hence solves \eqref{app.eq}. Notice that for every $\delta >0$, we can find sufficiently small $\epsilon=\epsilon(\delta)\in (0,1/2)$ so that $|\zeta|,|\pa_t \zeta|,|\pa_t^2 \zeta| \leq \delta$. Therefore, the new starting tuple $(\td m_0, c_0, \td R_0, \td \ph_0)$ also satisfies the inductive estimates \eqref{est.vp}-\eqref{est.ph} for sufficiently small $\epsilon$; in the estimate \eqref{est.R} we use $\pa_t \td m_0 = \pa_t \hat m_0$.

\subsection{Proof of Theorem \ref{thm.onsager}} Fix $\beta < \frac{1}{7}$ and choose $\alpha \in (\beta, \frac{1}{7})$. We let the density $\varrho(t,x) = \varrho_0(x)$ be stationary for all $t\in \R$, and set $M$, $\bar{b}(\al)$, and $\La_0$ as in Proposition \ref{ind.hyp}.
  For the non-trivial global energy less $\overline E$ defined by \eqref{def.E}, which is a function of time, the constructed starting tuple $(m_0, R_0, \ph_0)$ in Section \ref{sec 2.2} solves \eqref{app.eq} for $c= c_0$ and satisfies \eqref{est.vp}-\eqref{est.ph}, by adjusting $\bar b(\al)$ and $\La_0$ if necessary.

Now, we apply Proposition \ref{ind.hyp} iteratively to produce a sequence of approximate solutions $(m_q,R_q, \ph_q)$,
which solves \eqref{app.eq} with the energy less $\overline E$ and $c=c_q$, and satisfies \eqref{est.vp}-\eqref{est.ph}. Since the sequence $\{m_q\}$ also satisfies \eqref{cauchy}, it is Cauchy in $C^0([0, T];C^{\be}(\T^3))$. Indeed, for any $q<q'$, we have
the following estimates
\begin{align*}
\norm{m_{q'} - m_q}_{C^0([0, T];C^\be(\T^3))}
&\leq \sum_{l=1}^{q'-q}\norm{m_{q+l} - m_{q+l-1}}_{C^0([0, T];C^\be(\T^3))}\\
&\lec \sum_{l=1}^{q'-q}
\norm{m_{q+l} - m_{q+l-1}}_{0}^{1-\be}
\norm{m_{q+l} - m_{q+l-1}}_{1}^{\be}\\
&\lec \sum_{l=1}^{q'-q} 
\la_{q+1}^{\be} \de_{q+1}^{\frac12} = \sum_{l=1}^{q'-q} 
\la_{q+1}^{\be-\al} \to 0, 
\end{align*}
as $q$ goes to infinity because of $\be-\al<0$. Therefore, we obtain its limit $m$ in $ C^0([0, T];C^\be( \T^3))$. Since $(c_q, R_q, \ph_q)$ converges to $0$ in $C^0([0, T]\times \T^3)$, the limit $m$ solves the compressible equation \eqref{eqn.CE} with stationary density $\varrho$ and satisfies 
\begin{align*}
\pa_t \left( \frac{|m|^2}{2 \varrho} +  \varrho e( \varrho)\right) +
\div\left(\frac{m}{ \varrho} \left(
\frac{|m|^2}{2 \varrho} +  \varrho e( \varrho) + p( \varrho)\right)\right) = E'
\end{align*}
in the distributional distribution.
Here, we used the identity $\varrho\left(\pa_t + \frac{m}{\varrho}\cdot \na\right)\ka
= \pa_t \ka + \div (\ka m)$, $\ka=\frac 12 \tr(R)$. Since $E'<0$ for all $t\in [0,T]$, the constructed solution satisfies the entropy inequality \eqref{LEI} strictly. 

In particular, testing with a time-dependent function $\chi \in C^\infty_c ((0,T))$, we conclude
\[
- \dint_0^\infty \chi'(t) \int_{\mathbb T^3}  \frac{|m|^2}{2 \varrho} (t,x)\, dx\, dt = \int_0^\infty E'(t) \chi (t)\, dt. 
\] 
Given that $E$ is $C^1$, the latter implies that the total kinetic energy is a $C^1$ function and it in facts coincides with $E(t)/(2\pi)^3$ up to a constant addition, namely
\[
\dint_{\T^3} \frac{|m|^2}{2\varrho} (T,x)\, dx - \dint_{\T^3} \frac{|m|^2}{2\varrho} (0,x)\, dx = E(T,x ) - E(0,x)<0.
\] 
Therefore, the constructed solution has total kinetic energy dissipation \eqref{eqn.Onsager}.

Lastly, estimating 
\begin{align*}
    \norm{\pa_t m_q}_{0}
    \lec_{\varrho} \norm{m_q}_0\norm{m_q}_1 + \norm{\na p(\varrho)}_0 + \norm{R_q}_1 
    + \norm{\na \varrho}_0
    \lec_{\varrho, p} \la_q\de_q^\frac12,
\end{align*}
where $\norm{\cdot}_N= \norm{\cdot}_{C^0([0,T];C^N (\T^3))}$,
we have
\begin{align*}
    \norm{m_q}_{C^1([0,T]\times \T^3)}
    \lec_{\varrho, p} \la_q\de_q^\frac12.
\end{align*}
Then, the time regularity of the limit momentum $m$ can be concluded by interpolation argument. Hence $m\in C^\alpha ([0, T]\times \T^3)$.

\subsection{ Proof of Theorem \ref{thm}} 
In this argument we assume $T\geq 20$ {for convenience. Let $\varrho(t, \cdot) = \varrho_0$ for all $t\in \R$.} Fix $\beta< \frac 17$ and $\al\in (\be, \frac  17)$. {Then, choose $b$ and $\la_0$ in the range suggested in Proposition \ref{ind.hyp}.} Also, choose the initial approximate solution $(m_0, c_0, R_0, \ph_0)$ as in Section \ref{sec 2.2} with $E\equiv 0$. As before, we see that it solves \eqref{app.eq} {on $[0,T]+\tau_{-1}$} and satisfies \eqref{est.vp}-\eqref{est.ph}. 

Now, we apply Proposition \ref{ind.hyp} iteratively to produce a sequence of approximate solutions $(m_q,c_q,R_q, \ph_q)$, which solves \eqref{app.eq} with the energy loss $0$, and satisfies \eqref{est.vp}-\eqref{est.ph}. Since the sequence $\{m_q\}$ also satisfies \eqref{cauchy}, it is Cauchy in $C^0([0, T];C^{\be}(\T^3))$, and estimating the equation as before, we conclude that it is also Cauchy in $C^\beta ([0, T]\times \T^3)$.

On the other hand, fix $\bar q\in \N\cup\{0\}$ satisfying $b^{\bar  q}\geq \bar q$. At the $\bar q$th step using Proposition \ref{p:ind_technical} we can produce two distinct tuples, one which we keep denoting as above and the other which we denote by $(\td m_q, c_q, \td R_q, \td \ph_q)$ and satisfies \eqref{e:distance_and_support}, namely
\begin{align*}
\norm{\td m_{\bar q} - m_{\bar q}}_{C^0([0, T];L^2(\T^3))} \geq  {\e_0} \de_q^\frac 12, \quad \supp_t (m_q-\td m_q) \subset \cal I\, ,  
\end{align*}
with $\cal I  = (10, 10+3\tau_{\bar q-1})$. 
Applying now Proposition \ref{ind.hyp} iteratively, we can build a new sequence $(\td m_q, c_q, \td R_q,  \td \ph_q)$ of approximate solutions which satisfy \eqref{est.vp}-\eqref{cauchy} and \eqref{e:second_support}, inductively. Arguing as above, this second sequence converges to a solution $(\varrho_0, \td m)$ to the compressible Euler equation \eqref{eqn.CE}.
Indeed, $\td m\in C^\beta ([0, T]\times \T^3)$. We remark that for any $q\geq \bar q$,
\[
\supp_t(m_{q} -\td m_{q}) \subset \cal I + \sum_{q=\bar q}^\infty (\la_{q}\de_{q}^\frac12)^{-1} \subset[9,T],
\]
(by adjusting $\la_0$ to be even larger than chosen above, if necessary), and hence $\td m_{ {q}}$ shares initial data with $m_{{q}}$ { for all $q$. As a result, two solutions $\td m_q$ and $m_q$ have the same initial data.} However, the new solution $\td m$ differs from $m$ because
\begin{align*}
\norm{m - \td m}_{C^0([0, T];L^2(\T^3))}
&\geq \norm{m_{\bar q} -\td m_{\bar q}}_{C^0([0,T];L^2(\T^3))}
-\sum_{q=\bar q}^\infty \norm{m_{q+1}-m_q - (\td m_{q+1} - \td m_q)}_{C^0([0,T];L^2(\T^3))} \\
&\geq \norm{m_{\bar q} - \td m_{\bar q}}_{C^0([0,T];L^2(\T^3))}
-(2\pi)^\frac32 \sum_{q=\bar q}^\infty (\norm{m_{q+1}-m_q}_0 + \norm{\td m_{q+1}-\td m_q}_0)\\
&\geq  {\e_0} \de_{\bar q}^\frac 12  - 2(2\pi)^\frac32M \sum_{q=\bar q}^\infty\de_{q+1}^\frac12>0.
\end{align*} 
The last inequality follows from adjusting $\la_0$ to a larger one if necessary. 
By changing the choice of time interval $\cal I$ and the choice of $\bar{q}$, we can easily generate infinitely many solutions. {Since we choose $E=0$, the infinitely many solutions satisfy the entropy equality.}

\section{Construction of the momentum correction}\label{sec:correction.const}

In this section we detail the choice of the correction $n:= m_{q+1}-m_q$. As in the literature  which started from the paper \cite{DLSz2013}, the perturbation $n$ is, in first approximation, obtained from a family of highly oscillatory stationary solution of the {\it incompressible} Euler equation, which are modulated by the errors $R_q$ and $\varphi_q$ and transported along the course grain flow of the background vector field $m_q/\varrho$. For stationary Euler flows, we use Mikado flows, introduced in section \ref{sec 2.2},   
\[
\sum_{f\in \mathcal{F}} \gamma_f U_f (\lambda (x-p(f))).\]

In the construction of perturbation, $f$ will vary in a finite fixed set of directions $\mathcal{F}$ (which in fact will have cardinality ${270}$) and for each $f$ we will specify an appropriate choice of $\vartheta$. 
In a first approximation we wish to define our perturbation $m_{q+1} - m_q$ as
\[
\sum_{f\in \mathcal{F}} \gamma_f (R_q(t,x), \varphi_q (t,x), \varrho(t,x)) U_f (\lambda (x-p(f))
\]
where the coefficients $\gamma_f$ are appropriately chosen smooth functions (later on called ``weights''), $\lambda$ is a very large parameter and the $x(f)$ are appropriately chosen shifts to ensure the disjoint support condition of Lemma \ref{l:Mikado}. 
As already pointed out such Ansatz must be corrected and we need to modify the perturbation so that it is approximately advected by the velocity $m_q/\varrho$. Note that on large time-scales the flow of the velocity $m_q/\varrho$ does not satisfy good estimates, while it satisfies good estimates on a sufficiently small scale $\tau_q$. Following \cite{BuDLeIsSz2015}, \cite{Is2013}, and \cite{DLK20} this issue is solved by introducing a partition of unity in time and restarting the flow at a discretized set of times, roughly spaced according to the parameter $\tau_q$. Like the case in \cite{DLK20}, one has to face the delicate problem of keeping the supports of the various Mikado flows disjoint. This is done by discretizing the construction in space too, taking advantage of the fact that for sufficiently small space and time scales, the supports of the transported Mikado flows remain roughly straight thin tubes: the argument requires then a subtle combinatorial choice of the ``shifts''. As in \cite{DLK20}, the introduction of the space-time discretization deteriorates the estimates and accounts for the H\"older threshold $1/7$.

\subsection{Mikado directions}\label{ss:directions} To determine a set of suitable directions $f$, we recall two geometric lemmas \cite{Nash, DLSz2013, DLK20}. In the first we denote by $\mathbb{S}$ the subset of $\mathbb R^{3\times 3}$ of all symmetric matrices, set $|K|_\infty := |(k_{lm})|_\infty = \max_{l,m} |k_{lm}|$ for $K\in \mathbb R^{3\times 3}$.
\begin{lem}[Geometric Lemma I]\label{lem:geo1} 
Let $\cF= \{f_i\}_{i=1}^6$ be a set of vectors in $\Z^{3}$ and $C$ a positive constant such that
\begin{equation}\label{con.FIR}
\sum_{i=1}^6 f_i \otimes f_i = C\I, \quad\mbox{and}\quad
\{f_i\otimes f_i\}_{i=1}^6 \text{ forms a basis of }\,\, \mathbb{S}\, .
\end{equation}
Then, there exists a positive constant $N_0 = N_0(\cF)$ such that for any $N\leq N_0$, we can find functions $\{\Ga_{f_i}\}_{i=1}^6\subset C^\infty(S_{N}; (0,\infty))$, with domain $S_{N}:= \{\I-K: K \text{ is symmetric, } |K|_\infty\leq N\}$,
satisfying 
\begin{align*}
\I -K = \sum_{i=1}^6 \Ga_{f_i}^2(\I - K) (f_i\otimes f_i), 
\quad\forall (\I-K) \in S_{N}\, .
\end{align*}
\end{lem}

\begin{lem}[Geometric Lemma II] \label{lem:geo2}
Suppose that 
\begin{equation}\label{con.FIph}
\{f_1, f_2, f_3\}\subset \Z^3\setminus\{0\} \text { is an orthogonal frame and } 
f_4=-(f_1+ f_2+f_3).
\end{equation}
Then, for any $N_0>0$, there are affine functions $\{\Ga_{f_k}\}_{1\leq k \leq 4} \subset C^\infty(\cal{V}_{N_0}; [N_0, \infty))$ with domain $\cal{V}_{N_0} := \{m\in \R^3: |m| \leq N_0\}$ such that
\[
m = \sum_{k=1}^4 \Ga_{f_k} (m) f_k  \qquad\forall m \in \cal{V}_{N_0}\, .
\] 
\end{lem}

Based on these lemmas, we choose 27 pairwise disjoint families $\cF^j$ indexed by $j\in \mathbb{Z}_3^3$, where each $\cF^j$ consists further of two (disjoint) subfamilies $\cF^{j,R}\cup\cF^{j,\ph} $ with cardinalities $|\cF^{j,R}|=6$ and  $|\cF^{j,\ph}|=4$, chosen so that $\cF^{j,R}$ and $\cF^{j,\ph}$ satisfy \eqref{con.FIR} and \eqref{con.FIph}, respectively. For example, for $j= (0,0,0)$ we can choose
\begin{align*}
\cF^{j,R} = \{(1,\pm 1, 0), (1, 0, \pm 1), (0, 1, \pm1)\}, \quad
\cF^{j,\ph} = \{(1,2, 0), (-2, 1, 0), (0, 0, 1), (1,-3,-1)\}
\end{align*}
and then we can apply $26$ suitable rotations (and rescalings). 
Next, the function $\vartheta$ will be chosen for each $f$ in two different ways, depending on whether $f\in \cF^{j,R}$ or $f\in \cF^{j,\ph}$. Introducing the shorthand notation
$\langle u \rangle = \dint_{\T^3} u (x)\, dx$, we impose the moment conditions
\begin{equation}\label{con.psi2}
\begin{split}
\langle \psi_f \rangle = \langle \psi_f^3 \rangle = 0, \quad \langle \psi_f^2\rangle =1  \qquad \forall f\in \cF^{j,R},\\
\langle \psi_f\rangle = 0, \quad \langle \psi_f^3 \rangle=1
\qquad \forall f\in \cF^{j,\ph}.
\end{split}
\end{equation} 
The main point is that the Mikado directed along $f\in \cF^{j,R}$ will be used to ``cancel the error $R_q$'', while the ones directed along $f\in \cF^{j, \ph}$ will be used to ``cancel the error $\ph_q$'' and the different moment conditions will play a major role. In both cases we assume also that
\begin{equation}\label{e:eta}
\supp (\psi_f) \subset B \left(l_f, \frac{\eta}{10}\right)
{ := \{x\in \R^3: |x-y|<\textstyle{\frac {\eta}{10}} \text{ for some } y\in l_f\} }
\, , 
\end{equation}
where $\eta$ is a geometric constant which will be specified later, cf. Proposition \ref{p:supports}. 

\subsection{Regularization and drift}\label{ss:regularization} We start with smoothing the tuple $(m_q, R_q, \varphi_q)$. To this end, we first introduce the parameters $\ell$ and $\ell_t$, defined by
\[
\ell = \frac 1{\la_q^\frac34\la_{q+1}^\frac14} \left(\frac{\de_{q+1}}{\de_q}\right)^\frac38, \quad
\ell_t = \frac 1{\la_q^{\frac12-3\ga} \la_{q+1}^\frac12 \de_q^\frac14\de_{q+1}^\frac14 }\, .
\]
The space regularization of $m_q$ is defined by applying a ``low-pass filter'' which roughly speaking eliminates all the waves larger than $\ell^{-1}$. 
In order to do so we first introduce some suitable notation. First of all, for a function $f$ in the Schwartz space $\cal{S}(\R^3)$, the Fourier transform of $f$ and its inverse on $\R^3$ are denoted by 
\[
\widehat{f} (\xi) = \frac 1{(2\pi)^3}\int_{\R^3} f(x) e^{-ix \cdot \xi} dx,\quad
\widecheck{f} (x) = \int_{\R^3} f(\xi) e^{ix \cdot \xi} d\xi. 
\]   
As usual, we understand the Fourier transform on more general functions as extended by duality to $\mathcal{S}' (\R^3)$. Since practically all the objects considered in this note are functions, vectors and tensors defined on $I\times \mathbb T^3$ for some time domain $I\subset \mathbb R$, regarding them as spatially periodic functions on $I\times \mathbb R^3$, we will consider their Fourier transform as time-dependent elements of $\mathcal{S}' (\R^3)$. We then follow the standard convention on Littlewood-Paley operators. We let $\phi(\xi)$ be a radial smooth function supported in ${B(0,2)}$ which is identically $1$ on $\overline{B(0,1)}$. For any number $j\in \Z$ and distribution $f$ in $\R^3$, we set
\begin{align*}
\widehat{P_{\leq 2^j} f}(\xi) &:=\phi \left(\frac{\xi}{2^j}\right) \hat{f}(\xi),\quad
\widehat{P_{> 2^j} f}(\xi) := \left(1-\phi \left(\frac{\xi}{2^j}\right)\right) \hat{f}(\xi), 
\end{align*}
and for $j\in \Z$
\begin{align*}
\widehat{P_{2^j}f}(\xi) :=\left( \phi \left(\frac{\xi}{2^j}\right)
-\phi \left(\frac{\xi}{2^{j-1}}\right)\right) \hat{f}(\xi).
\end{align*}
For a positive real number $S$, we finally let $P_{\leq S}$ equal the operator $P_{\leq 2^J}$ for the largest $J$ such that $2^J\leq S$. 
We are thus ready to introduce the coarse scale velocity $m_\ell$ defined by 
\begin{align}\label{def.vl}
m_\ell = P_{\le \ell^{-1}}m_{{q}}, 
\end{align}
Note that, regarding $m_q$ as a spatially periodic function on $I\times \mathbb T^3$, $P_{\leq \ell^{-1}} m$ can be written as the space convolution of $m$ with the kernel $2^{3J} \widecheck{\phi} (2^J \cdot)$, which belongs to $\mathcal{S} (\R^3)$. In particular $m_\ell$ is also spatially periodic and will be in fact regarded as a function on $I\times \mathbb T^3$. Similar remarks apply to several other situations in the rest of this note.

The regularization of the errors $R_q$ and $\varphi_q$ is more laborious and follows the intuition that, while we need to regularize them in time and space, we want such regularization to give good estimates on their advective derivatives along $\frac{m_\ell}{\varrho}$, for which we introduce the ad hoc notation
\[
D_{t, \ell} := \pa_t + \frac{m_\ell}{\varrho} \cdot \na\, .
\]
First of all we let $\Phi(\tau,x;t)$ be the forward flow map with the drift velocity $\frac{m_\ell}{\varrho}$ defined on some time interval $[a,b]$ starting at the initial time $t\in [a,b)$:
\begin{align}\label{e:fflow_first_instance}
\begin{dcases}
\pa_\tau \Phi(\tau, x;t)= \frac{m_\ell}{\varrho} (\tau, \Phi(\tau, x;t))\\
\Phi(t, x;t) = x\, .
\end{dcases}
\end{align}

\begin{rem}\label{r:periodicity}
Strictly speaking the map above is defined on $[a,b] \times \mathbb R^3$. Note however that the periodicity of $\frac{m_\ell}{\varrho}$ implies that $\Phi$ induces a well-defined map from $[a,b]\times \mathbb T^3$ into $\mathbb T^3$. From now on we will implicitly identify both maps.
\end{rem}

We then take a standard mollifier $ \rho$ on $\R$, namely a nonnegative smooth bump function satisfying $\norm{ \rho}_{L^1(\R)}=1$ and $\supp \rho\subset (-1,1)$. As usual we set $ \rho_\delta (s) = \delta^{-1}  \rho( \delta^{-1} s)$ for any $\de>0$. We can thus introduce the mollification along the trajectory
\begin{align*}
( \rho_\delta   \ast_{\Phi} F) (t,x)
= \int_{\R} F(t+s,\Phi(t+s,x;t))   \rho_\delta (s) ds. 
\end{align*}
(Note that if $F$ and $m_\ell$ are defined on some time interval $[a,b]$, then $ \rho_\delta   \ast_{\Phi} F$ is defined on $[a,b]-\delta$.)
This mollification can be found in \cite{Is2013}
and is designed to satisfy 
\begin{equation}\begin{split}\label{char.mol.traj}
D_{t,\ell} ( \rho_\delta   \ast_{\Phi} F)(t,x)
&= \int (D_{t,\ell}F)(t+s, \Phi(t+s, x;t))  \rho_\delta (s) ds\\
&= - \int F(t+s, \Phi(t+s, x;t))  \rho'_\delta (s) ds\, .
\end{split}\end{equation}
 The regularized errors are then given by
\begin{align}\label{def.me}
R_\ell =  \rho_{\ell_t}   \ast_{\Phi} P_{\leq \ell^{-1}} R_{{q}}, \quad
\ph_\ell =  \rho_{\ell_t}   \ast_{\Phi} P_{\leq \ell^{-1}}\ph_{{q}}. 
\end{align}
These errors can be defined on $[0,T]+2\tau_q$ by the choice of sufficiently large $\la_0$. We will need later quite detailed estimates on the difference between the original tuple and the regularized one and on higher derivatives of the latter. Such estimates are in fact collected in Section \ref{s:estimates_mollification}.

\subsection{Partition of unity and shifts}\label{ss:discretization}  We first introduce nonnegative smooth functions $\{\chi_v\}_{v\in\Z^3}$ and $\{\th_u\}_{u\in \Z}$ whose sixth powers give suitable partitions of unity in space $\R^3$ and in time $\R$, respectively:
\[
\sum_{v\in \Z^3} \chi_v^6(x) =1, \quad
\sum_{u\in \Z} \th_u^6(t) =1.   
\]
Here, $\chi_v(x) = \chi_0(x-{2}\pi v)$ where $\chi_0$ is a non-negative smooth function supported in $Q(0, {9/8}\pi)$ satisfying $\chi_0 = 1$ on $\overline{Q(0, {7/8}\pi)}$, where from now on $Q (x, r)$ will denote the cube $\{y: |y-x|_\infty < r\}$ (with $|z|_\infty := \max \{|z_i|\}$). Similarly, $\th_u(t) = \th_0(t-u)$ where $\th_0\in C_c^\infty(\R)$ satisfies $\th_0 = 1$ on $[1/8, 7/8]$ and $\th_0 =0$ on $(-1/8,9/8)^c$. Then, we divide the integer lattice $\Z^3$ into 27 equivalent families $[j]$ with $j \in \mathbb{Z}_3^3$ via the usual equivalence relation
\[
v=(v_1,v_2,v_3) \sim \td v = (\td v_1,\td v_2,\td v_3) 
\iff v_i \equiv \td v_i\; \mod 3\quad \text{for all }i=1,2,3.
\]
We use these classes to define the set of indices 
\[
\mathscr{I} := \{(u,v, f): (u,v)\in \mathbb Z\times \mathbb Z^3 \quad\mbox{and}\quad f\in \mathcal{F}^{[v]}\}\, .
\]
For each $I$ we denote by $f_I$ the third component of the index and we further subdivide $\mathscr{I}$ into $\mathscr{I}_R\cup \mathscr{I}_\varphi$ depending on whether $f_I\in \mathcal{F}^{[v],R}$ or $f_I\in \mathcal{F}^{[v], \varphi}$. 
Next we introduce the parameters $\tau = \tau_q$ and $\mu = \mu_q$ with $\tau_q^{-1}>0$ and $\mu_q^{-1}\in \mathbb N \setminus \{0\}$, which are explicitly given by
\begin{align}\label{mu.tau}
\mu_q^{-1} = 3 \ceil{\la_q^{\frac12}\la_{q+1}^{\frac12} \de_q^\frac14\de_{q+1}^{-\frac14} /3}, \quad 
\ta_q^{-1} = 40\pi \dot{C}_{\varrho} {M}\eta^{-1}\cdot {\la_q^\frac12\la_{q+1}^\frac12 \de_q^\frac14\de_{q+1}^\frac14 },
\end{align}
where $\dot{C}_{\varrho}$ is chosen as a constant depending on $\varrho$ such that $\norm{\na (m_q/{\varrho})}_0 \leq \dot{C}_{\varrho}M\la_q\de_q^\frac12$.
(For the motivation of $\tau_q$ and $\eta$, see \eqref{con.tau} and Proposition \ref{p:supports}).
We define 
\begin{align*}
\th_I(t) = 
\begin{cases}
\th_u^3(\tau^{-1}t), \quad I\in \mathscr{I}_R\\
\th_u^2(\tau^{-1}t), \quad I\in \mathscr{I}_\ph,
\end{cases}
\quad
\chi_I(x)=
\begin{cases}
\chi_v^3(\mu^{-1}x), \quad I\in \mathscr{I}_R\\
\chi_v^2(\mu^{-1}x), \quad I\in \mathscr{I}_\ph\, .
\end{cases}
\end{align*}
Next, for each $I$ let $U_{f_I}$ be the corresponding Mikado flow. Moreover, given $I= (u,v,f)$, denote by $t_u$ the time $t_u = u\tau$ and let $\xi_I = \xi_u$ be the solution of the following PDE (which we understand as a map on $\mathbb  R \times \mathbb T^3$ taking values in $\mathbb T^3$, cf. Remark \ref{r:periodicity}):  
\begin{align}\label{def.bflow}
\begin{cases}
\pa_t \xi_u + (\frac{m_\ell}{\varrho} \cdot \na)\xi_u =0\\
\xi_u(t_u,x) =x 
\end{cases}
\end{align}
In the rest of the paper $\nabla \xi_I$ will denote the Jacobi Matrix of the partial derivatives of the components of the vector map $\xi_I$ and we will use the shorthand notations $\nabla \xi_I^\top$, $\nabla \xi_I^{-1}$ and $\nabla \xi_I^{-\top}$ for, respectively, its transpose, inverse and transport of the inverse. Moreover, for any vector $f\in \mathbb R^3$ and any matrix $A\in \mathbb R^{3\times s}$ the notation $\nabla \xi_I f$ and $\nabla \xi_I A$ (resp. $\nabla \xi_I^{-1} f$, etc.) will be used for the usual matrix product, regarding $f$ as a column vector (i.e. a $\mathbb R^{3\times 1}$-matrix). 
 
For each $I = (u,v,f)$ we will also choose a shift 
\[
z_I= z_{u,v} + \bar x_f \in \mathbb R^3
\] 
and, setting $\lambda= \lambda_{q+1}$, we are finally able to introduce the main part of our perturbation, which is achieved using the following ``master function''
\begin{equation}\label{e:master}
\dot{N}(R, \varphi, \varrho, t,x) := \sum_{I\in \mathscr{I}} \theta_I (t) \chi_I (\xi_I (t,x)) {\gamma}_I (R, \varphi,\varrho, t,x) \nabla \xi_I^{-1} (t,x) U_{f_I} (\lambda (\xi_I (t,x) - z_I))\, ,
\end{equation}
where the $\gamma_I$'s are smooth scalar functions (the ``weights'') whose choice will be specified in the next section.
In order to simplify our notation we will use $U_I$ for $U_{f_I} (\cdot - z_I)$, $\psi_I$ for $\psi_{f_I} (\cdot - z_I)$ and $\tilde{f}_I$ for $\nabla \xi_I^{-1} f_I$. We therefore have the writing
\begin{equation}\label{e:master2}
\dot{N} := \sum_{I\in \mathscr{I}} \theta_I \chi_I (\xi_I) {\gamma}_I \tilde{f}_I \psi_I (\lambda \xi_I)\, .
\end{equation}
Note that, since we want $\dot{N}$ to be a periodic function of $x$, we will impose 
\begin{equation}\label{e:periodicity}
z_{u,v} = z_{u, v'} \qquad \mbox{if ${\mu}(v-v')\in 2\pi \mathbb Z^3$.}
\end{equation}
Finally, the preponderant part of the correction $m_{q+1}-m_q$ will take the form
\begin{equation}\label{e:wo}
n_o (t,x) := \dot{N} (R_\ell (t,x), \varphi_\ell (t,x), \varrho(t,x), t,x)\, 
\end{equation}
which is well-defined on $[0,T]+ 2\tau_q$. (Indeed, it is possible to have $[0,T]+ 3\tau_q\subset [0,T]+ \tau_{q-1}$ by the choice of sufficiently large $\la_0$). In the rest of Section \ref{sec:correction.const}, without mentioning, our analysis is done in the time interval $[0,T]+ 2\tau_q$.
Given the complexity of several formulas and future computations, it is convenient to break down the functions $ \dot{N}$ and $n_o$ in more elementary pieces. To this aim we introduce the scalar maps
\[
n_I (t,x) := \theta_I (t) \chi_I (\xi_I (t,x)) \psi_I (\lambda (\xi_I (t,x)))\, ,
\]
using which we can write
\[
\dot{N} (R, \varphi, \varrho, t, x) = \sum_{I\in \mathscr{I}} \gamma_I (R, \varphi, \varrho, t,x) \nabla \xi_I (t,x)^{-1} f_I n_I (t,x) 
= \sum_{I\in \mathscr{I}} \gamma_I (R, \varphi, \varrho, t,x) \tilde{f}_I (t,x) n_I (t,x)\, 
\]
and
\[
n_o = \sum_{I\in \mathscr{I}} \gamma_I  \nabla \xi_I^{-1} f_I n_I =  \sum_{I\in \mathscr{I}} \gamma_I  \tilde{f}_I n_I\, .
\]

The crucial point in our construction is the following proposition, obtained in \cite[Proposition 3.5]{DLK20}
\begin{prop}\label{p:supports}
There is a constant $\eta = \eta (\mathcal{F})$ in \eqref{e:eta} such that it allows a choice of the shifts $z_I= z_{u,v} + \bar x_f$ which ensure that {for each $(\mu_q, \tau_q, \la_{q+1})$,} the conditions $\supp (n_I) \cap \supp (n_J)=\emptyset$ for every $I\neq J$ and that \eqref{e:periodicity} for every $u,v$ and $v'$. 
\end{prop}
Note that the difference between Proposition \ref{p:supports} and Proposition 3.5 in \cite{DLK20} is that in \cite{DLK20} the drift velocity is divergence-free while $m_{\ell}/{\varrho}$ is not. However, the proof is not relying on the divergence-free condition. Also, we remark that the proof of Proposition \ref{p:supports} requires the choice of parameters satisfying the following relations:
\begin{align}\label{con.tau}
    \mu_q^{-1}\ll \la_{q+1} \in \N, \quad
    \tau_q \Norm{\na ({m_q}/{\varrho})}_0 
    \leq \frac 1{10}, \quad
    \mu_q \tau_q \Norm{\na ({m_q}/{\varrho})}_0
    \leq \frac{\eta}{10\pi \la_{q+1}}
\end{align}
where $\eta$ is a positive constant determined by $\cF = \bigcup_{j\in \mathbb{Z}^3_3} \cF^j$, which has finite cardinality.

\subsection{Choice of the weights}\label{ss:weights} We next detail the choice of the functions $\gamma_I$, subdividing it into two cases.
\subsubsection{Energy weights} The weights $\ga_I$ for $I\in \mathscr{I}_\ph$ will be chosen so that the low frequency part of $\frac 12|n_o|^2 n_o$ makes a cancellation with the mollified unsolved current $\varrho^3 \ph_\ell$. 
Because of Proposition \ref{p:supports}, we have
\begin{align*}
|n_o|^2 n_o  =&\ \sum_{I\in \mathscr{I}}  \th_I^3 \chi_I^3(\xi_I) \ga_I^3 \psi_I^3(\la_{q+1}\xi_I) |\na \xi_I^{-1} f_I|^2\na \xi_I^{-1}  f_I\\
=&\underbrace{\sum_{I\in \mathscr{I}}  \th_I^3 \chi_I^3(\xi_I) \ga_I^3 \langle \psi_f^3\rangle   |\td f_I|^2\td f_I}_{=: (|n_o|^2 n_o)_L}
+\underbrace{\sum_{I\in \mathscr{I}}  \th_I^3 \chi_I^3(\xi_I) \ga_I^3 \left(\psi_I^3(\la_{q+1}\xi_I)-\langle \psi_I^3\rangle
\right) |\td f_I|^2\td f_I}_{=:(|n_o|^2 n_o)_H}\, .
\end{align*}
In order to find the desired $\ga_I$, we introduce the notation $\mathscr{I}_{u,v,\ph}$ for $\{I\in \mathscr{I}_\ph : I = (u,v,f)\}$ and we observe that, by \eqref{con.psi2}, it suffices to achieve
\begin{align}\label{W3L}
(|n_o|^2 n_o)_L
=\sum_{u,v} \th_u^6\left(\frac t{\tau_q}\right) \chi_v^6\left(\frac{\xi_u}{\mu_q}\right)\sum_{I \in \mathscr{I}_{u,v,\ph}} \ga_I^3 |\tilde{f}_I|^2 \tilde{f}_I\, .
\end{align}
Next we look for our coefficients in the following form:
\[
\ga_I =\frac{\la_q^{-\ga}\de_{q+1}^\frac12 \Ga_I}{|\tilde{f}_I|^\frac23}\, ,
\]
where $\Gamma_I$ will be specified in a moment.

Recall that $\xi_I$ is a solution to \eqref{def.bflow} and satisfies $\na \xi_I |_{t=t_u} = \I$ and 
\[
\na \xi_I^{-1} (t,x) = \na \Phi_u (t, \xi_I(t,x))\, ,
\] 
where $\Phi_u$ is the ``forward flow''
$\Phi (t,x;t_u)$ introduced in \eqref{e:fflow_first_instance} and thus solves
\begin{equation}\label{eqn.fflow}
\begin{cases}
\pa_t \Phi_u(t,x) =\frac{m_\ell}{\varrho} (t,\Phi_u(t,x))\\
\Phi_u(t_u, x) = x. 
\end{cases}
\end{equation}
This implies that  
\begin{align}
\norm{\na \xi_I}_{C^0(\cal{I}_u\times \R^3)} &\leq  \exp\left(2\tau_q \Norm{\na ({m_\ell}/{\varrho})}_0\right)
\leq  \exp(2\dot{C}_\varrho{M}\tau_q\la_q\de_q^\frac12), \label{est.naxi}
\end{align}
\begin{equation}\begin{split}
\norm{\I - \na \xi_I^{-1}}_{C^0(\cal{I}_u\times \R^3)} 
&{= \norm{\I - \na \Phi_u}_{C^0(\cal{I}_u\times \R^3)}}\\
&\le 2\tau_q \norm{\na \Phi_u}_{C^0(\cal{I}_u\times \R^3)}\Norm{\na ({m_\ell}/{\varrho})}_0\\
&\leq 2\dot{C}_{\varrho}M\tau_q\la_q\de_q^\frac12\exp(2\dot{C}_{\varrho}M\tau_q \la_q\de_q^\frac12 ) \label{bf.Id}
\end{split}\end{equation}
for the time interval $\cal{I}_u := [t_u -\frac 12\tau_q, t_u + \frac 32\tau_q]  \cap  [0,T]+2\tau_q$. 
Therefore, for sufficiently large $\la_0$, we have 
\begin{align*}
|\tilde{f}_I| = |\na \xi_I^{-1}f_I|&\geq \frac 34\\
\norm{2\la_q^{3\ga}\de_{q+1}^{-\frac32}(\na\xi_I)\varrho^3 \ph_\ell}_{C^0_x}&\leq 3C_1(\varrho)
\end{align*} 
on the support of $\th_I$ for some positive constant $C_1(\varrho)$.  
Since $\{f_I : I \in \mathscr{I}_{u,v, \ph}\}= \mathcal{F}^{[v], \ph}$ satisfies \eqref{con.FIph}, we can apply Lemma \ref{lem:geo2} with $N_0= 3C_1(\varrho)$ to solve
\begin{align*}
-2 \varrho^3 \ph_\ell
=\sum_{I\in \mathscr{I}_{u,v,\ph}} \ga_I^3|\tilde{f}_I|^2 \tilde{f}_I
\iff 
-2\la_q^{3\ga}\de_{q+1}^{-\frac32}(\na\xi_I)\varrho^3 \ph _\ell = \sum_{I\in \mathscr{I}_{u,v\ph}}\Ga_I^3 f_I
\end{align*}
on each support of $\th_I$ (observe that we have crucially used that $\xi_I=\xi_u$ is independent of $f_I$ for $I\in \mathscr{S}_{u,v,\ph}$). We are thus in the position to apply Lemma \ref{lem:geo2} to the set $\mathcal{F}^{[v],\ph} = \{f_I: I \in \mathscr{I}_{u,v,\ph}\}$ and we let $\Gamma_{f_I}$, $I\in \mathscr{I}_{u,v,\ph}$ be the corresponding functions. As a result, we can set
\begin{align}\label{Ga.ph}
\Ga_I(t,x) 
={\Ga}_{f_I}^{\frac{1}{3}}
({-2
{\la_q^{3\ga}\de_{q+1}^{-\frac 32}(\na\xi_I)\varrho^3 \ph_\ell}} 
)\,.
\end{align}
Note that the smoothness of the selected functions $\Gamma_{f_I}$ depends only on $C_1$ and that in fact Lemma \ref{lem:geo2} is just applied $27$ times, taking into consideration that $[v]\in \mathbb{Z}_3^3$. For later use we record here the important ``cancellation property'' that the choice of our weights achieves:
\begin{align}\label{can.ph}
\frac {1}{2\varrho^2} (|n_o|^2n_o)_L
=- \varrho \ph_\ell.
\end{align}

\subsubsection{Reynolds weights}\label{subsec:R}
Similarly to the previous section we decompose $n_o \otimes n_o$ into the low and high frequency parts,
\begin{align*} 
n_o \otimes n_o 
=&\ \sum_I \th_I^2 \chi_I^2(\xi_I) \ga_I^2 \psi_I^2(\la_{q+1}\xi_I)
\tilde f_I\otimes \tilde f_I\\
=&\underbrace{\sum_I \th_I^2 \chi_I^2(\xi_I) \ga_I^2 \langle\psi_I^2\rangle \tilde f_I\otimes \tilde f_I}_{=: (n_o \otimes n_o)_L}
+\underbrace{\sum_I \th_I^2 \chi_I^2(\xi_I) \ga_I^2 \left(\psi_I^2(\la_{q+1}\xi_I)-\langle\psi_I^2\rangle \right) \tilde f_I\otimes \tilde f_I}_{=: (n_o \otimes n_o)_H}.
\end{align*}
Since the weights for $I\in \mathscr{I}_\ph$ have already been established, for each fixed $(u,v)$ we denote by $I (u,v)$ the sets of indices $(u', v')$ such that $\max \{|u-u'|_\infty, |v-v'|_\infty\}\leq 1$ (where $|w|_\infty := \max \{|w_1|, |w_2|, |w_3|\}$ for any $w\in \mathbb R^3$) and {rewrite
\begin{align*}
& (n_o \otimes n_o)_L
= \sum_{u,v} \th^6_u\left(\frac t{\tau_q}\right) \chi^6_v\left(\frac{\xi_u}{\mu_q}\right) 
\sum_{I\in \mathscr{I}_{u,v,R}} \ga_I^2 \td f_I \otimes \td f_I
+\sum_{J\in  \mathscr{I}_{\ph}} \th_J^2 \chi_J^2(\xi_I) \ga_J^2 \langle\psi_J^2\rangle \tilde f_J\otimes \tilde f_J\\
=& \sum_{u,v} \th^6_u\left(\frac t{\tau_q}\right) \chi^6_v\left(\frac{\xi_u}{\mu_q}\right) \left[
\sum_{I\in \mathscr{I}_{u,v,R}} \ga_I^2 \td f_I \otimes \td f_I
+\sum_{\substack{J\in  \mathscr{I}_{u',v',\ph}\\ (u',v')\in I(u,v)}}\th_J^2 \chi_J^2(\xi_I) \ga_J^2 \langle\psi_J^2\rangle \tilde f_J\otimes \tilde f_J\right].
\end{align*}}
To make $n_o \otimes n_o$ cancel out $\varrho^2(\de_{q+1}\I - {R_\ell})$, we recall that $\tilde{f}_I\otimes \tilde{f}_I = \nabla \xi_I^{-1} (f_I\otimes f_I) \nabla \xi_I^{-\top}$ and set
\begin{equation}\begin{split}
\label{eq.Ga}
&\sum_{I\in \mathscr{I}_{u,v,R}} \ga_I^2 f_I \otimes f_I\\
&\quad = \na \xi_I 
\Bigg[\varrho^2(\de_{q+1} \I  -R_\ell) -\sum_{(u',v') \in I (u,v)} \sum_{J\in \mathscr{I}_{u',v',\ph}} \theta_J^2 \chi_J^2 \ga_{J}^2 (\xi_J) \langle\psi_{J}^2\rangle  \tilde{f}_J \otimes  \tilde{f}_J \Bigg]
\na \xi_I^{\top}\, . 
\end{split}\end{equation}
We now define $\mathcal{M}_I$ as 
\begin{align*}
\mathcal{M}_I &= \de_{q+1} [\na \xi_I \na \xi_I^{\top}-\I ] -
\na \xi_I R_\ell \na \xi_I^{\top}  \\
&\quad
-\na\xi_I \frac{1}{\varrho^2}\left[\sum_{(u',v') \in I (u,v)} \sum_{J\in \mathscr{I}_{u',v',\ph}}\th_{J}^2\chi_{J}^2(\xi_{J}) \ga_{J}^2 \langle \psi_{J}^2\rangle  \tilde f_J \otimes  \td f_J \right] \na\xi_I^{\top}
\end{align*}
and $\gamma_I = \delta_{q+1}^{\frac{1}{2}} \varrho \Gamma_I$ (for $I\in \mathscr{I}_{u,v,R}$) and impose
\begin{equation} \label{eq.Ga1}
\sum_{I\in \mathscr{I}_{u,v,R}} \Ga_I^2 f_I  \otimes  f_I
= \I + \de_{q+1}^{-1} \mathcal{M}_I
\end{equation}
In order to show that such a choice is possible observe that we can make $\norm{\de_{q+1}^{-1}\mathcal{M}_I}_{C^0(\supp(\th_I)\times \R^3)}$ sufficiently small, provided that $\la_0$ is sufficiently large, because of \eqref{est.naxi}, \eqref{bf.Id}, $\norm{\de_{q+1}^{-1}R_\ell}_0\lec \la_q^{-3\ga}$, and, $\norm{\de_{q+1}^{-1}\ga_J^2}_0\lec \la_q^{-2\ga}$ when $J \in \mathscr{I}_{u',v',\ph}$. We can thus apply Lemma \ref{lem:geo1} to $\{f_I : I \in \mathscr{I}_{u,v,R}\}= \mathcal{F}^{[v],R}$ and, denoting by $\Gamma_{f_I}$ the corresponding functions, we just need to set 
\[
\Ga_I=\Ga_{f_I} ( \I + \de_{q+1}^{-1} \mathcal{M}_I)\, .
\] 
Observe once again that this means applying Lemma \ref{lem:geo1} just $27$ times, given that there are $27$ different families $\mathcal{F}^{[v], R}$. 
We finally record the desired ``cancellation property'' that the choice of the weights achieves:
\begin{equation}\begin{split} \label{can.R}
( n_o\otimes n_o)_L
&=\sum_{u,v}\th_u^6\left(\frac t{\tau_q}\right) \chi_n^6\left(\frac{\xi_u}{\mu_q}\right)  \varrho^2 (\de_{q+1}\I-R_\ell)
=\varrho^2 (\de_{q+1}\I-R_\ell).
\end{split}\end{equation}

\subsection{Fourier expansion in fast variables and corrector $n_c$}\label{ss:correction} In the rest of this article, we use a representation of $n_o$, $n_o\otimes n_o$, and $\frac12 |n_o|^2n_o$ based on the Fourier series of $\psi_I$, $\psi_I^2$ and $\psi_I^3$.
Indeed, since $\psi_I$ is a smooth function on $\T^3$ with zero-mean, we have
\begin{equation}\label{rep.psi}
\psi_I (x) = \sum_{k\in \Z^3\setminus \{0\}} \dot{b}_{I,k}  e^{ i k\cdot x},\quad
\psi_I^2 (x) = \dot c_{I,0} 
 +\sum_{k\in \Z^3\setminus \{0\}} \dot c_{I,k}   e^{ i k\cdot x},\quad
\psi_I^3 (x) = \dot d_{I,0}  + 
\sum_{k\in \Z^3\setminus \{0\}} \dot d_{I,k}   e^{ i k\cdot x}
\end{equation}
In particular, 
\[
\dot c_{I,0} =\langle \psi_I^2\rangle, \quad
\dot d_{I,0} = \langle \psi_I^3\rangle.
\]
Since $\psi_I$ is in $C^\infty(\T^3)$, we have  
\begin{equation}\label{est.k}
\sum_{k\in \Z^3} |k|^{n_0+2}|\dot{b_{I,k}}|
+\sum_{k\in \Z^3} |k|^{n_0+2}|\dot{c_{I,k}}|
+\sum_{k\in \Z^3} |k|^{n_0+2}|\dot{d_{I,k}}|
 \lec 1 , \quad 
\sum_{k\in\Z^3} |\dot{c}_{I,k}|^2\lec 1. 
\end{equation}
for $n_0 = \ceil{\frac{2b(2+\al)}{(b-1)(1-\al)}}$. Also,  it follows from $f_I \cdot \na \psi_I = f_I \cdot \na \psi_I^2=f_I \cdot \na \psi_I^3=0 $ that
\begin{align}\label{coe.van}
\dot b_{I,k} (f_I\cdot k )=\dot c_{I,k} (f_I\cdot k )=\dot d_{I,k} (f_I\cdot k )=0.
\end{align}
Next, as a consequence of \eqref{can.ph}, \eqref{can.R}, and \eqref{rep.psi}, we have 
\begin{align}
&n_o = \sum_{u} \sum_{k\in \Z^3\setminus \{0\}} \de_{q+1}^\frac 12b_{u,k} e^{i\la_{q+1} k\cdot \xi_I} \label{rep.W}\\
&n_o\otimes n_o
= \varrho^2(\de_{q+1}\I -R_\ell)
+\sum_{u} \sum_{k\in \Z^3\setminus \{0\}}  \de_{q+1} c_{u,k} e^{ i\la_{q+1} k\cdot \xi_I} \label{alg.eq}\\
&\frac12 |n_o|^2 n_o =-\varrho^3 \ph_\ell
+\frac12\sum_{u} \sum_{k\in \Z^3\setminus \{0\}}  \de_{q+1}^\frac 32 d_{u,k} e^{ i\la_{q+1} k\cdot \xi_I}\label{e:rep3}\\
&\frac12|n_o|^2
=-\varrho^2 \ka_\ell +\frac32 \varrho^2 \de_{q+1}  
+\frac 12\sum_{u} \sum_{k\in \Z^3\setminus \{0\}}  \de_{q+1} \tr(c_{u,k}) e^{ i\la_{q+1} k\cdot \xi_I}.\label{e:rep4}
\end{align}
where {$\ka_\ell = \frac 12 \tr(R_\ell)$} and the relevant coefficients are defined as follows:
\begin{equation}\begin{split}\label{coef.def}
&b_{u,k} = \sum_{I: u_I=u} \th_I\chi_I(\xi_I) \de_{q+1}^{-\frac 12}\ga_I \dot b_{I,k} \tilde f_I =: \sum_{I:u_I =u} B_{I,k} \tilde{f}_I,\\
&c_{u,k}
=
\sum_{I: u_I=u} \th_I^2\chi_I^2(\xi_I)
\de_{q+1}^{-1}\ga_I^2 \dot c_{I,k}
\tilde f_I\otimes \tilde f_I,\\
&d_{u,k}
= \sum_{I: u_I=u} \th_I^3\chi_I^3(\xi_I)
\de_{q+1}^{-\frac 32}\ga_I^3 \dot d_{I,k}
|\tilde f_I|^2 \tilde f_I.
\end{split}\end{equation}
Observe that,  by the choice of $\th_I$, if $|u-u'|> 1$, then
\begin{align*}
\supp_{t,x} (b_{u,k})\cap \supp_{t,x} (b_{u',k'})
&=\supp_{t,x} (c_{u,k})\cap \supp_{t,x} (c_{u',k'})\\
&=\supp_{t,x} (d_{u,k})\cap \supp_{t,x} (d_{u',k'})
 = \emptyset
\end{align*}
for any $k, k'\in \Z^3\setminus\{0\}$.  

We next prescribe an additional correction $n_c$ to make $n=n_o+n_c$ divergence-free. 
Since we have \eqref{coe.van} and the identity $\na \times (\na \xi_I^{\top}U(\xi_I)) = \operatorname{cof}(\na \xi_I^\top)(\na \times U)(\xi_I) =\det(\na \xi_I) \na \xi_I^{-1} (\na \times U)(\xi_I) $ for any smooth function $U$ (see for example \cite{DaSz2016}), we have 
\[
\frac 1{\la_{q+1}\det(\na \xi_I)}\na \times \left( \dot{b}_{I,k}\na\xi_I^{\top} \frac{ik \times f_I}{|k|^2}  e^{i\la_{q+1} k\cdot {\xi_I}}\right) =\dot{b}_{I,k} \na \xi_I^{-1}f_I e^{i\la_{q+1} k\cdot \xi_I}.
\]
Using this, the preponderant part $n_o$ of the momentum correction can be written as
\begin{align*}
n_o 
&= \sum_{\substack{u\in \Z\\k\in \Z^3\setminus \{0\}}}  \de_{q+1}^\frac 12 \sum_{I:u_I=u} B_{I,k} \na \xi_I^{-1}f_I e^{i\la_{q+1} k\cdot \xi_I}\\
& =  \sum_{\substack{u\in \Z\\k\in \Z^3\setminus \{0\}}}  \frac{ \de_{q+1}^\frac 12}{\la_{q+1}}\sum_{I:u_I=u}  
\frac{B_{I,k}}{\det(\na \xi_I)} \na \times\left( \na \xi_I^{\top}\frac{ik \times f_I}{|k|^2} e^{i\la_{q+1} k\cdot \xi_I}\right).
\end{align*}
Note that $\det(\na \xi_I)$ is away from $0$ on the support of $B_{I,k}$. 
Therefore, we define
\begin{align}\label{def.Wc}
n_c = \frac {\de_{q+1}^\frac 12 } {\la_{q+1}} \sum_{\substack{u\in \Z\\k\in \Z^3\setminus\{0\}}} \sum_{u_I =u} \na \left(\frac{B_{I,k}}{\det(\na \xi_I)}\right) \times \left(\na\xi_I^{\top} \frac{ik \times f_I}{|k|^2} \right) e^{i\la_{q+1} k\cdot \xi_I}
=: \frac {\de_{q+1}^\frac 12}{\la_{q+1}\mu_q}\sum_{u,k} e_{u,k} e^{i\la_{q+1} k\cdot \xi_I}
\end{align}
where
\begin{align}\label{def.e}
e_{u,k} 
= {\mu_q}\sum_{I:u_I=u} \na(\det(\na \xi_I^{-1})\th_I\chi_I(\xi_I) \de_{q+1}^{-\frac 12}\ga_I \dot b_{I,k}) \times  \left(\na\xi_I^{\top} \frac{ik \times f_I}{|k|^2} \right).
\end{align}
In this way, the final momentum correction $m_{q+1}-m_q =: n= n_o+n_c$ can be written as
\begin{align*}
n = \na\times \left( \frac {\de_{q+1}^\frac12}{\la_{q+1}}
\sum_{I,k}  
\det(\na \xi_I^{-1}) B_{I,k} \na\xi_I^{\top} \frac{ik \times f_I}{|k|^2}e^{i\la_{q+1} k\cdot \xi_I} \right),
\end{align*}
and hence it is divergence-free {and mean-zero}. For later use, we remark that if $|u-u'| {>} 1$, $\supp_{t,x} (e_{u,k})\cap \supp_{t,x} (e_{u',k'})
 = \emptyset$ holds
for any $k, k'\in \Z^3\setminus\{0\}$.
 Also, by its definition, the correction $n$ has the representation
\begin{align}\label{def.w}
n= \sum_{u\in \Z} \sum_{k\in \Z^3\setminus \{0\}} \de_{q+1}^\frac 12 (b_{u,k} + (\la_{q+1}\mu_q)^{-1} e_{u,k} )e^{i\la_{q+1} k\cdot \xi_I}.
\end{align}

\section{Definition of the new errors}

\subsection{New Reynolds stress}
With the correction $n$ of the momentum defined as in the previous section, we reorganize the Euler-Reynolds system and the relaxed energy equation as the equations for the new Reynolds stress $R_{q+1}$ and for the unsolved current $\ph_{q+1}$, respectively.
  
We first define $R_{q+1}$. Using the momentum equation 
\begin{align*}
    \pa_t m_q + \na\cdot\left(\frac{m_q\otimes m_q} \varrho\right) + \na p(\varrho)  = \div( \varrho (R_q-c_q\I))
\end{align*}
at $q$th step with $c_q= \sum_{j=q+1}^\infty \de_j$, we can write the equation for $R_{q+1}$ as 
\begin{align*}
\div( \varrho R_{q+1})
&=\underbrace{ \varrho D_{t,\ell} \frac{n}{ \varrho} - \div(m_q-m_\ell)\frac{n}{ \varrho}}_{=\na \cdot  ( \varrho R_T)} 
+ \underbrace{\div\left(\frac{n\otimes n}{ \varrho} +  \varrho R_\ell - \de_{q+1}  \varrho \I\right)}_{=:\na \cdot ( \varrho  R_O)} \\
&\quad+\underbrace{  (n\cdot \na) \frac{m_\ell}{ \varrho} }_{=:\na\cdot  ( \varrho R_N)} + \underbrace{\div \left(\frac{(m_q-m_\ell)\otimes n}{ \varrho} + \frac{n\otimes (m_q-m_\ell)}{ \varrho} +  \varrho(R_q-R_\ell)\right) }_{=:\na\cdot  ( \varrho R_M)},
\end{align*}
where $D_{t,\ell} = \left(\pa_t+\frac{m _\ell}{ \varrho} \cdot \na\right)$, and decompose $R_O$ further as 
\[
\na\cdot (\varrho R_O)=\underbrace{\div\left(\frac{n_o\otimes n_o}{ \varrho} +  \varrho R_\ell - \de_{q+1}  \varrho \I\right)}_{=\na \cdot (\varrho R_{O1})}
+ \underbrace{\div\left(\frac{n_o\otimes n_c}{ \varrho}+\frac{n_c\otimes n_o}{ \varrho}+\frac{n_c\otimes n_c}{ \varrho}\right) }_{=\na \cdot (\varrho R_{O2})}.
\]
Then, define $R_{q+1}$ as
\begin{equation}\label{e:splitting_of_Reynolds}
R_{q+1} =  R_T +  R_N +  R_{O1} + R_{O2}+ R_M  + \frac 23 \frac{\zeta(t)}{\varrho}\I. 
\end{equation}
Here the last term does not affect $\div(\varrho R_{q+1})$ because $\zeta$ is a function of time, which will be specified in Section \ref{sec:new.current}. Our choice of $R_{O2}$ and $ R_M$ are
\begin{equation}\begin{split}\label{def.RO2}
 \varrho R_{O2} =  \frac{n_o\otimes n_c}{ \varrho}+\frac{n_c\otimes n_o}{ \varrho}+\frac{n_c\otimes n_c}{ \varrho}.
\end{split}\end{equation}
and
\begin{equation}\begin{split}\label{def.RM}
 \varrho R_M =
\underbrace{ \varrho(R_q-R_\ell)}_{=\varrho R_{M1}}+\underbrace{\frac{(m_q-m_\ell)\otimes n}{ \varrho} + \frac{n\otimes (m_q-m_\ell)}{ \varrho}}_{=\varrho R_{M2}},
\end{split}\end{equation}
which are the only two Reynolds stress errors which might have nonzero trace. For the other errors, we solve the divergence equation by using the inverse divergence operator $\cal{R}$ in Definition \ref{idv.defn} to get trace-free errors, namely we set
\begin{align*}
 \varrho R_{O1} &= \mathcal{R} \left(\div\left(\frac{n_o\otimes n_o}{ \varrho} +  \varrho R_\ell - \de_{q+1}  \varrho \I\right) \right)\\
 \varrho R_{N} &= \mathcal{R} \left( (n\cdot \na) \frac{m_\ell}{ \varrho}\right)\\
 \varrho R_{T} &= \mathcal{R} \left( \varrho D_{t,\ell} \frac{n}{ \varrho} - \div(m_q-m_\ell)\frac{n}{ \varrho}\right)\, . 
\end{align*}
Here, we used that $\textstyle{\varrho D_{t,\ell} \frac{n}{ \varrho} - \div(m_q-m_\ell)\frac{n}{ \varrho} = \pa_t n + \div \left(\frac{n\otimes m_\ell}{\varrho}\right)}$ and $\textstyle{(n\cdot\na)\frac{m_\ell}{\varrho} = \div \left(\frac{m_\ell \otimes n}{\varrho} \right)} $ and hence, they have zero average. As a result, we have
\[
\tr ( \varrho R_{q+1}) = \tr ( \varrho R_{O2} +  \varrho R_{M}) + 2\zeta,
\]
which gives 
\begin{equation}\label{askR}
\kappa_{q+1} :=\frac{1}{2} \tr R_{q+1}= \frac 12  \tr( R_{O2} +  R_{M}) + \frac{\zeta}{\varrho} 
\end{equation}

\subsection{New current}\label{sec:new.current}
Applying the frequency cut-off $P_{\leq \ell^{-1}}$ to the Euler-Reynolds system, we have
\[
\pa_t m_\ell + \na \cdot \left(\frac{m_\ell\otimes m_\ell}{\varrho}\right) + \na p_\ell(\varrho)
= \na \cdot P_{\le \ell^{-1}} (\varrho R_q - \varrho c_q \I) +  Q(m_q,m_q).
\]
where $p_\ell(\varrho) = P_{\le \ell^{-1}} p(\varrho)$ and $Q(m_q, m_q)= Q_{\ell, \varrho}(m_q,m_q)$ is defined as
\begin{equation}\label{e:quadratic_form_Q}
Q(m_q,m_q) := \na \cdot \left(\frac{m_\ell \otimes m_\ell}{\varrho} - P_{\le \ell^{-1}}\left(\frac{m_q\otimes m_q}{\varrho}\right)\right)\, .
\end{equation}
Also, we recall that the tuple $(m_q,c_q, R_q,\ph_q)$ solves 
\begin{align*}
&\pa_t \left( \frac{|m_q|^2}{2 \varrho} + P( \varrho)\right) +
\div\left(\frac{m_q}{ \varrho} \left(
\frac{|m_q|^2}{2 \varrho} +  \varrho P'( \varrho)\right)\right)\\
&\hspace{4cm}=  \varrho \left(\pa_t + \frac{m_q}{ \varrho}\cdot\na\right) \ka_q + \div((R_q-c_q\I) m_q) + \div ( \varrho\ph_q) + \pa_t E
\end{align*}
with $\ka_q := \frac 12 \tr(R_q)$. 
Using these equations, we can write the equation for $\ph_{q+1}$ as 
\begin{align*}
\varrho D_{t,q+1} & \ka_{q+1} 
+ \div ( \varrho\ph_{q+1})\\
=&  
\underbrace{\varrho D_{t,q} \left(\frac {|n|^2}{2\varrho^2} + \ka_q+ \frac{(m_q-m_\ell)\cdot n)}{\varrho^2} \right)
}_{=: \varrho D_{t,q+1} \ka_{q+1}- (\zeta_1+ \zeta_2+\zeta_3+\zeta_4)' + \na \cdot (\varrho\ph_T)} 
+ \underbrace{\na\cdot \left(\frac {|n|^2n}{2\varrho^2} + \varrho\ph_\ell \right)}_{=\na\cdot(\varrho\ph_O)} \\
&  \underbrace{-\div(R_{q+1} n)}_{= \na\cdot(\varrho\ph_R)}   
+\underbrace{\frac n{\varrho}\cdot (\div P_{\le \ell^{-1}}( \varrho (R_q-c_q \I)) + Q(m_q,m_q)) }_{=\nabla \cdot (\varrho\ph_{H1}) + \zeta_1'} \\
&+ \underbrace{\na \cdot \left(
\frac{|m_q-m_\ell|^2}{2\varrho}
\frac n{\varrho}\right)+ \na \cdot (\varrho(\ph_q-\ph_\ell)) }_{=\na\cdot(\varrho\ph_{M1})}
+\underbrace{ \frac n{\varrho} \cdot \na (p(\varrho)-p_\ell(\varrho))
}_{=\na \cdot(\varrho\ph_{M4}) + \zeta_4'}
\\
&+ \underbrace{\na\cdot  \left(\left (\frac{n\otimes n}{\varrho}+\varrho R_q -\varrho R_{q+1} - \de_{q+1}\varrho \I \right) \frac{m_q-m_\ell}{\varrho}\right) + \div (m_q-m_\ell) \frac{n\cdot m_\ell}{\varrho^2}} _{=\na\cdot(\varrho\ph_{M2})+\na\cdot(\varrho\ph_{M3})+ \zeta_2'}
\\
&+ \underbrace{\left(\frac{n\otimes n}{\varrho}-\de_{q+1}\varrho\I+\varrho R_q-\varrho R_{q+1} + \frac{(m_q-m_\ell)\otimes n}{\varrho}+ \frac{n\otimes (m_q-m_\ell)}{\varrho} \right): \na \frac{m_\ell}{\varrho}}_{=   \na\cdot(\varrho\ph_{H2})+  \zeta_3'}\\
\end{align*}
where $D_{t,q} = \left(\pa_t + \frac{m_{q}}{ \varrho}\cdot\na\right)$.
The functions {$ \zeta_1$}, $ \zeta_2$, $ \zeta_3$, and $ \zeta_4$ will be defined to invert the divergence. Then, we set
\begin{align*}
\varrho\varphi_O &:= 
\underbrace{\mathcal{R}\left(\na\cdot \left(\frac {|n_o|^2n_o}{2\varrho^2} +\varrho \ph_\ell \right)\right)}_{=: \varrho\varphi_{O1}} 
+ \underbrace{\frac {|n|^2n - |n_o|^2 n_o}{2\varrho^2}  }_{=: \varrho\varphi_{O2}} \\
\varrho\varphi_R &:= -R_{q+1}n\\
\varrho\varphi_{M1} &:=  \frac{|m_q-m_\ell|^2}{2\varrho}\frac n{\varrho}+ \varrho(\ph_q-\ph_\ell)\\
\varrho\varphi_{M2} &:= \left (\frac{n\otimes n}{\varrho}+\varrho R_q -\varrho R_{q+1} - \de_{q+1}\varrho \I \right) \frac{m_q-m_\ell}{\varrho}.
\end{align*}

Next, recall that the definition of $\ka_q$, $R_{O2}$, $R_M$ and $\ka_{q+1} = \frac12 \tr (R_{O2} + R_M) + \frac{\zeta}{\varrho}$ to get
\begin{equation}\begin{split}\label{reo.tra}
\frac {|n|^2}{2\varrho^2} + \ka_q + \frac{(m_q-m_\ell) \cdot n}{\varrho^2}
&= \frac 12\tr\left(  \frac{n_o\otimes n_o}{\varrho^2} - \de_{q+1} \I + R_\ell \right) + \frac 32 \de_{q+1}    
+ \frac 12 \tr( R_M + R_{O2}) 
\\
&=\frac 32 \de_{q+1} + \ka_{q+1} -\frac{\zeta}{\varrho} +\frac 12\tr\left(  \frac{n_o\otimes n_o}{\varrho^2}- \de_{q+1} \I + R_\ell\right).
\end{split}\end{equation}
Set $\zeta = \zeta_0 + \zeta_1 +\zeta_2 + \zeta_3 + \zeta_4$, where $\zeta_i$ will be determined below. \eqref{reo.tra} then gives the equation for $  \ph_T$; since we have $\varrho D_{t,q} (\zeta/\varrho) = \zeta' + \div((m_q \zeta)/\varrho)$, 
\begin{align*}
\nabla \cdot   (\varrho\ph_T) + \zeta _0' &= 
 \div \left(-\ka_{q+1}  n+ \frac12 \tr\left(  \frac{n_o\otimes n_o}{\varrho^2} - \de_{q+1} \I + R_\ell \right)(m_q-m_\ell) - \frac{m_q\zeta}{\varrho}
\right)\\
& +\frac{\varrho}2 D_{t,\ell}\tr\left(  \frac{n_o\otimes n_o}{\varrho^2} - \de_{q+1} \I + R_\ell \right) 
- \frac 12\tr\left(  \frac{n_o\otimes n_o}{\varrho^2} - \de_{q+1} \I + R_\ell \right)\div(m_q-m_\ell)
\end{align*}
Now, we define $ \zeta_0$ to make the divergence equation solvable and set $  \ph_T =   \ph_{T1} +   \ph_{T2}$ as 
\begin{align}
  \varrho\ph_{T1} &= -\ka_{q+1}  n+ \frac12 \tr\left(  \frac{n_o\otimes n_o}{\varrho^2} - \de_{q+1} \I + R_\ell \right)(m_q-m_\ell) - \frac{m_q\zeta}{\varrho}\label{e:phi_T1}\\
 \zeta_{0} (t) &= \int_0^t \Big\langle \frac{\varrho}2 D_{t,\ell}\tr\left(  \frac{n_o\otimes n_o}{\varrho^2}- \de_{q+1} \I + R_\ell \right)   \Big\rangle (s)\, ds \nonumber\\
&\quad - \int_0^t   \Big\langle\frac 12\tr\left(  \frac{n_o\otimes n_o}{\varrho^2} - \de_{q+1} \I + R_\ell \right)\div(m_q-m_\ell)  \Big\rangle (s)\, ds \nonumber\\
  \varrho\ph_{T2} &= \mathcal{R} \left(\frac{\varrho}2 D_{t,\ell}\tr\left(  \frac{n_o\otimes n_o}{\varrho^2} - \de_{q+1} \I + R_\ell  \right) 
\right) \nonumber\\
&\quad -
\mathcal{R}\left( \frac 12\tr\left(  \frac{n_o\otimes n_o}{\varrho^2} - \de_{q+1} \I + R_\ell \right)\div(m_q-m_\ell) \right).\label{e:phi_T2}
\end{align}
Here, we remark that, by the definition of $\mathcal{R}$, we have
\begin{equation}\label{e:average-free}
 \mathcal{R} (g (t, \cdot)) := \mathcal{R} (g (t, \cdot)- h(t))
\end{equation}
for every smooth periodic time-dependent vector field $g$ and for every $h$ which depends only on time.

In a similar way, we let
\begin{align}
 \zeta_1 (t) &:= \int_0^t \Big \langle \frac n{\varrho}\cdot (\div P_{\le \ell^{-1}}(\varrho(R_q-c_q \I)) + Q(m_q,m_q)) \Big\rangle (s)\, ds \nonumber\\
 \zeta_2 (t) &:= \int_0^t \Big\langle\div (m_q-m_\ell) \frac{n\cdot m_\ell}{\varrho^2}\Big\rangle (s)\, ds \nonumber\\
 \zeta_{3} (t) &:= \int_0^t \Big\langle\left(\frac{n\otimes n}{\varrho}-\de_{q+1}\varrho\I+\varrho R_q-\left(\varrho R_{q+1} -\frac23 \zeta\I\right)\right) : \na \frac{m_\ell}{\varrho}\Big\rangle (s)\, ds
 \label{e:define_varrho_2}
 \\
&\quad+ \int_0^t \Big\langle 
 \left(\frac{(m_q-m_\ell)\otimes n}{\varrho}+ \frac{n\otimes (m_q-m_\ell)}{\varrho} \right): \na \frac{m_\ell}{\varrho}\Big\rangle (s)\, ds \nonumber
 \\
 \zeta_4 (t) &:= \int_0^t \Big \langle \frac{n}{\varrho} \cdot \na(p(\varrho) - p_\ell (\varrho)) \Big\rangle (s)\, ds \nonumber
\end{align}
and
\begin{align*}
  \varrho\ph_{H1} &:= \mathcal{R} \left(\frac n{\varrho}\cdot (\div P_{\le \ell^{-1}}(\varrho(R_q-c_q \I)) + Q(m_q,m_q))\right)\\
  \varrho\ph_{M3} &:= \mathcal{R} \left( 
  \div (m_q-m_\ell) \frac{n\cdot m_\ell}{\varrho^2} \right)\\
  \varrho\ph_{M4} &:= \mathcal{R} \left( \frac{n}{\varrho} \cdot \na(p(\varrho) - p_\ell (\varrho)) \right)\\
  \varrho \ph_{H2} &:=\mathcal{R} \left(\left(\frac{n\otimes n}{\varrho}-\de_{q+1}\varrho\I+\varrho R_q-\left(\varrho R_{q+1}-\frac23\zeta \I\right)\right) : \na \frac{m_\ell}{\varrho} \right)\\
  &\quad + \mathcal{R} \left(\left(\frac{(m_q-m_\ell)\otimes n}{\varrho}+ \frac{n\otimes (m_q-m_\ell)}{\varrho} \right): \na \frac{m_\ell}{\varrho}  \right) -  \frac{2m_\ell\zeta}{3\varrho}  \, .
\end{align*}
Here, $\zeta_{3}$ is well-defined because $\varrho R_{q+1} - (2\zeta)/3 \I $ is independent of $\zeta$ (see \eqref{e:splitting_of_Reynolds}). 

\section{Preliminary estimates}\label{s:estimates_mollification}

We now start detailing the estimates which will lead to the proof of the inductive propositions. In this section, we set $\norm{\cdot}_N = \norm{\cdot}_{C^0([0,T]+\tau_q; C^N(\T^3))}$.

\subsection{Regularization} First of all we address a series of a-priori estimates on the regularized tuple and on their differences with the original one.
By its construction, we can easily see that
\begin{align*}
\hspace{2cm}&\norm{m_\ell}_N \lec_N \ell^{1-N} \la_q \de_q^\frac12 \le \ell^{-N} \de_{q}^\frac12, 
\hspace{2.7cm} \forall N\geq 1,\\
\norm{D_{t,\ell}^s R_\ell}_0 &\lec_{s} \ell_t^{-s}\la_q^{-3\ga} \de_{q+1}, \quad
\norm{D_{t,\ell}^s \ph_\ell}_0 \lec_{s} \ell_t^{-s}  \la_q^{-3\ga} \de_{q+1}^\frac32, \quad \forall s\geq 0. 
\end{align*}
{Also, there exists $\bar{b}(\al)>1$ such that
for any $b\in (1,\bar b(\al))$ we can find $\La_0=\La_0(\al,b,M, \varrho)$ with the following property: if $\la_0\geq \La_0$, then $|\na^{N+1} \Phi(t+\tau,x;t)|\lec_{M, \varrho}  \ell^{-N}$ holds for $N\geq 0$ and $\tau\in [-\ell_t,\ell_t]$. This implies}
\begin{align}
\ell_t^s\norm{D_{t,\ell}^s R_\ell}_N {\lec_{s,N,M, \varrho}} \ \ell^{-N} \la_q^{-3\ga} \de_{q+1}
\label{est.mR}
\\
\ell_t^s\norm{D_{t,\ell}^s \ph_\ell}_N 
{\lec_{s,N,M,\varrho}} \ \ell^{-N} \la_q^{-3\ga} \de_{q+1}^\frac32
\label{est.mph}.
\end{align}
The estimate can be obtained using \eqref{char.mol.traj} (see also
\cite[Section18]{Is2013}.)

On the other hand, the differences between the regularized objects $(v_\ell, p_\ell, R_\ell, \ph_\ell)$ and their original counterparts satisfy the following estimates.
 
\begin{lem}{There exists $\bar{b}(\al)>1$ such that
for any $b\in (1,\bar b(\al))$ we can find $\La_0(\al,b,M)$ with the following property. If $\la_0\geq \La_0$}
and $N\in \{0,1,2\}$ (recall that we follow the notational convention explained in Remark \ref{r:estimates_material_derivative}) then:
\begin{align}
&\norm{m_q-m_\ell}_N 
+ \de_q^{-\frac12}\norm{D_{t,\ell}(m_q-m_\ell)}_{N-1}  
\lec \ell^{2-N} \la_q^2 \de_q^\frac12, \label{est.v.dif} \\
&\norm{R_q-R_\ell}_N+ 
\de_{q+1}^{-\frac12}\norm{D_{t,\ell} (R_q-R_\ell)}_{N-1} 
\lec
\la_{q+1}^N {\la_q^\frac12}{\la_{q+1}^{-\frac12}}
\de_q^\frac14 \de_{q+1}^\frac34
\label{est.R.dif} \\
&\norm{\ph_q-\ph_\ell}_N+\de_{q+1}^{-\frac12}\norm{D_{t,\ell} (\ph_q-\ph_\ell)}_{N-1}
 \lec  
\la_{q+1}^N{\la_q^\frac12}{\la_{q+1}^{-\frac12}}\de_q^\frac14 \de_{q+1}^\frac54.\label{est.ph.dif} 
\end{align}
{Here, we allow the implicit constants to be depending on $M$, $ \varrho$ and $p$.}
\end{lem}

\begin{proof} Set $P_{>\ell^{-1}}F := F- P_{\leq \ell^{-1}} F = P_{>2^J}F$, where $J\in \N\cup\{0\}$ is the largest number such that $2^J\le \ell^{-1}$. By Bernstein's inequality, we have $\norm{F-P_{\le \ell^{-1}}F}_0 = \norm{P_{> \ell^{-1}}F}_0 \lec \ell^{j}\norm{\na^j F}_0$ for any $F\in C^j(\T^3)$. Using \eqref{est.vp} we then get for $N\leq2$
\begin{align}\label{diff.m}
&\norm{m_q-m_\ell}_N 
\lec \ell ^{2-N}\norm{\na^2 m_q}_0
\lec {\ell^{2-N}}\la_q^2 \de_q^\frac12, 
\end{align}
Also, we have
\begin{align*}
(\varrho_{\ell_t}   \ast_{\Phi} F-F)(t,x)
&=  \int_{\R} (F(t+s,\Phi(t+s,x;t)) - F(t,x))  \varrho_{\ell_t}(s) ds\\
&= \int_{\R} \int_0^s D_{t,\ell} F(t+\tau, \Phi(t+ \tau,x;t)) d\tau   \varrho_{\ell_t}(s) ds, 
\end{align*}
from which we conclude $ \norm{F-  \varrho_{\ell_t}   \ast_{\Phi} F}_{C^0([a,b]\times \T^3)} \lec \ell_t \norm{D_{t,\ell} F}_{C^0([a,b]+\ell_t\times \T^3)}$ because of $\supp( \varrho_{\ell_t}) \subset (-\ell_t,\ell_t)$. 
In addition, we have the following decomposition,
\begin{align}
F- \varrho_{\ell_t}   \ast_\Phi P_{\le \ell^{-1}} F &= (F-P_{\le \ell^{-1}}F) + (P_{\le \ell^{-1}}F- \varrho_{\ell_t}   \ast_\Phi P_{\le \ell^{-1}} F),  \label{rep.dif}\\
D_{t,\ell}P_{\le \ell^{-1}} F 
&= P_{\le \ell^{-1}}D_{t,\ell} F + \left[\frac{m_\ell}{\varrho}\cdot \na, P_{\le \ell^{-1}}\right]F\, ,
\label{rep.DtF}
\end{align}
where as usual $[A,B]$ denotes the commutator $AB - BA$ of the two operators $A$ and $B$.
Note that $D_{t,\ell} F$ can be further decomposed as  $D_{t,\ell} F = D_{t,q} F + \frac{(m_q-m_\ell)}{\varrho}\cdot \na F$. Then, using \eqref{est.R}, \eqref{est.ph}, \eqref{def.me}, and \eqref{est.com1}, we obtain
\begin{align*}
\norm{R_q-R_\ell}_0 
&\lec \norm{P_{> \ell^{-1}} R_q}_0
+ \ell_t\norm{D_{t,\ell}  P_{\le \ell^{-1}}R_q}_{C(\mathcal{I}^q \times \T^3)}\\
&\lec \ell^2\norm{R_q}_2 +\ell_t(  \norm{D_{t,\ell}  R_q}_{C(\mathcal{I}^q \times \T^3)} + \ell \norm{\na \frac{m_\ell}{\varrho}}_{C(\mathcal{I}^q \times \T^3)}\norm{\na R_q}_{C(\mathcal{I}^q \times \T^3)})\\
&\lec  ( (\ell\la_q)^2 + \ell_t \la_q\de_q^\frac12)\la_q^{-3\ga} \de_{q+1}
\lec  {\la_q^\frac12}{\la_{q+1}^{-\frac12}}\de_q^\frac14 \de_{q+1}^\frac34,
\end{align*}
and
\begin{align*}
\norm{\ph_q-\ph_\ell}_0
&\lec
\norm{P_{> \ell^{-1}} \ph_q}_0 + \ell_t \norm{D_{t,\ell}P_{\le \ell^{-1}}\ph_q}_{C(\mathcal{I}^q \times \T^3)} \\
&\lec   \ell^2\norm{\ph_q}_2 
+ \ell_t(\norm{D_{t,\ell}\ph_q}_{C(\mathcal{I}^q \times \T^3)} + \ell^{1} \norm{\na \frac{m_\ell}{\varrho}}_{C(\mathcal{I}^q \times \T^3)} \norm{\na \ph_q}_{C(\mathcal{I}^q \times \T^3)}) \\
&\lec 
((\ell\la_q)^2 + \ell_t \la_q\de_q^\frac12)
\la_q^{-3\ga} \de_{q+1}^\frac32
\lec  {\la_q^\frac12}{\la_{q+1}^{-\frac12}}\de_q^\frac14 \de_{q+1}^\frac54,
\end{align*}
where $\mathcal{I}^q = [0,T]+\tau_{q-1}$.
Furthermore, we have 
for $N=1,2$
\begin{align*}
\norm{R_q-R_\ell}_N 
&\lec \norm{R_q}_N + \norm{R_\ell}_N 
\lec  {\la_q^N} \la_q^{-3\ga} \de_{q+1}
\lec \la_{q+1}^N {\la_q^\frac12}{\la_{q+1}^{-\frac12}}\de_q^\frac14 \de_{q+1}^\frac34\\
\norm{\ph_q-\ph_\ell}_N
&\lec
+ {\norm{\ph_q}_N + \norm{\ph_\ell}_N} 
\lec {\la_q^N}
\la_q^{-3\ga} \de_{q+1}^\frac32
\lec  \la_{q+1}^N {\la_q^\frac12}{\la_{q+1}^{-\frac12}}\de_q^\frac14 \de_{q+1}^\frac54.
\end{align*}

Now, we consider the advective derivatives. We remark that for $F_\ell= P_{\le \ell^{-1}} F$, we can write
\begin{align*}
D_{t,\ell} (F-F_\ell) 
= D_{t,\ell} P_{> \ell^{-1}}F 
=  P_{> \ell^{-1}} D_{t,\ell} F + \left[\frac{m_\ell}{\varrho} \cdot \na,  P_{> \ell^{-1}}\right] F.
\end{align*}
Then, we apply this to $F=m$ and $F=p(\varrho)$ and use \eqref{diff.m} and \eqref{est.com3} to obtain
\begin{align}
\norm{D_{t,\ell} (m_q-m_\ell)}_{N-1}
&\lec \norm{P_{>\ell^{-1}}D_{t,\ell} m_q}_{N-1} + \norm{[(m_\ell/\varrho) \cdot \na,  P_{> \ell^{-1}}] m_q}_{N-1}
\label{est.adv.m}
\\
&\lec \norm{P_{>\ell^{-1}}D_{t,q} m_q}_{N-1} + \ell \norm{((m_q-m_\ell)/\varrho \cdot \na)m_q}_N \nonumber\\
&\quad + \Norm{[m_\ell/\varrho \cdot \na,  P_{> \ell^{-1}}]m_q}_{N-1}
\nonumber\\
&\lec \ell^{2-N} (\la_q \de_q^\frac 12)^2 \nonumber. 
\end{align}
Here, we obtain the estimate for $\norm{P_{>\ell^{-1}} D_{t,q}m_q}_{N-1}$ from the relaxed momentum equation:
\begin{align*}
    \norm{P_{>\ell^{-1}} D_{t,q} m_q}_{N-1}
    &\leq \norm{P_{>\ell^{-1}} p(\varrho)}_{N}
    +\norm{P_{>\ell^{-1}}(\varrho R_q)}_{N}
    +\norm{P_{>\ell^{-1}}  \varrho}_{N}
    +\norm{P_{>\ell^{-1}} (\div(m_q/\varrho) m_q)}_{N-1}\\
    &\lec \ell^2 (\norm{p(\varrho)}_{N+2}+\norm{\varrho}_{N+2}) + \ell^{2-N}\norm{\varrho R_q}_2
    +\ell \norm{m_q (-\pa_t \varrho/\varrho + m_q \cdot \na \varrho^{-1}) }_N\\
    &\lec \ell^{2-N}(\la_q\de_q^\frac12)^2
\end{align*}

In a similar way, we have
\begin{equation}\begin{split}\label{est.hDt}
\norm{D_{t,\ell} P_{> \ell^{-1}}R_q}_{N-1}
\lec \norm{D_{t,\ell} R_q}_{N-1} + \norm{[(m_\ell/\varrho) \cdot \na,  P_{> \ell^{-1}}] R_q}_{N-1}
\lec {\la_{q+1}^{N-1} \la_q} \de_q^\frac12 \la_q^{-3\ga} \de_{q+1}\\
\norm{D_{t,\ell} P_{> \ell^{-1}}\ph}_{N-1}
\lec \norm{D_{t,\ell} \ph_q}_{N-1} 
+ \norm{[(m_\ell/\varrho) \cdot \na,  P_{> \ell^{-1}}] \ph_q}_{N-1} 
\lec {\la_{q+1}^{N-1} \la_q} \de_q^\frac12 \la_q^{-3\ga} \de_{q+1}^\frac32.
\end{split}\end{equation}
Simply applying the triangle inequality, it can be easily shown that
\begin{equation}\begin{split}\label{est.Dtdif}
\norm{D_{t,\ell}(P_{\le \ell^{-1}} R_q - \varrho_{\ell_t}  \ast_{\Phi} P_{\le \ell^{-1}} R_q)}_{N-1}
\le 2\norm{D_{t,\ell} P_{\le \ell^{-1}} R_q}_{N-1} 
\lec \la_{q+1}^{N-1} \la_q\de_q^\frac12 \la_q^{-3\ga}\de_{q+1}\\
\norm{D_{t,\ell}(P_{\le \ell^{-1}} \ph_q - \varrho_{\ell_t}  \ast_{\Phi} P_{\le \ell^{-1}} \ph_q)}_{N-1}
\le 2\norm{D_{t,\ell} P_{\le \ell^{-1}} \ph_q}_{N-1} 
\lec \la_{q+1}^{N-1} \la_q\de_q^\frac12 \la_q^{-3\ga}\de_{q+1}^\frac32.
\end{split}\end{equation}
Combining \eqref{rep.dif}, \eqref{est.hDt}, and \eqref{est.Dtdif}, it follows that
\begin{align*}
\norm{D_{t,\ell}(R_q-R _\ell)}_{N-1} &\lec  \la_{q+1}^N\de_{q+1}^\frac12\cdot {\la_q^\frac12}{\la_{q+1}^{-\frac12}}\de_q^\frac14 \de_{q+1}^\frac34\\
\norm{D_{t,\ell}(\ph_q-\ph _\ell)}_{N-1} &\lec  \la_{q+1}^N\de_{q+1}^\frac12 \cdot{\la_q^\frac12}{\la_{q+1}^{-\frac12}}\de_q^\frac14 \de_{q+1}^\frac54.
\end{align*}
\end{proof}
 
\subsection{Quadratic commutator} We next deal with a quadratic commutator estimate, which is a version of the estimate in \cite{CoETi1994} leading to the proof of the positive part of the Onsager conjecture for the incompressible Euler equations. In the compressible case, the situation is more complicated due to the presence of the density $\varrho$. Indeed, it leads to additional commutator terms which need to be estimated. Another difference from the incompressible case found in \cite{DLK20} is that we can do the estimate for a fixed, finite number of derivatives but cannot estimate all derivatives. This is because $\varrho$ is just a smooth function and thus only a fixed number of its derivatives can be controlled.

\begin{lem}\label{lem:est.Qvv} For any integer $N\in [0, \overline N] $, $\overline N \in \N$ independent of $q$, $Q_{\varrho}(m_q,m_q)
$ defined as in \eqref{e:quadratic_form_Q} satisfies 
\begin{align}\label{est.Qvv}
\norm{Q(m_q,m_q)}_N\lec \ell^{1-N} (\la_q\de_q^\frac12)^2, \quad
\norm{D_{t,\ell} Q(m_q,m_q)}_{N} \lec \ell^{-N} \de_q^{\frac12}(\la_q\de_q^\frac12)^2. 
\end{align}
{Here, we allow the implicit constants to be depending on $M$, $\varrho$, {$p$} and $\overline N$.}
\end{lem}

\begin{proof} For the convenience, we drop the index $\varrho$ from $Q_{\varrho}$ and $q$ from $m_q$. 
We write $Q(m,m)$ as follows
\begin{align*}
Q(m,m) &= \na \cdot \left(\frac{m_\ell \otimes m_\ell}{\varrho} - \frac{P_{\le \ell^{-1}} (m\otimes m)}{\varrho}\right) 
+ \na \cdot \left(\frac{P_{\le \ell^{-1}} (m\otimes m)}{\varrho} - P_{\le \ell^{-1}}\left(\frac{m\otimes m}{\varrho}\right)\right)\\
&=: Q_1 + Q_2
\end{align*}
Writing $Q_1 $ as
\begin{align*}
    Q_1 &= \underbrace{{\varrho}^{-1} \na \cdot \left(m_\ell\otimes m_\ell - P_{\le \ell^{-1}} (m\otimes m)\right)}_{=:Q_{11}} + \underbrace{\left(m_\ell\otimes m_\ell - P_{\le \ell^{-1}} (m\otimes m)\right) :\na {\varrho}^{-1}}_{=:Q_{12}}
\end{align*}
and using \eqref{est.com} on setting $f=g=m$, we have
\begin{align*}
\norm{Q_{1}}_N \leq \norm{Q_{11}}_N + \norm{Q_{12}}_N\lec 
\norm{m_\ell\otimes m_\ell - P_{\le \ell^{-1}} (m\otimes m)}_{N+1}
\lec \ell^{1-N}\norm{m}_1^2 \lec \ell^{1-N}(\la_q\de_q^\frac12)^2. 
\end{align*}
On the other hand, $Q_2$ can be written as
\begin{align*}
Q_2 &= \frac{P_{\le \ell^{-1}} \div (m\otimes m)}{\varrho} - P_{\le \ell^{-1}}\left(\varrho^{-1}\div (m\otimes m)\right) \\
&\quad + P_{\le \ell^{-1}} (m \otimes m) : \na \varrho^{-1} - P_{\le \ell^{-1}}( m \otimes m :\na \varrho^{-1}).
\end{align*}
Here, $\na \varrho^{-1}$ means that $\na (1/\varrho)$ rather than $(\na \varrho)^{-1}$. 
Since both the first two and last two terms are of the form $(P_{\le \ell^{-1}}f) g - P_{\le \ell^{-1}}(fg)$, we can apply \eqref{est.com0} to get that 
\begin{align*}
\norm{Q_2}_N \lesssim \ell^{1-N} \la_q \de_q^{\frac12} + \ell^{1-N}\lec \ell^{1-N} (\la_q \de_q^{\frac12})^2.
\end{align*}

Now, we consider the advective derivative. We first show $\norm{D_{t,\ell}(\varrho Q_{11})}_N \lec \ell^{-N} \de_q^\frac12(\la_q\de_q^\frac12)^2$. Then, the desired estimate $\norm{D_{t,\ell}Q_{11}}_N \lec \ell^{-N} \de_q^\frac12(\la_q\de_q^\frac12)^2$ easily follows because
\begin{align*}
\norm{D_{t,\ell}Q_{11}}_{N} \lesssim \norm{D_{t,\ell} (\varrho Q_{11})}_{N} + \sum_{N_1+N_2=N} \norm{D_{t,\ell} \varrho}_{N_1} \norm{Q_{11}}_{N_2}.
\end{align*}
We first note that using the equation, obtained by taking mollification $P_{\leq \ell^{-1}}$ to the relaxed momentum equation,
\begin{align*}
    D_{t,\ell} m_\ell 
    &=- \div (m_\ell/\varrho) m_\ell - \na p_\ell(\varrho) + P_{\leq \ell^{-1}}
    \div(\varrho R_q - c_q\varrho\I) + Q(m, m) \\
    &= (m_\ell P_{\leq \ell^{-1}}\pa_t \varrho)/\varrho - m_\ell (m_\ell\cdot\na)\varrho^{-1} - \na p_\ell(\varrho) + P_{\leq \ell^{-1}}
    \div(\varrho R_q - c_q\varrho\I) + Q(m, m),
\end{align*}
we get 
\begin{align}\label{adv.ml0}
    \norm{D_{t,\ell}m_\ell}_0
    \lec_{\varrho, {p}} \norm{m}_0 + \norm{m}_0^2 + 1 + \norm{ R_q}_1 + \norm{Q(m,m)}_0
    \lec \la_q\de_q
\end{align}
and for $N\geq 1$,
\begin{equation}\begin{split}\label{adv.mlN}
    \norm{D_{t,\ell}m_\ell}_N
    &\lec \ell^{1-N}(\norm{m_\ell P_{\leq\ell^{-1}}\pa_t \varrho}_{1}\norm{\varrho^{-1}}_N
    + \norm{m_\ell \otimes m_\ell}_{1}\norm{\na \varrho^{-1}}_N
    + \norm{p_\ell(\varrho)}_2+\norm{\varrho}_2 )\\
    &\quad+ \ell^{1-N}\norm{R_q}_2 + 
    \norm{Q(m,m)}_{N}\\
    &\lec \ell^{1-N}(\la_q\de_q^\frac12)^2.  \end{split}
\end{equation}
As a consequence, we have
\begin{align}\label{est.Dtvl}
\norm{D_{t,\ell} \na m _\ell }_N \leq
\norm{\na D_{t,\ell}m_\ell}_N + \norm{(\na (m_\ell/\varrho)\cdot \na) m_\ell}_N
\lec\ell^{-N} (\la_q\de_q^\frac12)^2.
\end{align} 
Also, similar to \eqref{est.adv.m}, one can obtain
\begin{align}\label{adv.m}
    \norm{D_{t,\ell} m}_0
    \lec \la_q\de_q.
\end{align}
Since $\varrho Q_{11}$ can be decomposed into
\begin{align}\label{eqn.Q11}
\varrho Q_{11}
=  (m _\ell-m)\cdot \na m _\ell + [m\cdot \na, P_{\le \ell^{-1}}] m + m_\ell \div m_\ell - P_{\le \ell^{-1}} ( m \div m)
\end{align}
their advective derivative can be estimated as follows; using the density equation,
\begin{align*}
&\norm{D_{t,\ell} (m_\ell \div m_\ell - P_{\le \ell^{-1}} ( m \div m))}_{N} = \norm{D_{t,\ell} (m_\ell (\pa_t \varrho)_\ell) - D_{t,\ell}P_{\le \ell^{-1}} ( m \pa_t \varrho)}_{N}\\
&\le \norm{(D_{t,\ell} m_\ell) (\pa_t \varrho)_\ell + m_\ell (D_{t,\ell}(\pa_t \varrho)_\ell) - P_{\le \ell^{-1}} D_{t,\ell} ( m\pa_t \varrho)}_{N} + \norm{[\frac{m_\ell}{\varrho}, P_{\le \ell^{-1}}] ( m \pa_t \varrho)}_{N}\\
&\lec \norm{D_{t,\ell} m_\ell}_N + \sum_{N_1+N_2 = N} \norm{m_\ell}_{N_1}\norm{D_{t,\ell}(\pa_t\varrho)_\ell}_{N_2} + \ell^{-N}\norm{D_{t,\ell}(m\pa_t \varrho)}_0\\
&\quad + \norm{[{m_\ell}P_{\leq \ell^{-1}}{\varrho}^{-1}, P_{\le \ell^{-1}}] ( m \pa_t \varrho)}_{N}
+\norm{[{m_\ell}P_{>\ell^{-1}}{\varrho}^{-1}, P_{\le \ell^{-1}}] ( m \pa_t \varrho)}_{N}\\
&\lec \ell^{1-N}(\la_q\de_q^\frac12)^2
+ \ell^{-N}\de_q^\frac12 + \ell^{-N}\la_q\de_q + \ell^{1-N}(\la_q\de_q^\frac12)^2 + \ell^{1-N}
\la_q\de_q^\frac12
\lec \ell^{-N}\de_q^\frac12(\la_q\de_q^\frac12)^2.
\end{align*}
Here, we denoted $P_{\ell^{-1}} \pa_t \varrho$ by $(\pa_t \varrho)_\ell$ and used \eqref{adv.mlN}, \eqref{est.vp}, \eqref{adv.m}, \eqref{est.com1}, and \eqref{est.com3}.

To estimate the first two terms in \eqref{eqn.Q11}, we recall that $\widehat{P_{\le \ell^{-1}} f}(\xi) =\widehat{P_{\le 2^J} f}(\xi) =  \phi \left(\frac{\xi}{2^J}\right) \widehat{f}(\xi)$ for some radial function $\phi \in \cal{S}$, where $J\in \N\cup\{0\}$ is the maximum number  satisfying $2^J\le \ell^{-1}$. For the convenience, we set $\widecheck{\phi}_{\ell}(x) = 2^{3J}\widecheck{\phi}(2^Jx)$. Then, by Poison summation formula, $P_{\le \ell^{-1}} f(x) = \int_{\R^3} f(x-y) \widecheck{\phi} _\ell(y)dy$ holds.  
Using this, the advective derivative of the commutator term can be written as follows,
\begin{align*}
D_{t,\ell}[m\cdot \na, P_{\le \ell^{-1}}] m
&=  (\pa_t + \frac{m_\ell}{\varrho}(x)\cdot \na) \int ((m(x)-m(x-y))\cdot \na)m(x-y) \widecheck{\phi} _\ell(y) dy\\
&=  \int  ((D_{t,\ell}m(x)-D_{t,\ell}m(x-y)) \cdot  \na )m(x-y) \widecheck{\phi} _\ell(y) dy\\
&\quad-\int (\frac{m _\ell}{\varrho}(x)-\frac{m _\ell}{\varrho}(x-y))_a\na_a  m_b(x-y)  \na_b m(x-y) \widecheck{\phi} _\ell(y) dy\\
&\quad+ \int ((m(x)-m(x-y))\cdot D_{t,\ell}\na) m(x-y) \widecheck{\phi} _\ell(y) dy\\
&\quad + \int (m(x)-m(x-y))_a (\frac{m _\ell}{\varrho}(x)-\frac{m _\ell}{\varrho}(x-y))_b (\pa_{ab}m)(x-y) \widecheck{\phi} _\ell(y) dy. 
\end{align*} 
Based on the decomposition, we use \eqref{est.vp}, \eqref{est.v.dif}, and $\norm{|y|^n\widecheck{\phi} _\ell}_{L^1(\R^3)}\lec \ell^n$, $n\geq 0$, to get
\begin{align*}
\norm{D_{t,\ell}(\varrho Q_{11})}_0
&\lec \norm{D_{t,\ell}(m-m _\ell)}_0  \norm{\na m_\ell}_0 + \norm{m-m _\ell}_0\norm{D_{t,\ell}\na m _\ell}_0 
+ \ell\norm{\na D_{t,\ell} m}_0\norm{\na m}_0
+\ell\norm{m}_1^3 \\
&\quad+ \ell \norm{\na m}_0 \norm{D_{t,\ell} \na m}_0 + \ell^2 \norm{ m}_1^2 \norm{\na^2 m}_0 + \de_q^\frac12(\la_q\de_q^\frac12)^2
\lec \de_q^\frac12(\la_q\de_q^\frac12)^2
\end{align*}
where we used \eqref{adv.mlN}. 
In the case of $N\geq 1$, we write
\begin{align*}
\norm{D_{t,\ell} (\varrho Q_{11}) }_N &\leq \norm{(\pa_t + m_\ell P_{\le \ell^{-1}} \varrho^{-1} \cdot \na) (\varrho Q_{11}) }_N + \norm{((m_\ell P_{> \ell^{-1}} \varrho^{-1}) \cdot \na) (\varrho Q_{11}) }_N\\
&\lesssim \ell^{-N} \norm{(\pa_t + m_\ell P_{\le \ell^{-1}} \varrho^{-1} \cdot \na) (\varrho Q_{11}) }_0 \\
&\quad+ \sum_{N_1+N_2+N_3=N} \norm{m_\ell}_{N_1} \norm{P_{> \ell^{-1}} \varrho^{-1}}_{N_2} \norm{ \na (\varrho Q_{11}) }_{N_3}\\
&\lesssim \ell^{-N} (\norm{D_{t,\ell} (\varrho Q_{11}) }_0 + \norm{(m_\ell P_{> \ell^{-1}} \varrho^{-1} \cdot \na) \varrho Q_{11} }_0)  + \ell^{2-N} (\la_q \de_q^\frac12)^2\\
&\lesssim \ell^{-N} (\norm{D_{t,\ell} (\varrho Q_{11} )}_0 + \de_q^\frac12 (\la_q \de_q^\frac12)^2) + \ell^{1-N} (\la_q \de_q^\frac12)^3 \\
&\lesssim \ell^{-N} \de_q^\frac12 (\la_q \de_q^\frac12)^2.
\end{align*}
The second inequality follows from Bernstein's inequality. 
 
 In order to estimate the advective derivative of $Q_{12}$, we first estimate
\begin{align*}
\norm{D_{t,\ell} &\left(m_\ell\otimes m_\ell - P_{\le \ell^{-1}} (m\otimes m)\right)}_N \\
&\leq \norm{ (D_{t,\ell}m)_\ell\otimes m_\ell +
 m_\ell \otimes(D_{t,\ell}m)_\ell - P_{\le \ell^{-1}}D_{t,\ell} (m\otimes m)}_N\\
&\quad + 2 \norm{([m_\ell P_{\leq \ell^{-1}}{\varrho}^{-1}\cdot \na, P_{\le \ell^{-1}}] m) \otimes m_\ell}_N + \norm{ [m_\ell P_{\leq \ell^{-1}}{\varrho}^{-1}\cdot \na, P_{\le \ell^{-1}}] (m \otimes m)}_N\\
&\quad + 2 \norm{([m_\ell P_{> \ell^{-1}}{\varrho}^{-1}\cdot \na, P_{\leq \ell^{-1}}] m) \otimes m_\ell}_N + \norm{ [m_\ell P_{> \ell^{-1}}{\varrho}^{-1}\cdot \na, P_{\le \ell^{-1}}] (m \otimes m)}_N\\
&\lesssim \ell^{2-N} \norm{D_{t,\ell}m}_1 \norm{m}_1 + \sum_{N_1+N_2=N} \ell^{1-N_1} \norm{m}_1^2  \norm{m_\ell}_{N_2} + \ell^{1-N} \norm{m}_1 \norm{ m \otimes m}_1\\
&\quad+ \sum_{N_1+N_2=N} \ell^{1-N_1} \norm{\na (m_\ell P_{> \ell^{-1}}{\varrho}^{-1})}_{N_1} \norm{ m}_1 \norm{m_\ell}_{N_2} + \ell^{1-N} \norm{\na (m_\ell P_{> \ell^{-1}}{\varrho}^{-1})}_N \norm{m \otimes m}_1\\
&\lesssim \ell^{2-N} \la_q^2 \de_q^\frac12 \la_q \de_q^\frac12 + \ell^{1-N} (\la_q \de_q^\frac12)^2 + \ell^{1-N} (\la_q \de_q^\frac12)^2 + \ell^{1-N} \la_q \de_q^\frac12 + \ell^{1-N} \la_q \de_q^\frac12\\ &\lesssim
\ell^{1-N} (\la_q \de_q^\frac12)^2
\lec\ell^{-N} \la_q \de_q^\frac12
\end{align*}
In the second inequality, we have used \eqref{est.com} for the first term, \eqref{est.com1} for the second and third term and \eqref{est.com3} for the latter two terms in the estimate. The third inequality follows from $\norm{m_\ell P_{> \ell^{-1}}{\varrho}^{-1}}_{\overline{N}+1} \lesssim_{\overline{N}, \varrho} 1$. 
Then, it follows that
\begin{align*}
    \norm{D_{t,\ell}Q_{12}}_{N}
    &\lec \norm{ \na \varrho^{-1}:  D_{t,\ell}\left(m_\ell\otimes m_\ell - P_{\le \ell^{-1}} (m\otimes m)\right)}_N\\
    &\quad+\norm{\left(m_\ell\otimes m_\ell - P_{\le \ell^{-1}} (m\otimes m)\right): D_{t,\ell}\na\varrho^{-1}}_{N}
    \lec \ell^{-N} \de_q^\frac12(\la_q\de_q^\frac12)^2.
\end{align*}

Finally, it remains to estimate the term $Q_2$. We write it as
\begin{align*}
Q_2 = \frac{P_{\le \ell^{-1}} \div (m\otimes m)}{\varrho} - P_{\le \ell^{-1}}\left(\frac{\div (m\otimes m)}{\varrho}\right) + (P_{\le \ell^{-1}} (m \otimes m)) \na \varrho^{-1} - P_{\le \ell^{-1}}( m \otimes m \na \varrho^{-1})
\end{align*}
Since both the first and last two terms are of the form $(P_{\le \ell^{-1}}f) g - P_{\le \ell^{-1}}(fg)$, we can apply \eqref{est.com0} to get that 
\begin{align*}
\norm{Q_2}_N \lesssim \ell^{1-N} \la_q \de_q^{\frac12} + \ell^{1-N}
\end{align*}
Observe that we can estimate the advective derivative of $(P_{\le \ell^{-1}}f) g - P_{\le \ell^{-1}}(fg)$ as follows
\begin{align*}
&\norm{D_{t,\ell} (f_\ell g - (fg)_\ell)}_N = \norm{D_{t,\ell}(f_\ell g) - D_{t,\ell}(fg)_\ell)}_N \\
&= \norm{(D_{t,\ell}f_\ell)g + f_\ell D_{t,\ell}g - (D_{t,\ell}(fg))_\ell - [{m_\ell}/{\varrho} \cdot \na, P_{\le \ell^{-1}}](fg)}_N\\
&\leq \norm{(D_{t,\ell}f)_\ell g - ((D_{t,\ell}f)g)_\ell}_N + \norm{f_\ell D_{t,\ell}g  - (fD_{t,\ell}g)_\ell}_N+ \Norm{\left[\left[({m_\ell}/{\varrho}) \cdot \na, P_{\le \ell^{-1}}\right],g\right]f}_{N}\\
&\leq \norm{(D_{t,\ell}f)_\ell g - ((D_{t,\ell}f)g)_\ell}_N + \sum_{N_1+N_2=N} \norm{f_\ell}_{N_1} \norm{(D_{t,\ell}g - (D_{t,\ell}g)_\ell)}_{N_2}   \\
&\quad +\norm{f_\ell (D_{t,\ell}g)_\ell-(fD_{t,\ell}g)_\ell}_N+\Norm{\left[\left[({m_\ell}/{\varrho}) \cdot \na, P_{\le \ell^{-1}}\right],g\right]f}_{N}.
\end{align*}
We estimate this by applying \eqref{est.com0} to the the first norm, \eqref{est.com} to the third. For the last term, we first note that since $\varrho$ is smooth, we can bound a finite number of its derivatives and so we can bound $\norm{m_\ell/\varrho}_N \lec_N \ell^{-N}\de_q^\frac12$ for $N \in [1, \overline N]$; then we use Lemma \ref{lem.com4}. Therefore, we get the estimate for advective derivative of $Q_2$ as
\begin{align*}
&\norm{D_{t,\ell} Q_2}_N\\ 
&\lesssim \ell^{1-N} \norm{D_{t,\ell} \div (m \otimes m)}_0 \norm{\varrho^{-1}}_{\max (1,N)} + \sum_{N_1+N_2=N} \ell^{-N_1}  \norm{\div (m \otimes m)}_0
\norm{D_{t,\ell} \varrho^{-1}}_{N_2}
\\
& \quad + \ell^{2-N} \norm{\div (m \otimes m)}_1 \norm{D_{t,\ell} \varrho^{-1}}_{1} 
+ (\ell^{1-N} \de_q^\frac12\norm{\na \div (m \otimes m)}_0  
+ \ell^{-N} \de_q^\frac12 \norm{\div (m \otimes m)}_0) 
\\
& \quad +\ell^{1-N} \norm{D_{t,\ell} (m \otimes m)}_0 \norm{\na \varrho^{-1}}_{\max (1,N)} + 
\sum_{N_1+N_2=N} \ell^{-N_1}
 \norm{m \otimes m}_0 \norm{D_{t,\ell} \na \varrho^{-1}}_{N_2} \\
& \quad+ \ell^{2-N} \norm{ m \otimes m}_1 \norm{D_{t,\ell} \na \varrho^{-1}}_{1} 
+ (\ell^{1-N}\de_q^\frac12\norm{\na (m \otimes m)}_0
+ \ell^{-N}\de_q^\frac12 \norm{m \otimes m}_0) \\
&\lesssim \ell^{-N} (\ell \la_q^2 \de_q + \la_q\de_q) + \ell^{-N} \la_q \de_q^\frac12 + \ell^{2-N} \la_q^2 \de_q^\frac12 \ell^{-1} \de_q^\frac12
+ \ell^{1-N} \la_q^2 \de_q + \ell^{-N} \la_q \de_q \\
&\quad + \ell^{1-N} \la_q \de_q  + \ell^{-N}
 + \ell^{2-N} \la_q \de_q^\frac12 \ell^{-1} \de_q^\frac12  + \ell^{1-N} \la_q \de_q + \ell^{-N}\de_q^\frac12\\
&\lesssim \ell^{-N} \de_q^\frac12 (\la_q \de_q^\frac12)^2
\end{align*}
Observe that we have used $\norm{D_{t,\ell} \div m \otimes m}_0 \leq 2 \norm{(D_{t,\ell}m) \otimes m}_1 + \norm{\div \frac{m_\ell}{\varrho}}_0 \norm{m \otimes m}_1 \lesssim \la_q^2 \de_q + \ell^{-1} \de_q^\frac12 \la_q \de_q^\frac12$
\end{proof}

\subsection{Estimates on the backward flow} Finally we address the estimates on the backward flow $\xi_I$.

\begin{lem}\label{lem:est.flow} {For every $b>1$ there exists $\La_0=\La_0(b, \varrho)$ such that for $\la_0\geq \La_0$} the backward flow map $\xi_I$ satisfies the following estimates on 
the time interval $\cal{I}_u = [t_u - \frac 12 \tau_q , t_u + \frac 32\tau_q]\cap [0,T]+2\tau_q$
\begin{align}
&\norm{\I - \na \xi_I}_{C^0(\cal{I}_u \times \R^3)} \leq \frac 15 \label{est.flow1}\\
&\norm{D_{t,\ell}^s \na \xi_I}_{C^0(\cal{I}_u; C^N(\R^3))}\lec_{\varrho, {p}, N,M} \ell^{-N} (\la_q \de_q^\frac12)^s \label{est.flow2} \\
&\norm{D_{t,\ell}^s (\na \xi_I)^{-1}}_{C^0(\cal{I}_u; C^N(\R^3))}
\lec_{\varrho,{p},N,M} \ell^{-N} (\la_q \de_q^\frac12)^s, \label{est.flow3}
\end{align}
for any $N\geq 0$ and $s=0,1,2$. Note that the implicit constants in the inequalities are independent of the index $I=(m,n,f)$. {In particular,
\begin{align}\label{est.flow.indM}
\norm{\na \xi_I}_{C^0(\cal{I}_u; C^N(\R^3))} 
+ \norm{(\na \xi_I)^{-1}}_{C^0(\cal{I}_u; C^N(\R^3))}
\lec_{\varrho,N} \ell^{-N}.
\end{align}
The implicit constant in this inequality is also independent of $M$.}
\end{lem}
\begin{proof} {First, we can find $\La_0(b,\varrho)$ such that for any $\la_0\geq \La_0(b)$, $\tau_q\norm{\na \frac{m_\ell}{\varrho}}_0\leq \frac 1{10}$ holds.} Then, \eqref{est.flow1} easily follows from \eqref{bf.Id}. Also, 
\begin{align}
&\norm{\na \xi_I}_{C^0(\cal{I}_u; C^N(\R^3))} \lec_N 1 + \tau_q
\norm{\na \frac{m_\ell}{\varrho}}_{N} \lec_{\varrho,N} 1+ \ell^{-N} \lec_{\varrho,N} \ell^{-N} \label{est.flow01}
\end{align}
from which follows $\norm{(\na \xi_I)^{-1}}_{C^0(\cal{I}_u; C^N(\R^3))} \lec_N \ell^{-N}$. 
Since we have
\begin{align*}
D_{t,\ell} \na \xi_I = -(\na\xi_I)(\na \frac{m_\ell}{\varrho}), 
\quad D_{t,\ell}^2 \na \xi_I = (\na\xi_I)(\na \frac{m_\ell}{\varrho})^2-(\na \xi_I)D_{t,\ell}\na \frac{m_\ell}{\varrho},
\end{align*}
using \eqref{est.Dtvl} and \eqref{est.flow01}, $\norm{D_{t,\ell}^s \na \xi_I}_{C^0(\cal{I}_u; C^N(\R^3))} \lec_{\varrho,{p},N,M} \ell^{-N}(\la_q\de_q^\frac12)^s $ easily follows. Lastly, we have 
\begin{align*}
D_{t,\ell}(\na\xi_I)^{-1}
= \na \frac{m_\ell}{\varrho} (\na \xi_I)^{-1}, \quad
D_{t,\ell}^2 (\na\xi_I)^{-1}
= D_{t,\ell}\na \frac{m_\ell}{\varrho} (\na \xi_I)^{-1} + (\na \frac{m_\ell}{\varrho})^2 (\na\xi_I)^{-1}.
\end{align*}
Therefore, \eqref{est.flow3} can be obtained similarly. 
\end{proof}

\section{Estimates in the momentum correction}

The main point of this section is to get the estimates on the momentum correction. {In this section, we set $\norm{\cdot}_N = \norm{\cdot}_{C^0([0,T]+\tau_q; C^N(\T^3))}$.}

The following proposition provides the estimates for the perturbation $n$. 
\begin{prop}\label{p:velocity_correction_estimates} For $N=0,1,2$ and $s=0,1,2$, the following estimates hold for $n_o$, $n_c$, and $n=n_o+n_c$:
\begin{align}
& \tau_q^s \norm{D_{t,\ell}^s n_o}_{N}
{\lec_{{\varrho, p,}M}} \la_{q+1}^{N} \de_{q+1}^\frac12 	\label{est.W}	\\
&\tau_q^s\norm{D_{t,\ell}^s n_c}_{N}
 \lec_{{\varrho, p,}M} \la_{q+1}^{N}\frac { \de_{q+1}^\frac12}{\la_{q+1}\mu_q} 		\label{est.Wc} \\
&\tau_q^s\norm{D_{t,\ell}^s n}_{N}
 {\lec_{{\varrho, p,}M}} \la_{q+1}^{N}\de_{q+1}^\frac12, \label{est.w}
\end{align}
where the implicit constants are independent of $s$, $N$, and $q$. Moreover, 
\begin{align}\label{est.w.indM}
\norm{n}_{N} \lec_{\varrho} \la_{q+1}^N \de_{q+1}^\frac12,
\end{align}
where the implicit constant is additionally independent of {$p$} and $M$. 
\end{prop}

The latter estimates are in fact a simple consequence of estimates on the functions $b_{u,k}$, $c_{u,k}$, $d_{u,k}$ and $e_{u,k}$ defined in \eqref{coef.def} and \eqref{def.e}

\begin{lem}\label{lem:est.coe} For any $N\geq 0$ and $s=0,1,2$, the coefficients $b_{u,k}$, $c_{u,k}$, $d_{u,k}$, and $e_{u,k}$ defined by \eqref{coef.def} and \eqref{def.e} satisfy the following,
\begin{align}
\ta_q^s\norm{D_{t,\ell}^s b_{u,k}}_N
&{\lec} \ \mu_q^{-N}\max_I |\dot{b}_{I,k}| \label{est.b}\\
\ta_q^s\norm{D_{t,\ell}^s c_{u,k}}_N
&{\lec} \ \mu_q^{-N}\max_I |\dot{c}_{I,k}| \label{est.c}\\
\ta_q^s\norm{D_{t,\ell}^s d_{u,k}}_N
&{\lec}\  \mu_q^{-N}\max_I |\dot{d}_{I,k}| \label{est.d}\\
\ta_q^s\norm{D_{t,\ell}^s e_{u,k}}_N
&{\lec}\  \mu_q^{-N}\max_I |\dot{b}_{I,k}|.\label{est.e},
\end{align}
where the implicit constants in the inequalities \eqref{est.b}-\eqref{est.e} depend on $\varrho$, {$p$}, $N$, and  $M$.
{Moreover, for $N=0,1,2$,
\begin{align}\label{est.be.indM}
\norm{b_{u,k}}_N + \norm{e_{u,k}}_N
\lec_{\varrho} \mu_q^{-N} \max_I |\dot{b}_{I,k}|,
\end{align}
where the implicit constant is independent of $M$ and $N$.}
\end{lem}

\begin{rem}\label{r:Fourier_coefficients}
Observe that, by the definition of the respective coefficients, the moduli $|\dot{b}_{I,k}|$,  $|\dot{c}_{I,k}|$ and $|\dot{d}_{I,k}|$ just depend on the third component of the index $I=(u,v,f)$, since they involve the functions $\psi_f$, but not the ``shifts'' $z_{u,v}$. In particular, the set of their possible values is a finite number, independent of  $q$ and just depending on the collection of the family of functions $\psi_f$ and on the frequency $k$.
\end{rem}

\begin{proof}
First of all, it is easy to see that for any $s=0,1,2$ and $N\geq 0$, 
\begin{equation}\begin{split}
\label{est.th.chi}
\norm{D_{t,\ell}^s\th_I}_{C^0(\R)} 
&= \norm{\pa_t^s\th_I}_{C^0(\R )} \lec \tau_q^{-s}, \\ 
\norm{\chi_I(\xi_I)}_{C^0(\cal{I}_u; C^N(\R^3))} 
&\lec_{\varrho, N} \mu_q^{-N}, \  D_{t,\ell}^s [\chi_I(\xi_I)] = 0,
\end{split}\end{equation}
where $\cal{I}_u = [t_u - \frac 12\tau_q, t_u + \frac 32 \tau_q]$. Indeed, the estimate of $\chi_I(\xi_I)$ follows from \eqref{est.flow2}, Lemma \ref{lem:est.com}, and $\ell^{-1}\leq \mu_q^{-1}$. We remark that the implicit constants are independent of $I$. 

Recall that when $f \in \mathscr{I}_\ph$, 
\[
\ga_I=\frac{\la_q^{-\ga}\de_{q+1}^\frac12\Ga_I}{|\tilde{f}_I|^{\frac{2}{3}}} =
\frac{\la_q^{-\ga}\de_{q+1}^\frac12\Ga_I}{|\nabla \xi_I^{-1} f_I|^{\frac{2}{3}}}
\] 
for 
\[
\Ga_I(x) 
=\Gamma_{f_I}^{\frac{1}{3}}
(-2 \la_q^{3\ga}\de_{q+1}^{-\frac 32}(\na\xi_I)\varrho^3\ph_\ell)
\]
where $\Gamma_f$'s are the functions given by Lemma \ref{lem:geo2}.  
First it is easy to see that \eqref{est.flow3} implies 
\begin{align}\label{est4}
\norm{D_{t,\ell}^s [(\na\xi_I)^{-1}f_I]}_{C^0(\cal{I}_u; C^N(\R^3))}
{\lec_{{\varrho, p}, N,M}} (\la_q \de_q^\frac 12)^s \ell^{-N}.
\end{align}
Also, using \eqref{est.flow2} and \eqref{est.mph},  
\begin{align}\label{est5}
\norm{D_{t,\ell}^s(2\la_q^{3\ga} \de_{q+1}^{-\frac32}(\na \xi_I)\varrho^3 \ph_\ell)}_{C^0(\cal{I}_u; C^N(\R^3))}
{ \lec_{{\varrho, p}, N,M}} (\ell_t^{-s} + (\la_q \de_q^\frac12)^s) \ell^{-N}\lec \ta_q^{-s}\ell^{-N}.
\end{align}
Next, for any smooth functions $\Ga=\Ga(x)$ and $g=g(t,x)$ we have
\begin{equation}\begin{split}\label{formula0}
&\norm{D_{t,\ell} \Ga(g)}_{C^N_x} \lec \sum_{N_1+N_2=N} \norm{D_{t,\ell} g}_{C^{N_1}_x} \norm{(\na\Ga)(g)}_{C^{N_2}_x},\\
&\norm{D_{t,\ell}^2 \Ga(g)}_{C^N_x} 
\lec \sum_{N_1+N_2=N} \norm{D_{t,\ell}^2 g}_{C^{N_1}_x} \norm{(\na\Ga)(g)}_{C^{N_2}_x}
+ \norm{D_{t,\ell} g\otimes D_{t,\ell} g}_{C^{N_1}_x} \norm{(\na^2\Ga)(g)}_{C^{N_2}_x},
\end{split}\end{equation}
and therefore we obtain by Lemma \ref{lem:est.com}
\begin{align*}
&\norm{D_{t,\ell}^s |(\na\xi_I)^{-1}f_I|^{-\frac 23}}_{C^0(\cal{I}_u; C^N(\R^3))}
{\lec_{{\varrho, p},N,M}} (\la_q \de_q^\frac 12)^s \ell^{-N}\\
&\norm{D_{t,\ell}^s[ ({\Ga_{f_I}})^\frac13(-2\la_q^{3\ga} \de_{q+1}^{-\frac32}(\na \xi_I)\varrho^3 \ph_\ell)]}_{C^0(\cal{I}_u; C^N(\R^3))}
{\lec_{{\varrho, p},N,M}} \ta_q^{-s} \ell^{-N}.  
\end{align*}
Here, we used \eqref{est4} and \eqref{est5} which we can apply thanks to the fact that $|(\na \xi_I)^{-1} f_I| \geq  \frac 34$ and $\Ga_{f_I}\geq 3$ (according to our choice of $N_0$ in applying Lemma \ref{lem:geo2}). 
Also the implicit constant in the second inequality can be chosen to be independent of $I$ because of the finite cardinality of the functions $f_I$. {On the other hand, in the case of $s=0$ and $N=0,1,2$, because of \eqref{est.ph} and \eqref{est.flow.indM}, the implicit constants in both inequalities can be chosen to be independent of $N$ and $M$.}
Therefore, it follows that, when $I \in \mathscr{I}_\ph$, 
\begin{align}\label{est.Ga.ph}
\norm{D_{t,\ell}^s \de_{q+1}^{-\frac12}\ga_I}_N {\lec_{{\varrho, p},N,M}} \ta_q^{-s} \ell^{-N}.
\end{align}
In particular, for $N=0,1,2$,
{
\begin{align*}
\norm{\de_{q+1}^{-\frac12}\ga_I}_N \lec_{\varrho} \ell^{-N},
\end{align*}
where the implicit constant depends on an upper bound of $\norm{\varrho}_{C([0,T]+\tau_{-1}; C^2(\T^3))}$ and $\e_0$ more precisely.
}

On the other hand, when $I\in \mathscr{I}_R$, recall that $\de_{q+1}^{-\frac 12} \varrho^{-1} \ga_I = \Ga_I = \Ga_{f_I} (\I+\de_{q+1}^{-1}\mathcal{M}_I)$ for a finite collection of smooth functions $f_I$ chosen through Lemma \ref{lem:geo1}. First $\mathcal{M}_I$ can be estimated as 
\begin{align*}
&\norm{\mathcal{M}_I}_{C^0(\cal{I}_u; C^N(\R^3))} \\
&\lec_N\ \de_{q+1}  \norm{(\na \xi_I)(\na \xi_I)^{\top} - \I}_{C^0(\cal{I}_u; C^N(\R^3))}  + \sum_{\substack{N_1+N_2\\+N_3=N}}\norm{\na \xi_I}_{C^0(\cal{I}_u; C^{N_1}_x)} \norm{R_\ell}_{N_2} \norm{\na \xi_I}_{C^0(\cal{I}_u; C^{N_3}_x)}\\
&\ + \norm{\varrho^{-2}}_{C^0(\cal{I}_u; C^{N}_x)} \sum\sum \norm{\na \xi_I}_{C^0(\cal{I}_u; C^{N_1}_x)} \norm{\chi_{J}^2(\xi_{J})}_{C^0(\cal{I}_u; C^{N_2}_x)} \norm{\ga_{J}^2}_{C^0(\cal{I}_u; C^{N_3}_x)}  \norm{\na \xi_I}_{C^0(\cal{I}_u; C^{N_4}_x)}\\
&\lec_{{\varrho, }N,M} \de_{q+1} \mu_q^{-N}, 
\end{align*}
where the double summations in the third line is taken over $N_1+N_2+N_3+N_4=N$ and $J=(u',v',f')$ satisfying {$f\in \cF_{{J}, \ph} := \cF^{[v'],\ph}$} and the last line follows from
\eqref{est.flow1}, \eqref{est.flow2}, \eqref{est.mR}, \eqref{est.th.chi}, and \eqref{est.Ga.ph}. In the case of $N=0,1,2$, we note that the implicit constants can be chosen to be independent of $N$ and $M$. 
Similarly, we have
\begin{align*}
\norm{D_{t,\ell}^s \mathcal{M}_I}_N \lec_{{\varrho, p,}N,M} \de_{q+1}\ta_q^{-s} \mu_q^{-N},
\end{align*}

Then, \eqref{formula0} and Lemma \ref{lem:est.com} imply that when $f_I \in \cF_{I, R}=\cF_{(u,v,f), R}:= \cF^{[v],R}$, for $s=0,1,2$ and $N\geq 0$, 
\begin{align}\label{est.Ga.R}
\norm{D_{t,\ell}^s \de_{q+1}^{-\frac 12} \varrho^{-1} \ga_I}_N = \norm{D_{t,\ell}^s(\Ga_{f_I} (\I+\de_{q+1}^{-1}\mathcal{M}_I))}_N
\lec_{{\varrho, p,}N,M} \ta_q^{-s} \mu_q^{-N}.
\end{align}
In particular, for $s=0, N=0,1,2$, the implicit constant can be chosen to be independent of $M$ and $N$ but {to depend only on $\e_0$ and an upper bound of $\norm{\varrho}_{C([0,T]+\tau_{-1}; C^2(\T^3))}$};
\begin{align*}
\norm{\de_{q+1}^{-\frac 12} \varrho^{-1} \ga_I}_N 
\lec_{\varrho }  \mu_q^{-N}.
\end{align*}
Finally, recall the definition of $b_{u,k}$, $c_{u,k}$, $d_{u,k}$, and $e_{u,k}$. Then, the estimates \eqref{est.b}-\eqref{est.be.indM} follows from \eqref{est.th.chi}, \eqref{est.Ga.ph}, \eqref{est.Ga.R}, and \eqref{est.flow.indM}.
\end{proof}

\begin{proof}[Proof of Proposition \ref{p:velocity_correction_estimates}]
Using \eqref{est.be.indM}, \eqref{est.flow.indM}, and \eqref{est.k}, we easily have  $\la_{q+1}^{-N}\norm{n_o}_{N} \lec_{{\varrho,}N} \de_{q+1}^\frac12$ and  $\la_{q+1}^{-N}\norm{n_c}_N \lec_{{\varrho,}N} (\la_{q+1}\mu_q)^{-1}\de_{q+1}^\frac12$, recalling 
Remark \ref{r:Fourier_coefficients}. On the other hand, we observe that $D_{t,\ell} e^{i\la_{q+1}k\cdot \xi_I} = 0$ because of $D_{t,\ell} \xi_I =0$. Hence the remaining inequalities in \eqref{est.W} and \eqref{est.Wc} are obtained in a similar fashion. Finally, \eqref{est.w} follows from \eqref{est.W} and \eqref{est.Wc}. Note that all estimates used in the proof have implicit constants independent of $q$. Moreover, the finite cardinalities of the range of $N$ and $s$ make it possible to choose the implicit constants in \eqref{est.W}, \eqref{est.Wc}, and \eqref{est.w} independent of $N$ and $s$ too. {Furthermore, when $s=0$, we can also make the implicit constants independent of $M$.}
\end{proof}

\section{A microlocal lemma} 

We will need in the sequel a suitable extension of \cite[Lemma 4.1]{IsVi2015} whose proof can be found in \cite[Lemma 8.1]{DLK20}. We will use the notation
\[
\crF[f] (k) = \dint_{\T^3} f(x) e^{-ix \cdot k} dx,\quad
f (x) = \sum_{k\in \Z^3}\crF[f](k)   e^{ik \cdot x} 
\]   
for the Fourier series of periodic functions.

\begin{lem}[Microlocal Lemma]\label{mic} Let $T$ be a Fourier multiplier defined on $C^\infty(\T^3)$ by
\[
\crF[Th](k) = \mathfrak{m}(k)\crF[h](k), \quad \forall k\in \Z^3
\]
for some $\mathfrak{m}$ which has an extension in $\cal{S}(\R^3)$ (which for convenience we keep denoting by $\mathfrak{m}$). Then, for any $n_0\in \N$, $\la>0$, and any scalar functions $a$ and $\xi$ in $C^\infty(\T^3)$, 
$T(a e^{i\la \xi})$ can be decomposed as
\begin{equation*}
\begin{split}
T(a e^{i\la \xi}) 
=&\   \left[a \mathfrak{m}(\la\na \xi)
+\sum_{k=1}^{2{n_0}} C_{k}^\la(\xi,a): (\na^k \mathfrak{m})(\la \na\xi)
+\e_{n_0}(\xi,a)\right]e^{i\la \xi}
\end{split}
\end{equation*}
for some tensor-valued coefficient {$C_{k}^\la(\xi,a)$} and a remainder $\e_{n_0}(\xi, a)$ which is specified in the following formula: 
\begin{equation}\label{def.eN}\begin{split}
\e_{n_0}(\xi, a)(x)
&= 
\sum_{\substack{n_1+n_2\\=n_0+1}} \frac{(-1)^{n_1}c_{n_1,n_2}}{n_0!} \\
&\cdot\int_0^1\int_{\R^3} \widecheck{\mathfrak{m}}(y) e^{-i\la \na\xi(x)\cdot y} ((y\cdot\na)^{n_1}a)(x-ry) e^{i\la Z[\xi](r)}\be_{n_2} [\xi](r) (1-r)^{n_0} dydr,
\end{split}
\end{equation}
where $c_{n_1,n_2}$ is a constant depending only on $n_1$ and $n_2$, and the function $\be_n[\xi]$ is
\begin{align*}
\be_{n} [\xi](r)
&=B_{n} (i\la Z'(r), i\la Z''(r), \cdots, i\la Z^{(n)}(r)),\\
Z{[\xi]}(r) &=Z{[\xi]}_{x,y}(r) = r\int_0^1(1-s) (y\cdot\na )^2 \xi(x-rsy) ds,
\end{align*}
with $B_n$ denoting the $n$th complete exponential Bell polynomial; 
\begin{align}\label{e:Bell}
B_n(x_1,\dots,x_n)
=\sum_{k=1}^n B_{n,k}(x_1,x_2,\dots,x_{n-k+1}),
\end{align} 
where
\begin{align*}
B_{n,k}(x_1,x_2,\dots,x_{n-k+1}) =
 \sum\frac{n!}{j_1!  j_2! \cdots j_{n-k+1}!}
\left(\frac{x_1}{1!}\right)^{j_1}
\left(\frac{x_2}{2!}\right)^{j_2}\cdots
\left(\frac{x_{n-k+1}}{(n-k+1)!}\right)^{j_{n-k+1}},
\end{align*}
and the summation is taken over $\{j_k\}\subset \N\cup\{0\}$ satisfying
\begin{align}\label{sum.bel}
j_1 + j_2 + \cdots + j_{n-k+1} = k, \quad
j_1 + 2 j_2 + 3 j_3 + \cdots + (n-k+1)j_{n-k+1} = n.
\end{align}
\end{lem}

\bigskip
We now collect an important consequence on the anti-divergence operator $\mathcal{R}$.

\begin{cor}\label{cor.mic2} Let $N=0,1,2$ and $F = \sum_{k\in \Z^3\setminus\{0\}}\sum_{u\in \Z} a_{u,k}e^{i\la_{q+1}k\cdot\xi_u}$. Assume that a function $a_{u,k}$ fulfills the following requirements. 
\begin{enumerate}[(i)]
\item The support of $a_{u,k}$ satisfies $\supp(a_{u,k}) \subset (t_u - \frac 12\tau_q, t_u + \frac 32\tau_q)\times \R^3$. In particular, 
for $u$ and $u'$ neither same nor adjacent, we have
\begin{equation}\label{dis.amk}
\supp(a_{u,k}) \cap \supp(a_{u',k'}) = \emptyset, \quad \forall k, k' \in \Z^3\setminus\{0\}.
\end{equation}
\item For any $0\le j\le n_0+1$ and $(u,k)\in \Z\times \Z^3$,  
\begin{equation}\label{con.a}
    \norm{a_{u,k}}_j+ (\la_{q+1}\de_{q+1}^\frac 12)^{-1}\norm{D_{t,\ell} a_{u,k}}_j\lec_j  \mu_q^{-j} |\dot{a}_k|, \quad
\sum_k |k|^{n_0+2} |\dot{a}_k| \leq  a_F, 
\end{equation}
for some $a_F>0$, where $ n_0 = {\ceil{\frac{2b(2+\al)}{(b-1)(1-\al)}}}$ and {$\norm{\cdot}_{j} = \norm{\cdot}_{C(\mathcal{I}; C^j(\T^3))}$ on some time interval $\mathcal{I}\subset \R$. }
\end{enumerate}
Then, for any $b\in (1,3)$, we can find $\La_0=\La_0(b, {\varrho, p})$ such that for any $\la_0\geq \La_0$, $\cR F$ satisfies the following inequalities: 
\begin{align*}
\norm{\cR F}_{N} \lec \la_{q+1}^{N-1} a_F, \quad
\norm{   D_{t,q+1} \cR F}_{N-1} \lec \la_{q+1}^{N-1}\de_{q+1}^\frac12 a_F\,
\end{align*}
upon setting $   D_{t,q+1} = \pa_t + \frac{m_{q+1}}{\varrho} \cdot \na $.
\end{cor}

The proof of Corollary \ref{cor.mic2} follows from the same argument in the proof of \cite[Corollary 8.2]{DLK20} with slight revision because the the backward flow map $\xi_u$ has velocity $m_\ell/\varrho$ instead of $v_\ell$ in \cite{DLK20}. The flow map $\xi_u$ and $m_{\ell}/\varrho$, on the other hand, satisfy the same estimates as for the original flow map and its velocity leading to in \cite{DLK20} (see Lemma \ref{lem:est.flow}, \eqref{est.vp}, \eqref{est.v.dif}), while the velocity $m_{\ell}/\varrho$ no longer has frequency localization to $\leq \ell^{-1}$. As a result, the same argument works except for one part in the material derivative estimate which relies on frequency localization. Also, we remark that through careful examination of the proof in \cite[Corollary 8.2]{DLK20} one could see that \eqref{con.a} only for $0\leq j\leq n_0+1$ is needed. For the completeness, we sketch the proof and point out the needed revision below.

\begin{proof}[Sketch of the proof.]
The proof is relying on the decomposition   
\begin{align}\label{dec.mic}
F 
= \cP_{\gtrsim \la_{q+1}} \left(\sum_{u,k} a_{u,k} e^{i\la_{q+1} k\cdot \xi_u} \right) - \sum_{u,k} \e_{n_0}^{\la_{q+1}}(k\cdot \xi_u, a_{u,k}) e^{i\la_{q+1} k\cdot \xi_u},
\end{align}
where $\cP_{\gtrsim \la_{q+1}}$ is defined by
\[
\cP_{\gtrsim \la_{q+1}} = \sum_{2^j \geq \frac 38\la_{q+1}} P_{2^j}
\]
and
\[
\e_{n_0}^{\la_{q+1}} (k\cdot \xi_u, a_{u,k})
 = \sum_{2^j\geq \frac 38\la_{q+1}} \e_{n_0, j} (k\cdot \xi_u, a_{u,k}).
\]
The remainder $ \e_{n_0, j}(\xi,a)$ is obtained by applying Lemma \ref{mic} to $P_{2^j}$ and $n_0= {\ceil{\frac{2b(2+\al)}{(b-1)(1-\al)}}}$. In particular, the remainder part of $F$ has frequency localization 
\begin{align}\label{rem.low}
\cP_{\lec \la_{q+1}} F := F - \cP_{\gtrsim \la_{q+1}}F
=- \sum_{k,u}\e_{n_0}^{\la_{q+1}}(k\cdot \xi_u, a_{u,k}) e^{i\la_{q+1} k\cdot \xi_u}
\end{align}
and satisfies
\begin{align}
\norm{\sum_{u,k}\e_{n_0}^{\la_{q+1}}(k\cdot \xi_u, a_{u,k})}_0 
&\lec_{n_0} (\la_{q+1}\mu_q)^{-(n_0+1)} a_F\lec \la_{q+1}^{-2}a_F, \label{est.eN}
\\
\norm{\sum_{u,k}D_{t,\ell} \e_{n_0}^{\la_{q+1}}(k\cdot \xi_u, a_{u,k})}_0 
&\lec_{n_0}   \la_{q+1}\de_{q+1}^\frac12 (\la_{q+1}\mu_q)^{-(n_0+1)} a_F 
\lec \la_{q+1}\de_{q+1}^\frac12\cdot\la_{q+1}^{-2} a_F. \label{est.DteN}
\end{align}
Using this, one can easily obtain $\norm{\cR F}_{N} \lec \la_{q+1}^{N-1} a_F$. To estimate the material derivative of $\mathcal{R}F$, we use the following decomposition,
\begin{align*}
    D_{t, q+1} \mathcal{R}F
    &= \mathcal{R} D_{t,\ell} F
    + \left[\frac{m_\ell} {\varrho} \cdot \na, \mathcal{R} \right]F
    + \left(\frac{n+ m_q-m_\ell}{\varrho} \right)\cdot \na \mathcal{R}F.
\end{align*}
The first and the last terms on the right hand side can be estimated as in \cite[Corollary 8.2]{DLK20}. To estimate the second term, we further decompose it into 
\begin{align*}
\left[m_\ell P_{\leq \ell^{-1}}\varrho^{-1} \cdot \na, \mathcal{R} \right]F
    + \left[m_\ell P_{> \ell^{-1}}\varrho^{-1} \cdot \na, \mathcal{R} \right]F.   
\end{align*}
Since $m_{\ell} P_{\leq \ell^{-1}}\varrho^{-1} = P_{\lec \ell^{-1}}(m_{\ell} P_{\leq \ell^{-1}}\varrho^{-1})$, we can estimate it as in \cite[Corollary 8.2]{DLK20} (see also the proof of Lemma \ref{lem:com*}.) Therefore, it suffices to estimate the remaining term;
\begin{align*}
    \Norm{\left[m_\ell P_{> \ell^{-1}}\varrho^{-1} \cdot \na, \mathcal{R} \right]F}_{N-1}
    &\lec \sum_{N_1+N_2=N-1}
    \norm{m_\ell P_{> \ell^{-1}}\varrho^{-1}}_{N_1}\norm{\na F}_{N_2}\\
    &\lec \la_{q+1}^N \ell^2 a_F
    \lec \la_{q+1}^{N-1}\de_{q+1}^\frac12 a_F,
\end{align*}
where we used $\norm{m_\ell P_{>\ell^{-1}}\varrho^{-1}}_{N_1}\lec{\ell^{2} \norm{m_\ell}_{0}} \norm{\varrho^{-1}}_{N_1+2}$ and the choice of $b<3$ and sufficiently large $\La_0$.

\end{proof}

\section{Estimates on the Reynolds stress}

In this section, we obtain the relevant estimates for the new Reynolds stress and its new advective derivative $   D_{t,q+1} R_{q+1} = \partial_t R_{q+1} + \frac{m_{q+1}}{\varrho}\cdot \nabla R_{q+1}$, summarized in the following proposition. For technical reasons it is however preferable to estimate rather $R_{q+1} - \frac{2}{3}  \zeta \I$, as indeed the estimates on the function $ \zeta (t)$ are akin to those for the new current, which will be detailed in the next section. For the remaining sections, we set $\norm{\cdot}_N = \norm{\cdot}_{C^0([0,T]+\tau_q; C^N(\T^3))}$.

\begin{prop}\label{p:Reynolds}
{There exists  $\bar{b}(\al)>1$ with the following property. For any $1<b<\bar{b}(\al)$ we can find 
$\Lambda_0 = \Lambda_0 (\al,b, M, {\varrho, p})$} such that the following estimates hold for every  $\lambda_0 \geq \Lambda_0$: 
\begin{equation}\label{est.asR}\begin{split}
\norm{R_{q+1} - {\textstyle{\frac{2}{3}}}  \zeta{/\varrho } \I \,}_N &\lec {C_{\varrho, p, M}}\la_{q+1}^N \cdot
 {\la_q^\frac12}\la_{q+1}^{-\frac12} \de_q^\frac14\de_{q+1}^\frac34 
 \leq \frac12 \la_{q+1}^{N-3\ga} \de_{q+2}, 
\\
\norm{   D_{t,q+1} (R_{q+1} - {\textstyle{\frac{2}{3}}}  \zeta{/\varrho } \I)}_{N-1} &\leq {C_{\varrho, p, M}} \la_{q+1}^N \de_{q+1}^\frac12\cdot
 {\la_q^\frac12}\la_{q+1}^{-\frac12} \de_q^\frac14\de_{q+1}^\frac34  \leq \frac12\la_{q+1}^{N-3\ga}\de_{q+1}^\frac12 \de_{q+2}.
\end{split}\end{equation}
for ant $N=0,1,2$ ( $\norm{\cdot}_{-1}$ is an empty statement), where $C_{\varrho, p, M}$ depends only upon $\varrho$, {$p$} and the $M>1$ of Proposition \ref{ind.hyp} and Proposition \ref{p:ind_technical}.
\end{prop}

Taking into account \eqref{e:splitting_of_Reynolds}, we will just estimate the separate terms $   R_T$, $   R_N$, $   R_{O1}$, $   R_{O2}$ and $   R_M$.
For the errors $ R_{O2}$ and $ R_M$, we use a direct estimate, while the other errors, including the inverse divergence operator, are estimated by Corollary \ref{cor.mic2}. In the following subsections, we fix $n_0 = \ceil{\frac{2b(2+\al)}{(b-1)(1-\al)}}$ so that $\la_{q+1}^2({\la_{q+1}\mu_q})^{-(n_0+1)}\lec \de_{q+1}^\frac12$ for any $q$ {and allow the dependence on $M$ of the implicit constants in $\lec$.} Also, we remark that 
\begin{align}\label{rel.par}
\frac {1}{\la_{q+1}\tau_q} +\frac{\de_{q+1}^\frac12}{\la_{q+1}\mu_q} \lec_M  {\la_q^\frac12}{\la_{q+1}^{-\frac12}} \de_q^\frac14\de_{q+1}^\frac14.
\end{align}
We note that it is enough to estimate $\varrho R{-\frac23 \zeta }$ and its advective derivative for the various Reynolds error terms. For the convenience, we restrict the range of $N$ as in \eqref{est.asR} in this section, without mentioning it further.

\subsection{Transport stress error}
Recall that
\[
 \varrho R_{T} = \mathcal{R} \left( \varrho D_{t,\ell} \frac{n}{ \varrho} - \div(m_q-m_\ell)\frac{n}{ \varrho}\right)
\]
Since $\textstyle{\varrho D_{t,\ell} ({n}/{ \varrho}) - \div(m_q-m_\ell)({n}/{ \varrho}) = \pa_t n + \div \left({(n\otimes m_\ell)}/{\varrho}\right)}$, we see that it has zero mean and thus we can apply the inverse divergence operator. Now as $D_{t,\ell} \xi_I =0$, we have  
\begin{align*}
\varrho D_{t,\ell} (n/ \varrho) 
&= \varrho D_{t,\ell} \left( \frac1\varrho \sum_{u\in\Z} \sum_{k\in \Z^3\setminus \{0\}} \de_{q+1}^\frac 12(b_{u,k} + (\la_{q+1}\mu_q)^{-1}e_{u,k}) e^{i\la_{q+1} k\cdot \xi_I}\right)\\
&=\sum_{u} \sum_{k} \de_{q+1}^\frac 12 \varrho D_{t,\ell} \left(\frac{b_{u,k}}{\varrho} + (\la_{q+1}\mu_q)^{-1}\frac{e_{u,k}}{\varrho}\right) e^{i\la_{q+1} k\cdot \xi_I}.
\end{align*}
Since $b_{u,k}$ and $e_{u,k}$ satisfy
$\supp(b_{u,k}), \supp(e_{u,k})\subset (t_u-\frac 12\tau_q, t_u + \frac 32\tau_q) \times \R^3$
and 
\begin{align*}
&\Norm{D_{t,\ell} (\frac{b_{u,k}}{\varrho} + (\la_{q+1}\mu_q)^{-1}\frac{e_{u,k}}{\varrho})}_{\bar{N}} + (\la_{q+1}\de_{q+1}^\frac 12)^{-1}\Norm{D_{t,\ell}^2 (\frac{b_{u,k}}{\varrho} + (\la_{q+1}\mu_q)^{-1}\frac{e_{u,k}}{\varrho})}_{\bar{N}}\\
&\qquad\qquad \lec_{\bar{N},M} \mu_q^{-\bar{N}} \frac {|\dot{b}_{I,k}|}{\tau_q},
\end{align*}
for any $\bar{N}\geq 0$ by \eqref{est.b} and \eqref{est.e}. We now have
\begin{align*}
\div(m_q-m_\ell)\frac{n}{ \varrho} = \sum_{u\in\Z} \sum_{k\in \Z^3\setminus \{0\}} \de_{q+1}^\frac 12 \frac{\div(m_q-m_\ell)}{\varrho} (b_{u,k} + (\la_{q+1}\mu_q)^{-1}e_{u,k}) e^{i\la_{q+1} k\cdot \xi_I}
\end{align*}
Since $\div(m_q-m_\ell)  = -\pa_t \varrho + (\pa_t \varrho)_\ell$, we can estimate it similarly as {
\begin{align*}
&\norm{\varrho^{-1}\div(m_q-m_\ell) (b_{u,k} + (\la_{q+1}\mu_q)^{-1}e_{u,k})}_{\bar N}   
\\
&+(\la_{q+1}\de_{q+1}^\frac12)^{-1}\norm{D_{t,\ell }(\varrho^{-1}\div(m_q-m_\ell) (b_{u,k} + (\la_{q+1}\mu_q)^{-1}e_{u,k}))
}_{\bar N}
\lec_{\varrho, p, M, \bar{N}} \mu_q^{-\bar N}\frac{|\dot{b}_{I,k}|}{\tau_q}. 
\end{align*}
}
We can thus apply Corollary \ref{cor.mic2} to get
\begin{align}\label{est.RT}
\norm{ R_T}_N
\lec \la_{q+1}^N\frac {\de_{q+1}^\frac 12}{\la_{q+1} \tau_q} , \quad
\norm{   D_{t,q+1}  R_T}_{N-1}  \lec \la_{q+1}^N\de_{q+1}^\frac12 \frac {\de_{q+1}^\frac 12}{\la_{q+1} \tau_q}.
\end{align}
 
 \subsection{Nash stress error}
Recall $ \varrho R_{N} = \mathcal{R} \left( (n\cdot \na) \frac{m_\ell}{ \varrho}\right)$ and observe that
\begin{align*}
(n\cdot \na) \frac{m_\ell}{ \varrho}
&=  \sum_{u} \sum_{k\in \Z^3\setminus \{0\}} \de_{q+1}^\frac 12((b_{u,k} + (\la_{q+1}\mu_q)^{-1}e_{u,k})\cdot \na)\frac{m_\ell}{ \varrho} e^{i\la_{q+1} k\cdot \xi_I}.
\end{align*}
Since $b_{u,k}$ and $e_{u,k}$ satisfy
$\supp(b_{u,k}), \supp(e_{u,k})\subset (t_u-\frac 12\tau_q, t_u + \frac 32\tau_q) \times \R^3$ and
\begin{align*}
\norm{(b_{u,k} + (\la_{q+1}\mu_q)^{-1}e_{u,k})\cdot \na)\frac{m_\ell}{ \varrho}}_{\bar{N}}
&\lec_{\bar N} \mu_q^{-{\bar{N}}} |\dot{b}_{I,k}|\la_q\de_{q}^\frac12\\
\norm{D_{t,\ell}[(b_{u,k} + (\la_{q+1}\mu_q)^{-1}e_{u,k})\cdot \na)\frac{m_\ell}{ \varrho}]}_{\bar{N}}
&\lec_{\bar N} \la_{q+1}\de_{q+1}^\frac12 \mu_q^{-{\bar{N}}} |\dot{b}_{I,k}|\la_q\de_{q}^\frac12
\end{align*}
for any ${\bar{N}}\geq0$ by \eqref{est.vp}, \eqref{est.Dtvl}, \eqref{est.b}, and \eqref{est.e}, we apply Corollary \ref{cor.mic2} and obtain
\begin{align}\label{est.RN}
\norm{ R_N}_N 
\lec \la_{q+1}^N \frac {\de_{q+1}^\frac 12}{\la_{q+1}\tau_q} , \quad 
\norm{   D_{t,q+1} R_N}_{N-1}  \lec \la_{q+1}^N\de_{q+1}^\frac12 \frac {\de_{q+1}^\frac 12}{\la_{q+1}\tau_q}.
\end{align}

\subsection{Oscillation stress error}
Recall that $ R_O =  R_{O1} +  R_{O2}$ where
\begin{align*}
 \varrho R_{O1} &= \mathcal{R} \left(\div\left(\frac{n_o\otimes n_o}{ \varrho} +  \varrho R_\ell - \de_{q+1}  \varrho \I\right) \right)\\
 \varrho R_{O2} &=  \frac{n_o\otimes n_c}{ \varrho}+\frac{n_c\otimes n_o}{ \varrho}+\frac{n_c\otimes n_c}{ \varrho}.
\end{align*}
We compute
\begin{align*}
\div \left( \frac{n_o\otimes n_o}{ \varrho} +  \varrho R_\ell - \de_{q+1}  \varrho \I\right)
&= \div \left( \sum_{u\in \Z,\, k\in \Z^3\setminus \{0\}} \de_{q+1} \frac{c_{u,k}}{\varrho} e^{ i\la_{q+1} k\cdot \xi_I} \right)\\
&=\sum_{u,k} \de_{q+1} \div\left(\frac{c_{u,k}}{\varrho}\right) e^{ i\la_{q+1} k\cdot \xi_I},
\end{align*}
because of $\dot{c}_{I,k} (f_I\cdot k) =0$. Also, since we have
\begin{align*}
D_{t,\ell} \div \frac{c_{u,k}}{\varrho} = \div\left( D_{t,\ell} \frac{c_{u,k}}{\varrho}\right) 
- \left(\na \frac{m_\ell}{\varrho}\right)_{ij}\left(\na \frac{c_{u,k}}{\varrho}\right)_{ji},
\end{align*}
it follows from \eqref{est.c} that $\norm{\div \frac{c_{u,k}}{\varrho}}_{\bar{N}}+(\la_{q+1}\de_{q+1}^\frac12)^{-1}\norm{D_{t,\ell} \div \frac{c_{u,k}}{\varrho}}_{\bar{N}} \lec_{\bar{N},M}  \mu_q^{-\bar{N}} \frac{|\dot{c}_{I,k}|}{\mu_q}$ for any $\bar{N}\geq 0$. Finally using $\supp(c_{u,k})\subset (t_u-\frac12\tau_q, t_u +\frac32\tau_q)\times \R^3$, we apply Corollary \ref{cor.mic2} to get
\begin{equation}\label{est.RO1}
\begin{split}
\norm{\varrho R_{O1}}_N
&\lec \la_{q+1}^N\cdot \frac {\de_{q+1}}{\la_{q+1}\mu_q}, \quad
\norm{  D_{t,q+1} \varrho  R_{O1}}_{N-1}\lec\la_{q+1}^N\de_{q+1}^\frac12\cdot \frac {\de_{q+1}}{\la_{q+1}\mu_q} . 
\end{split}
\end{equation}
On the other hand, we use \eqref{est.W}, \eqref{est.Wc}, \eqref{est.w}, and \eqref{est.v.dif} to estimate $ R_{O2}$ as follows,
\begin{align*}
\norm{ R_{O2}}_N
&\lec  \sum_{N_0+N_1+N_2=N} \norm{\varrho^{-1}}_{N_0}\norm{n_o}_{N_1}\norm{n_c}_{N_2} + \sum_{N_0+N_1+N_2=N}\norm{\varrho^{-1}}_{N_0}\norm{n_c}_{N_1}\norm{n_c}_{N_2}\\
&\lec \la_{q+1}^N\cdot\frac{\de_{q+1}}{\la_{q+1}\mu_q} ,\\
\norm{   D_{t,q+1} R_{O2}}_{N-1}
&\leq \norm{D_{t,\ell}  R_{O2}}_{N-1} + \Norm{\frac{(n+m_q-m_\ell)}{\varrho}\cdot\na  R_{O2}}_{N-1}\\
&\lec \norm{\varrho^{-1}}_{N} \sum_{N_1+N_2=N-1} 
(\norm{D_{t,\ell} n_o}_{N_1}\norm{n_c}_{N_2} 
+ \norm{ n_o}_{N_1}\norm{D_{t,\ell} n_c}_{N_2}
+\norm{D_{t,\ell}n_c}_{N_1}\norm{n_c}_{N_2})\\
&\quad+ \norm{\varrho^{-1}}_{N}\sum_{N_1+N_2=N-1} 
(\norm{n}_{N_1} + \norm{m_q-m_\ell}_{N_1})\norm{   R_{O2}}_{N_2+1}
\lec \la_{q+1}^{N}\de_{q+1}^\frac12 \cdot\frac{ \de_{q+1}}{\la_{q+1}\mu_q}.
\end{align*}
Therefore, we have 
\begin{equation}\label{est.RO}
\norm{ R_O}_N \lec \la_{q+1}^N \frac{\de_{q+1}}{\la_{q+1}\mu_q}, \quad
\norm{   D_{t,q+1}  R_O}_{N-1}   
\lec \la_{q+1}\de_{q+1}^{\frac12}
\cdot \la_{q+1}^N \frac{\de_{q+1}}{\la_{q+1}\mu_q}.
\end{equation} 
\subsection{Mediation stress error}
Recall that 
\[
 \varrho R_M =
\varrho(R_q-R_\ell)+\frac{(m_q-m_\ell)\otimes n}{ \varrho} + \frac{n\otimes (m_q-m_\ell)}{ \varrho}.
\]
Using \eqref{est.R.dif}, \eqref{est.v.dif}, and \eqref{est.w}, we have
\begin{align*}
\norm{ R_M}_{N} 
&\lec \norm{\varrho}_{N}\norm{R_q-R_\ell}_N+\sum_{N_0+N_1+N_2=N}\norm{\varrho^{-2}}_{N_0}\norm{m_q-m_\ell}_{N_1}\norm{n}_{N_2}\\
&\lec \la_{q+1}^N \cdot ( {\la_q^\frac12}\la_{q+1}^{-\frac12} \de_q^\frac14\de_{q+1}^\frac34  + (\ell\la_q)^2\de_q^\frac12\de_{q+1}^\frac12)
\lec \la_{q+1}^N \cdot  {\la_q^\frac12}\la_{q+1}^{-\frac12} \de_q^\frac14\de_{q+1}^\frac34.
\end{align*}
To estimate $   D_{t,q+1}   R_M$, we use the
decomposition $D_{t,q+1}   R_M  = D_{t,\ell}  R_M + {\left(\frac{m_q-m_\ell+n}{\varrho}\right)}\cdot\na  R_M $ additionally to obtain
\begin{align*}
\norm{   D_{t,q+1} (\varrho R_M)}_{N-1}
&\lec \norm{  { D_{t, \ell}(\varrho} (R_q-R_\ell))}_{N-1} 
+  \Norm{    D_{t, \ell}\left( \frac{(m_q-m_\ell)\otimes n}{\varrho}\right)}_{N-1}  \\
&\hspace{1cm}+\Norm{\left(\frac{m_q-m_\ell+n}{\varrho}\right)\cdot\na (\varrho R_M)}_{N-1}
\lec \la_{q+1}^N \de_{q+1}^\frac12\cdot  {\la_q^\frac12}\la_{q+1}^{-\frac12} \de_q^\frac14\de_{q+1}^\frac34.
\end{align*} 
To summarize, we obtain 
\begin{align}\label{est.RM}
\norm{ R_M}_{N}
\lec  \la_{q+1}^N  {\la_q^\frac12}\la_{q+1}^{-\frac12} \de_q^\frac14\de_{q+1}^\frac34, \quad
\norm{   D_{t,q+1}  R_M}_{N-1}
\lec \la_{q+1}^N \de_{q+1}^\frac12  {\la_q^\frac12}\la_{q+1}^{-\frac12} \de_q^\frac14\de_{q+1}^\frac34.
\end{align} 
 \
 
{Finally, Proposition \ref{p:Reynolds} follows from
\eqref{est.RT}-\eqref{est.RM} and \eqref{rel.par}.}

\section{Estimates for the new current}\label{sec.cur} 

In this section, we obtain the last needed estimates, on the new unsolved current $\ph_{q+1}$ and on the remaining part of the Reynolds stress $\frac{2}{3}  \zeta{/\varrho }\I$, which we
summarize in the following proposition.

\begin{prop}\label{p:current}
{There exists  $\bar{b}(\al)>1$ with the following property. For any $1<b<\bar{b}(\al)$ there is 
$\Lambda_0 = \Lambda_0 (\al,b,M, \varrho, p)$ such that} the following estimates hold for $\lambda_0 \geq \Lambda_0 $:
\begin{align}\begin{split}
\label{est.asph}
\norm{\ph_{q+1}}_N
&\leq \la_{q+1}^{N-3\ga} \de_{q+2}^\frac32 , \qquad\quad \forall N=0,1,2,\\
\norm{   D_{t,q+1} \ph_{q+1}}_{N-1} 
&\leq \la_{q+1}^{N-3\ga}\de_{q+1}^\frac12 \de_{q+2}^\frac32, \quad \forall N=1,2\, ,
\end{split}
\\
\norm{ \zeta}_0 + \norm{\zeta'}_0&\leq {\frac{\e_0^2}{20\underline{M}}}   \la_{q+1}^{-3\ga} \de_{q+2}^\frac32
\label{est.zeta}
\end{align}
{for $\underline M$ defined as in \eqref{est.vp}}.
\end{prop}

Without mentioning, we assume that $N$ is in the range above and allow the dependence on $M$ of the implicit constants in $\lec$ in this section.
For convenience, we single out the following fact, which will be repeatedly used: note that there exists $\bar{b}(\al)>1$ such that for any $1<b<\bar{b}(\al)$ and a constant ${\td C_{M, \varrho, p}}$ depending only on $\varrho$, ${p}$, and $M$, we can find $\La_0=\La_0(\al, b,M,\varrho,p)$ which gives
\[
{\td C_{M, \varrho, p}}\left[ \frac {\de_{q+1}}{\la_{q+1}\tau_q} +\frac{\de_{q+1}^\frac32}{\la_{q+1}\mu_q} + \frac {\la_q^\frac12}{\la_{q+1}^\frac12} \de_q^\frac14\de_{q+1}^\frac54\right]
\leq  \la_{q+1}^{-3\ga} \de_{q+2}^\frac32,
\] 
for any $\lambda_0 \geq \La_0$. {This is possible because $\al<\frac17$.}

\subsection{High frequency current error}
We start by observing that $  \ph_{H1}$ is
 \begin{equation}\label{e:formula_H1}
   \varrho\ph_{H1} = \mathcal{R} \left(\frac n{\varrho}\cdot (\div P_{\le \ell^{-1}}(\varrho(R_q-c_q \I)) + Q(m_q,m_q))
   \right)
 \end{equation}
 by \eqref{e:average-free}. We thus can apply Corollary \ref{cor.mic2} to
\begin{align*}
\frac n{\varrho}\cdot & (\div P_{\le \ell^{-1}}(\varrho(R_q-c_q \I)) + Q(m_q,m_q))
\\
&= \sum_{\substack{u,k\\k \neq 0}} \frac1\varrho (\div P_{\le \ell^{-1}}(\varrho(R_q-c_q \I)) + Q(m_q,m_q)) \de_{q+1}^\frac12(b_{u,k} + (\la_{q+1}\mu_q)^{-1}e_{u,k})e^{i\la_{q+1}k\cdot \xi_I}.
\end{align*}
Indeed, using \eqref{est.R}, \eqref{est.Qvv}, \eqref{est.b}, \eqref{est.e}, we obtain
\begin{align*}
&\norm{  \ph_{H1} }_N
\lec \la_{q+1}^{N-1} \de_{q+1}^\frac12 (\la_q^{1-3\ga}\de_{q+1} + (\ell \la_q) \la_q\de_q) 
\lec \la_{q+1}^N \frac{\la_q^{1-3\ga}}{\la_{q+1}} \de_{q+1}^\frac32.
\end{align*}
Furthermore, \eqref{est.vp}, \eqref{est.R}, \eqref{est.v.dif}, \eqref{est.com1}, and {\eqref{est.com3}} imply
\begin{align*}
&\norm{D_{t,\ell} \div P_{\le \ell^{-1}} (\varrho R_q)}_{N-1}\\
&\leq \norm{\div P_{\le \ell^{-1}}  (D_{t,\ell} \varrho R_q)}_{N-1}
+ \norm{\div [m_\ell P_{\leq \ell^{-1}}\varrho^{-1} \cdot\na, P_{\le \ell^{-1}}] \varrho R_q}_{N-1}\\
&\quad+ \norm{\div [m_\ell P_{> \ell^{-1}}\varrho^{-1} \cdot\na, P_{\le \ell^{-1}}] \varrho R_q}_{N-1} + \norm{(\na (m_\ell/\varrho))_{ki} \pa_k P_{\le \ell^{-1}} (\varrho R_q)_{ij}}_{N-1}\\
&\lec \la_{q+1}^{N-1} (\norm{D_{t,\ell} (\varrho R_q)}_{1}
+ \norm{ [m_\ell P_{\leq \ell^{-1}}\varrho^{-1} \cdot\na, P_{\le \ell^{-1}}] \varrho R_q}_{1}) \\
&\quad + \norm{ [m_\ell P_{> \ell^{-1}}\varrho^{-1} \cdot\na, P_{\le \ell^{-1}}] \varrho R_q}_{N}
+\sum_{N_1+N_2=N-1}\norm{m_\ell}_{N_1{+1}} \ell^{-N_2-1}\norm{\varrho R_q}_0
\\
&\lec \la_{q+1}^{N}\de_{q+1}^\frac12 \la_q^{1-3\ga} \de_{q+1}.
\end{align*}
In the last inequality, we used
\begin{align*}
\norm{m_\ell P_{>\ell^{-1}}\varrho^{-1}}_{N'}
&\lec 
\norm{m_\ell}_{0}\norm{P_{>\ell}\varrho^{-1}}_{N'}
+\sum_{\substack{N_1+ N_2= N'\\N_1\geq 1}} \norm{m_\ell}_{N_1}\norm{P_{>\ell}\varrho^{-1}}_{N_2}\\
&\lec 1 + \sum_{\substack{N_1+ N_2= N'\\N_1\geq 1}} \de_q^\frac12  \norm{\na^{N_1}\varrho^{-1}}_{N_2} 
{\lec 1}
\end{align*}
for $N'\in [1, n_0+1]$. 
In a similar way, we also get $\norm{D_{t,\ell} \div P_{\le \ell^{-1}} (\varrho c_q \I)}_{N-1}{ \lec \la_{q+1}^N \de_{q+1}^\frac12 \la_q^{1-3\ga}\de_{q+1}}$. Then, {using \eqref{est.Qvv}}, it follows that
\begin{align*}
\norm{   D_{t,q+1} \ph_{H1}}_{N-1}
\lec\la_{q+1}^N\de_{q+1}^\frac12\cdot \frac{\la_q^{1-3\ga}}{\la_{q+1}} \de_{q+1}^\frac32.
\end{align*}
In order to deal with $  \ph_{H2}$, we use the definition of $R_{q+1}$ to get
\begin{equation}\begin{split}
\label{rep:low.freq.app}
\frac{n\otimes n}{\varrho}&-\de_{q+1}\varrho \I+ \varrho R_q - \varrho R_{q+1} + {\textstyle{\frac{2}{3}}} \zeta \I\\
&=
\left(\frac{n_o\otimes n_o}{\varrho}-\de_{q+1}\varrho \I+\varrho R_\ell\right )-  \varrho R_{O1}  - \varrho R_T -\varrho R_N-\varrho R_{M2} \, .
\end{split}\end{equation}
Thus, we can write
\begin{align}
  \varrho \ph_{H2} &=\mathcal{R} \left(\left(\frac{n\otimes n}{\varrho}-\de_{q+1}\varrho\I+\varrho R_q-\varrho R_{q+1}+\frac23\zeta \I\right) : \na \frac{m_\ell}{\varrho}
  \right) \nonumber\\ 
  &\quad + \mathcal{R} \left(\left(\frac{(m_q-m_\ell)\otimes n}{\varrho}+ \frac{n\otimes (m_q-m_\ell)}{\varrho} \right): \na \frac{m_\ell}{\varrho} 
  \right) -  \frac{2m_\ell\zeta}{3\varrho} \label{e:formula_H2} \\ 
&= \idv{\left(\left(\frac{n_o\otimes n_o}{\varrho}-\de_{q+1}\varrho \I+\varrho R_\ell \right)-  \varrho R_{O1}  -\varrho R_T -\varrho R_N\right):\na \frac{m_\ell}{\varrho}}-  \frac{2m_\ell\zeta}{3\varrho}
.\nonumber
\end{align}
To estimate the term with $\zeta$ first, assuming {\eqref{est.zeta}}, we have
\begin{align*}
    \norm{(m_\ell \zeta)/\varrho^2}_0
    &\leq \norm{\zeta}_N \norm{m_\ell/\varrho^2}_N
    \leq \frac 1{10}\la_{q+1}^{N-3\ga} \de_{q+2}^\frac32   
\end{align*}
for $N=0,1,2$ and
\begin{align*}
    \norm{D_{t,q+1} ((m_\ell \zeta)/\varrho^2) }_{N-1}
    &\leq (\norm{(D_{t,q+1} m_\ell) \varrho^{-2}}_{N-1}
    +\norm{m_\ell D_{t,q+1} \varrho^{-2}}_{N-1} )\norm{\zeta}_0
    +\norm{\zeta'}_0\norm{m_\ell/\varrho^2}_{N-1}\\
    &\leq \frac 1{10} \la_{q+1}^{N-3\ga}\de_{q+1}^\frac12 \de_{q+2}^\frac32
\end{align*}
for $N=1,2$. In the last inequality, we used {
\begin{align}\label{nmtdv.ml}
\norm{D_{t,q+1} m_\ell}_{N-1}
\lec {\la_{q+1}^{N-1}\de_{q+1}^\frac12\la_q\de_q^\frac12 + }(\la_q^N + \ell^{-N}(\ell\la_q)^3)\de_q,    
\end{align}
obtained from \eqref{adv.ml0}, \eqref{adv.mlN}, \eqref{diff.m}, and {\eqref{est.w.indM}}} and
\begin{equation}\label{est.mtdv.rho}
    \begin{split}
      &\norm{D_{t,q+1}\varrho}_0 
    \lec  \norm{\pa_t \varrho}_0 + \norm{m_{q+1}/\varrho}_0 \norm{\na\varrho}_0
    \lec 1\\
    &\norm{D_{t,q+1}\varrho}_N
    \lec  \norm{\pa_t \varrho}_N + \sum_{N_1+N_2=N}\norm{m_{q+1}/\varrho}_{N_1} \norm{\na\varrho}_{N_2}
    \lec \la_{q+1}^{N}\de_{q+1}^\frac12, \hspace{1cm} \forall N=1,2.  
    \end{split}
\end{equation}
Apply, on the other hand, Corollary \ref{cor.mic2} with \eqref{alg.eq}, \eqref{est.c}, and \eqref{est.vp}, we have
\begin{align*}
\Norm{\idv{\left(\frac{n_o\otimes n_o}{\varrho}-\de_{q+1}\varrho \I+\varrho R_\ell \right):\na \frac{m_\ell}{\varrho}}}_N
&\lec \la_{q+1}^N \la_q \de_q^\frac12 \frac {\de_{q+1}}{\la_{q+1}},\\
\Norm{   D_{t,q+1} \idv{\left(\frac{n_o\otimes n_o}{\varrho}-\de_{q+1}\varrho \I+\varrho R_\ell \right):\na \frac{m_\ell}{\varrho}}}_{N-1}
& \lec  \la_{q+1}^N\de_{q+1}^\frac12 \la_q \de_q^\frac12 \frac {\de_{q+1}}{\la_{q+1}}.
\end{align*}

To estimate the remaining term, by \eqref{dec.mic} and \eqref{rem.low}, recall that the Reynolds stress errors $ R_\tri$, which represents either $ R_{O1}$, $ R_T$, or $ R_N$, can be written as $ R_\tri  = \cR G_\tri$ satisfying 
\begin{align}\label{est.Gtr} 
\norm{G_\tri}_N \lec \la_{q+1}^N \left(\frac{\de_{q+1}}{\mu_q} + \frac{\de_{q+1}^\frac12}{\tau_q} \right), 
\quad
\norm{D_{t,\ell} G_\tri}_{N-1} \lec \la_{q+1}^N \de_{q+1}^\frac12
\left(\frac{\de_{q+1}}{\mu_q} + \frac{\de_{q+1}^\frac12}{\tau_q} \right) 
\end{align}
Furthermore, such $G_\tri$ has the form $\sum_{u,k}g_\tri^{u,k} e^{i\la_{q+1}k\cdot\xi_I}$ and has a decomposition 
\begin{equation}\label{dec.Gtri}
G_\tri = \cP_{\gtrsim \la_{q+1}} G_\tri+\cP_{\lec \la_{q+1}} G_\tri\, ,    
\end{equation}
as in \eqref{dec.mic} and \eqref{rem.low}, where $\cP_{\lec \la_{q+1}} G_\tri$ satisfies
\begin{equation}\begin{split}
\norm{\cP_{\lec \la_{q+1}}G_\tri}_0 &\lec \la_{q+1}^{-2}\de_{q+1}^\frac12 \left(\frac{\de_{q+1}}{\mu_q} + \frac{\de_{q+1}^\frac12}{\tau_q} \right),\\ 
\norm{D_{t,\ell}\cP_{\lec \la_{q+1}}G_\tri}_0 &\lec \la_{q+1}^{-1}\de_{q+1} \left(\frac{\de_{q+1}}{\mu_q} + \frac{\de_{q+1}^\frac12}{\tau_q} \right). \label{est.DtGtr}
\end{split}\end{equation}
Indeed, they follow from \eqref{est.eN} and \eqref{est.DteN}.
Now, we write 
\begin{align*}
    \na \frac{m_\ell}{\varrho} &= \na \left( m_\ell \cP_{\lec \frac1{64}\la_{q+1}} \frac{1}{\varrho}\right) + \na \left( m_\ell \cP_{\gtrsim \frac1{64}\la_{q+1}} \frac{1}{\varrho}\right)\\&=: \na \left(m_\ell / \varrho\right)_1 + \na \left(m_\ell / \varrho\right)_2
\end{align*}
 Now $\na \left(m_\ell / \varrho\right)_1$ has the frequency localized to ${\leq} \frac1{32} \la_{q+1}$ since $m_\ell$ had frequency localized to $\ell^{-1}$ and $\ell^{-1} \leq \frac 1{128}\la_{q+1}$ for sufficiently large $\la_0$. Thus, {$\cR \cP_{\gtrsim \la_{q+1}} G_\tri:\na \left(m_\ell / \varrho\right)_1$ has the frequency localized to $\gtrsim \la_{q+1}$ and}
\begin{equation}\begin{split} \label{est.high.1}
\norm{\idv{\cR \cP_{\gtrsim \la_{q+1}} G_\tri:\na \left(m_\ell / \varrho\right)_1}}_N
&\lec \frac 1{\la_{q+1}} \norm{\cR  \cP_{\gtrsim \la_{q+1}} G_\tri:\na \left(m_\ell / \varrho\right)_1}_N\\
&\lec \frac 1{\la_{q+1}^2} \sum_{N_1+N_2=N} \norm{ G_\tri}_{N_1} \norm{\na \frac{m_\ell}{\varrho}}_{N_2}
\end{split}\end{equation}
On the other hand, $\cR \cP_{\lec \la_{q+1}} G_\tri:\na \left(m_\ell / \varrho\right)_1$ has the frequency localized to $\lec \la_{q+1}^{-1}$, so that
\begin{equation}\begin{split} \label{est.low.1}
\norm{\idv{\cR \cP_{\lec \la_{q+1}} G_\tri:\na \left(m_\ell / \varrho\right)_1}}_N
&\lec \la_{q+1}^N \norm{\idv{\cR \cP_{\lec \la_{q+1}} G_\tri:\na \left(m_\ell / \varrho\right)_1}}_0
\\
&\lec \la_{q+1}^N \norm{\cP_{\lec \la_{q+1}} G_\tri}_0\norm{\na \frac{m_\ell}{\varrho}}_0.
\end{split}\end{equation}
Observe that since $\varrho$ is smooth in space-time and bounded below by a positive constant $\e_0$, we have by Bernstein's inequality that 
\begin{equation}\label{est.varrho_smooth}
    \norm{P_{\gtrsim \la} \varrho^{-1}}_{N'}+ \norm{\pa_t P_{\gtrsim \la} \varrho^{-1}}_{N'} 
    \lec \la^{-2}(\norm{\na^2 \varrho^{-1}}_{N'}+ \norm{\na^2\pa_t  \varrho^{-1}}_{N'})
    \lec \la^{-2} 
\end{equation}
for any $\la \geq 1$ and for any $N'=0,1,2,3$. So we see $\norm{(m_\ell/\varrho)_2}_{N'} \lesssim \frac1{\la_{q+1}^2} \norm{m_\ell}_{N'}$ for $N'=0,1,2,3$. It then follows that 
\begin{equation}\begin{split} \label{est.low.2}
\norm{\idv{\cR G_\tri:\na \left(m_\ell / \varrho\right)_2}}_N 
&\lec \sum_{N_1+N_2=N} \norm{ G_\tri}_{N_1} \norm{\na \left(m_\ell / \varrho\right)_2}_{N_2}\\
&\lec \frac1{\la_{q+1}^{2}} \sum_{N_1+N_2=N} \norm{ G_\tri}_{N_1} \norm{ m_\ell}_{N_2+1}
\end{split}\end{equation}
Therefore, using \eqref{est.Gtr} and \eqref{est.DtGtr}, we obtain
\begin{align*}
\Norm{\cal{R}( R_{O1}:\na \frac{m_\ell}{\varrho})}_N+\Norm{\cal{R}( R_T:\na \frac{m_\ell}{\varrho})}_N
&+\Norm{\cal{R}( R_N:\na \frac{m_\ell}{\varrho})}_N\\
&\lec\la_{q+1}^N
\left(\frac {\de_{q+1}^\frac12}{\la_{q+1}\tau_{q}}+ \frac {\de_{q+1}}{\la_{q+1}\mu_{q}}\right) \frac{\la_q\de_q^\frac12}{\la_{q+1}}.
\end{align*}
To estimate their advective derivatives, consider the decomposition
\begin{align*}
   D_{t,q+1} \idv{ R_{\triangle}:\na \frac{m_\ell}{\varrho}} 
= D_{t,\ell}\idv{ R_{\triangle}:\na \frac{m_\ell}{\varrho}} + \left(\frac{n+ (m_q-m_\ell)}{\varrho}\cdot \na \right)\idv{ R_{\triangle}:\na \frac{m_\ell}{\varrho}}. 
\end{align*}
We can easily see that 
\begin{align*}
\Norm{\frac{n+ (m_q-m_\ell)}{\varrho}\cdot \na \idv{ R_{\triangle}:\frac{m_\ell}{\varrho}}}_{N-1}
&\lec  \la_{q+1}^N\de_{q+1}^\frac12
\left(\frac {\de_{q+1}^\frac12}{\la_{q+1}\tau_{q}}+ \frac {\de_{q+1}}{\la_{q+1}\mu_{q}}\right) \frac{\la_q\de_q^\frac12}{\la_{q+1}}.
\end{align*}
As for the first term, we use again the decompositions $ R_\tri = \cR \cP_{\gtrsim \la_{q+1}} G_\tri + \cR \cP_{\lec \la_{q+1}} G_\tri$ and $\na \frac{m_\ell}{\varrho} = \na \left(m_\ell / \varrho\right)_1 + \na \left(m_\ell / \varrho\right)_2$ and consider
\begin{align*}
\norm{D_{t,\ell}\cal{R}( R_{\triangle}:\na \left(m_\ell / \varrho\right)_1)}_{N-1}
&\leq 
\norm{D_{t,\ell}\cal{R}(\cR \cP_{\gtrsim \la_{q+1}} G_\tri :\na \left(m_\ell / \varrho\right)_1)}_{N-1}\\
&\quad+\norm{D_{t,\ell}\cal{R}( \cR \cP_{\lec \la_{q+1}} G_\tri:\na \left(m_\ell / \varrho\right)_1)}_{N-1}.
\end{align*}
In order to estimate the first term on the right hand side, consider the decomposition 
\begin{align*}
D_{t,\ell}\cR P_{\gtrsim \la_{q+1}} H
&=\cR {P}_{\gtrsim \la_{q+1}} D_{t,\ell}  H 
+\cR \left[\frac{m_\ell}{\varrho}\cdot\na, P_{\gtrsim \la_{q+1}}\right]  H 
+ \left[\frac{m_\ell}{\varrho}\cdot\na, \cR\right] P_{\gtrsim \la_{q+1}} H,
\end{align*}
for any smooth {vector-valued} function $H$ and Littlewood-Paley operator $P_{\gtrsim\la_{q+1}}$ projecting to the frequency $\gtrsim \la_{q+1}$. 
{Since Lemmas \ref{lem:com2} is still valid when $\ell^{-1}$ is replaced by $C\la_{q+1}$, additionally using \eqref{est.varrho_smooth},} we have
\begin{align*}
\norm{[(m_\ell/\varrho)_1 \cdot\na,P_{\gtrsim \la_{q+1}}] H}_{N-1}
&{\lec \la_{q+1}^{N-2} \norm{\na(m_\ell/\varrho)_1}_0 \norm{\na H}_0}
\leq \la_{q+1}^{N-2} \norm{m_q}_1 \norm{\na H}_0\\
\norm{[(m_\ell/\varrho)_2\cdot\na,P_{\gtrsim \la_{q+1}}] H}_{N-1}
&{\lec \la_{q+1}^{N-2} \norm{\na(m_\ell/\varrho)_2}_{1} \norm{\na H}_0
\lec \la_{q+1}^{N-4}} \norm{m_q}_1 \norm{\na H}_0
\end{align*}
Also, by Lemma \ref{lem:com*} we obtain
{
\begin{align}\label{est1}
\Norm{\left[{m_\ell}{P_{\leq \ell^{-1}}\varrho^{-1}}\cdot\na, \cR\right]P_{\gtrsim \la_{q+1}} H}_{N-1} 
\lec \sum_{N_1+N_2=N-1} \ell \Norm{\na ({m_\ell}{P_{\leq \ell^{-1}}\varrho^{-1}})}_{N_1} \norm{H}_{N_2} ,
\end{align}
and using $\norm{\mathcal{R}f}_0\lec \norm{f}_0$ and $\norm{P_{>\ell^{-1}}\varrho^{-1}}_{N'}\lec \ell^3 \norm{\na^3 \varrho^{-1}}_{N'}$ we get
\begin{align*}
    \Norm{\left[{m_\ell}{P_{> \ell^{-1}}\varrho^{-1}}\cdot\na, \cR\right]P_{\gtrsim \la_{q+1}} H}_{N-1} 
&\lec \sum_{N_1+N_2+N_3=N-1}
\norm{m_{\ell}}_{N_1}
\norm{P_{>\ell^{-1}}\varrho^{-1}}_{N_2}\norm{\na P_{\gtrsim \la_{q+1}}H}_{N_3}\\
&\lec \ell^3 \sum_{N_1+N_3=N-1}
\norm{m_{\ell}}_{N_1}
\norm{\na P_{\gtrsim \la_{q+1}}H}_{N_3}.
\end{align*}
}
Since $P_{\gtrsim \la_{q+1}} D_{t,\ell}  H$ and $ [(m_\ell/\varrho)_1\cdot\na, P_{\gtrsim \la_{q+1}}]  H$ have frequencies localized to $\gtrsim \la_{q+1}$, it follows that
\begin{equation}\begin{split}\label{est.H}
\norm{&D_{t,\ell}\cR P_{\gtrsim \la_{q+1}} H}_{N-1}\\
&\lec \norm{\cR P_{\gtrsim \la_{q+1}} D_{t,\ell}  H }_{N-1} 
+\norm{\cR [(m_\ell/\varrho)_1\cdot\na, P_{\gtrsim \la_{q+1}}]  H }_{N-1}
\\
&\quad+\norm{\cR [(m_\ell/\varrho)_2\cdot\na, P_{\gtrsim \la_{q+1}}]  H }_{N-1}+\norm{ [\frac{m_\ell}{\varrho}\cdot\na, \cR] P_{\gtrsim\la_{q+1}} H}_{N-1}\\
&\lec 
\frac 1{\la_{q+1}} \norm{ P_{\gtrsim \la_{q+1}} D_{t,\ell}  H }_{N-1} 
+  \frac 1{\la_{q+1}}\norm{ [(m_\ell/\varrho)_1\cdot\na, P_{\gtrsim \la_{q+1}}]  H }_{N-1}\\
&\quad {+\norm{ [(m_\ell/\varrho)_2\cdot\na, P_{\gtrsim \la_{q+1}}]  H }_{N-1}}
+\norm{ [{m_\ell}/{\varrho}\cdot\na, \cR] P_{\gtrsim \la_{q+1}} H}_{N-1}\\
&\lec \frac 1{\la_{q+1}} \norm{ D_{t,\ell}  H }_{N-1} 
+ \la_{q+1}^{N-3}\norm{m_q}_1  \norm{\na H}_0
+\sum_{N_1+N_2=N-1} \ell \norm{\na \frac{m_\ell}{\varrho}}_{N_1}\norm{H}_{N_2}.
\end{split}\end{equation}
Now, we apply it to $H= \cR \cP_{\gtrsim \la_{q+1}} G_\tri:\na \left(m_\ell / \varrho\right)_1$. For such $H$, we have $H= P_{\geq \frac18\la_{q+1}}H$ for sufficiently large $\la_0$, so that
\begin{equation}\begin{split}\label{est.HL}
&\norm{D_{t,\ell}\cR(\cR  \cP_{\gtrsim \la_{q+1}}  G_\tri:\na \left(m_\ell / \varrho\right)_1)}_{N-1}\\
&\lec \frac 1{\la_{q+1}}\norm{D_{t,\ell}(\cR  \cP_{\gtrsim \la_{q+1}}  G_\tri: \na \left(m_\ell / \varrho\right)_1)}_{N-1} 
+ \la_{q+1}^{N-3}\norm{ m_q}_1 
\norm{ \cR \cP_{\gtrsim \la_{q+1}} G_\tri:\na \left(m_\ell / \varrho\right)_1}_1\\
&\quad+\sum_{N_1+N_2=N-1} \ell\Norm{\na \frac{m_\ell}{\varrho}}_{N_1}\norm{\cR \cP_{\gtrsim \la_{q+1}} G_\tri:\na \left(m_\ell / \varrho\right)_1}_{N_2}\\
&\lec \la_{q+1}^{N-2} \de_{q+1}^\frac12
\left(\frac{\de_{q+1}}{\mu_q} + \frac{\de_{q+1}^\frac12}{\tau_q} \right) \la_q\de_q^\frac12.
\end{split}\end{equation}
Indeed, the second inequality can be obtained by applying \eqref{est.H} again to $H= G_\tri$,
\begin{align*}
&\norm{D_{t,\ell}(\cR \cP_{\gtrsim \la_{q+1}} G_\tri: \na \left(m_\ell / \varrho\right)_1)}_{N-1} \\
&\lec \sum_{N_1+N_2=N-1}
\norm{D_{t,\ell}\cR  \cP_{\gtrsim \la_{q+1}}  G_\tri}_{N_1} \norm{\na (m_\ell / \varrho)}_{N_2}
+\norm{\cR  \cP_{\gtrsim \la_{q+1}}  G_\tri}_{N_1}\norm{ D_{t,\ell}\na (m_\ell / \varrho)}_{N_2}\\
&\lec \sum_{N_1+N_2=N-1}
\left(\frac1{\la_{q+1}}\norm{D_{t,\ell}G_\tri}_{N_1}
+ \la_{q+1}^{N_1-2}\norm{m_q}_1\norm{\na G_\tri}_0 \right)\norm{\na (m_\ell / \varrho)}_{N_2}\\
&\quad +\sum_{\substack{N_{11}+N_{12} = N_1\\ N_1+N_2=N-1}} \ell \norm{\na m_\ell}_{N_{11}}\norm{G_\tri}_{N_{12}}\norm{\na (m_\ell / \varrho)}_{N_2}
+\sum_{N_1+N_2=N-1}\frac{\norm{ G_\tri}_{N_1}}{\la_{q+1}}\norm{ D_{t,\ell}\na (m_\ell / \varrho)}_{N_2}\\
&\lec \la_{q+1}^{N-1} \de_{q+1}^\frac12
\left(\frac{\de_{q+1}}{\mu_q} + \frac{\de_{q+1}^\frac12}{\tau_q} \right) \la_q\de_q^\frac12,
\end{align*}
{where we used 
\begin{equation}\begin{split}\label{est.Dtvl2}
    \norm{D_{t,\ell} \na (m_\ell/\varrho)}_{N_2}
    &\leq
    \norm{\varrho^{-1}D_{t,\ell} \na m_\ell}_{N_2}
    +\norm{D_{t,\ell} m_\ell \otimes\na\varrho^{-1}}_{N_2}
    \\
    &\hspace{1.5cm}+\norm{ \na m_\ell D_{t,\ell}\varrho^{-1}}_{N_2}+\norm{m_\ell\otimes D_{t,\ell}{\na}\varrho^{-1}}_{N_2}
    \lec \la_{q+1}^{ 
    N_2+1}\de_{q+1}^\frac12 \la_q\de_q^\frac12,
\end{split}
\end{equation}
which follows from \eqref{adv.mlN} and \eqref{est.Dtvl}.}
Also, we get
\begin{align*}
\norm{\cR \cP_{\gtrsim \la_{q+1}} G_\tri:\na \left(m_\ell / \varrho\right)_1)}_{N-1}
&\lec \sum_{N_{1}+N_{2} = N-1}
\norm{\cR \cP_{\gtrsim \la_{q+1}} G_\tri}_{N_{1}} \norm{\na \left(m_\ell / \varrho\right)}_{N_{2}}\\
&\lec \la_{q+1}^{N-2} \left(\frac{\de_{q+1}}{\mu_q} + \frac{\de_{q+1}^\frac12}{\tau_q} \right) \la_q\de_q^\frac12.
\end{align*}
As for the remaining term $\norm{D_{t,\ell}\cal{R}( \cR \cP_{\lec \la_{q+1}} G_\tri:\na \left(m_\ell / \varrho\right)_1))}_{N-1}$, 
we set $
D_{t, \ell}^{L} := \pa_t + (m_\ell/\varrho)_1 \cdot \na$
and estimate 
\begin{align*}
\norm{D_{t, \ell}^{L} &\cR(\cR \cP_{\lec \la_{q+1}}G_\tri:\na \left(m_\ell / \varrho\right)_1))}_0
\\
&\lec\norm{ \cR D_{t, \ell}^{L}(\cR \cP_{\lec \la_{q+1}} G_\tri:\na \left(m_\ell / \varrho\right)_1))}_0
+
 \norm{[(m_\ell/\varrho)_1\cdot\na ,\cR] (\cR \cP_{\lec \la_{q+1}} G_\tri:\na \left(m_\ell / \varrho\right)_1))}_0\\
&\lec  \norm{ D_{t, \ell}^{L}\cR \cP_{\lec \la_{q+1}} G_\tri}_0\norm{\na (m_\ell / \varrho)_1}_0 + \norm{  \cP_{\lec \la_{q+1}}G_\tri}_0\norm{ D_{t, \ell}^{L}\na (m_\ell / \varrho)_1}_0\\
&\qquad\qquad +\norm{(m_\ell/\varrho)_1}_0\norm{\na(\cR \cP_{\lec \la_{q+1}} G_\tri:\na (m_\ell/\varrho)_1)}_0 \\
&\lec ( \norm{ D_{t, \ell}^{L}\cP_{\lec \la_{q+1}} G_\tri}_0
+ \norm{[(m_\ell/\varrho)_1\cdot \na, \cR] \cP_{\lec \la_{q+1}} G_\tri}_0)\norm{\na (m_\ell/\varrho)_1}_0\\
&\quad+
\norm{  \cP_{\lec \la_{q+1}} G_\tri}_0\norm{ D_{t,\ell}\na (m_\ell/\varrho)_1}_0
+\la_{q+1}\norm{\cP_{\lec \la_{q+1}} G_\tri}_0\norm{\na (m_\ell/\varrho)_1}_0\\
&\lec
\la_q\de_q^\frac12(\norm{ D_{t, \ell}^{L}\cP_{\lec \la_{q+1}} G_\tri}_0
+\la_{q+1}\norm{\cP_{\lec \la_{q+1}} G_\tri}_0)
\lec \la_{q+1}^{-1}\de_{q+1}^\frac12 \left(\frac {\de_{q+1}^\frac12}{\tau_{q}}+ \frac {\de_{q+1}}{\mu_{q}}\right) \la_q\de_q^\frac12.
\end{align*}
Here, we used $\norm{\cR g}_0 \lec \norm{g}_0$. 
Now observe that the frequency of $D_{t, \ell}^{L} \cR(\cR \cP_{\lec \la_{q+1}} G_\tri:\na \left(m_\ell / \varrho\right)_1))$ is  localized to $\lec \la_{q+1}$, so that 
\begin{align*}
\norm{&D_{t, \ell}^{L} \cR(\cR \cP_{\lec \la_{q+1}} G_\tri:\na \left(m_\ell / \varrho\right)_1))}_{N-1}
\lec \la_{q+1}^{N-1} \norm{D_{t, \ell}^{L} \cR(\cR \cP_{\lec \la_{q+1}} G_\tri:\na \left(m_\ell / \varrho\right)_1))}_0.
\end{align*}
Thus we see
\begin{align*}
\norm{&D_{t,\ell} \cR(\cR \cP_{\lec \la_{q+1}} G_\tri:\na \left(m_\ell / \varrho\right)_1))}_{N-1}\\
&\le \norm{D_{t, \ell}^{L} \cR(\cR \cP_{\lec \la_{q+1}} G_\tri:\na \left(m_\ell / \varrho\right)_1))}_{N-1} + \norm{((m_\ell/\varrho)_2 \cdot \na) \cR(\cR \cP_{\lec \la_{q+1}} G_\tri:\na \left(m_\ell / \varrho\right)_1))}_{N-1}\\
&\lec \la_{q+1}^{N-1} \norm{D_{t, \ell}^{L} \cR(\cR \cP_{\lec \la_{q+1}} G_\tri:\na \left(m_\ell / \varrho\right)_1))}_0 + \la_{q+1}^{-2} \norm{\cR(\cR \cP_{\lec \la_{q+1}} G_\tri:\na \left(m_\ell / \varrho\right)_1))}_{N}\\
&\lec \la_{q+1}^N\de_{q+1}^\frac12
\left(\frac {\de_{q+1}^\frac12}{\la_{q+1}\tau_{q}}+ \frac {\de_{q+1}}{\la_{q+1}\mu_{q}}\right) \frac{\la_q\de_q^\frac12}{\la_{q+1}}
\end{align*}
As a result, combining with \eqref{est.HL}, we obtain
\begin{align*}
\norm{D_{t,\ell}\cR( R_\triangle:\na \left(m_\ell / \varrho\right)_1))}_{N-1}
\lec \la_{q+1}^N\de_{q+1}^\frac12
\left(\frac {\de_{q+1}^\frac12}{\la_{q+1}\tau_{q}}+ \frac {\de_{q+1}}{\la_{q+1}\mu_{q}}\right) \frac{\la_q\de_q^\frac12}{\la_{q+1}},
\end{align*}
Now we need to estimate $D_{t,\ell}\cR( R_\triangle:\na \left(m_\ell / \varrho\right)_2)$. Consider the decomposition
\begin{align*}
D_{t,\ell}\cR H
&=\cR D_{t,\ell}  H  
+\left [\frac{m_\ell}{\varrho}\cdot\na, \cR\right] H,
\end{align*}
We observe that by setting $H =  R_\triangle:\na \left(m_\ell / \varrho\right)_2$, we get
\begin{align*}
\norm{D_{t,\ell}\cR H}_{N-1}
&=\norm{\cR D_{t,\ell} ( R_\triangle:\na \left(m_\ell / \varrho\right)_2)}_{N-1} 
+ \norm{[\frac{m_\ell}{\varrho}\cdot\na, \cR] ( R_\triangle:\na \left(m_\ell / \varrho\right)_2)}_{N-1}\\
&\lesssim \norm{ D_{t,\ell}  ( R_\triangle:\na \left(m_\ell / \varrho\right)_2) }_{N-1} 
+ \norm{[\frac{m_\ell}{\varrho}\cdot\na, \cR] ( R_\triangle:\na \left(m_\ell / \varrho\right)_2)}_{N-1}
\end{align*}
Thus we get, using \eqref{est.varrho_smooth},
\begin{align*}
    \norm{D_{t,\ell}( R_\triangle:\na \left(m_\ell / \varrho\right)_2))}_{N-1} &\lesssim \sum_{\substack{N_1 + N_2\\ = N-1}} \norm{D_{t,\ell} R_\triangle}_{N_1} \norm{\na \left(m_\ell / \varrho\right)_2}_{N_2} + \norm{ R_\triangle}_{N_1} \norm{D_{t,\ell}\na \left(m_\ell / \varrho\right)_2}_{N_2}\\
    &\lesssim \la_{q+1}^N \de_{q+1} \frac1{\la_{q+1}^2 \tau_q} \la_q \de_q^{\frac12}.
\end{align*}
{Here, we used the estimate $\norm{D_{t,\ell}\na \left(m_\ell / \varrho\right)_2}_{N_2} \lec \de_{q+1}^\frac12 \la_q \de_q^\frac12$,
obtained similar to \eqref{est.Dtvl2} but additionally using \eqref{est.varrho_smooth}.  } The remaining term can be estimated as
\begin{align*}
    &\Norm{\left[\frac{m_\ell}{\varrho}\cdot\na, \cR\right]( R_\triangle:\na \left(m_\ell / \varrho\right)_2))}_{N-1}\\
    &\lesssim \sum_{N_1 + N_2 + N_3 = N-1} \Norm{\frac{m_\ell}{\varrho}}_{N_1} \norm{ R_\triangle}_{N_2+1} \norm{\na \left(m_\ell / \varrho\right)_2}_{N_3}  + \norm{\frac{m_\ell}{\varrho}}_{N_1} \norm{ R_\triangle}_{N_2} \Norm{\na \left(m_\ell / \varrho\right)_2}_{N_3+1}\\
    &\lesssim \la_{q+1}^N \de_{q+1} \frac{\de_q^{\frac12}}{\la_{q+1}^2 \tau_q}.
\end{align*}
We therefore get
\begin{align*}
\norm{   D_{t,q+1} \cal{R}( R_\tri:\na \frac{m_\ell}{\varrho})}_{N-1}
\lec
\la_{q+1}^N\de_{q+1}^\frac12
\left(\frac {\de_{q+1}^\frac12}{\la_{q+1}\tau_{q}}+ \frac {\de_{q+1}}{\la_{q+1}\mu_{q}}\right) \frac{\la_q\de_q^\frac12}{\la_{q+1}},
\end{align*}
and the estimates for $  \ph_{H2}$ follow, 
\begin{align*}
\norm{  \ph_{H2}}_N 
\lec \la_{q+1}^N \la_q \de_q^\frac12 \frac {\de_{q+1}}{\la_{q+1}},  \quad
\norm{   D_{t,q+1}  \ph_{H2}}_{N-1} 
\lec\la_{q+1}^N\de_{q+1}^\frac12 \la_q \de_q^\frac12 \frac {\de_{q+1}}{\la_{q+1}}.
\end{align*}
To summarize, we get
\begin{align*}
\norm{  \ph_{H}}_N 
\leq \frac 15\la_{q+1}^{N-3\ga} \de_{q+2}^\frac32 ,  \quad
\norm{   D_{t,q+1}  \ph_{H}}_{N-1} 
\leq \frac 15\la_{q+1}^{N-3\ga}\de_{q+1}^\frac12 \de_{q+2}^\frac32\, .
\end{align*}

\subsection{Estimates on $\zeta$}\label{sec:est.zeta}

{In this section, we prove
\begin{align}\label{est.zeta'}
 \norm{\zeta'}_0
    \leq \frac{\e_0^2}{40\underline{M}(1+T+\tau_0)} \la_{q+1}^{-3\ga} \de_{q+2}^\frac32,
\end{align}
which implies \eqref{est.zeta} by integration in time.}

\subsubsection{Estimates on $ \zeta_1$ and $ \zeta_3$} 

By \eqref{rep:low.freq.app}, $\zeta_3$ can be written as {
\begin{align*}
    \zeta_3'
    &= \underbrace{\dint_{\T^3} \left(\frac{n_o\otimes n_0}{\varrho} - \de_{q+1}\varrho\I + \varrho R_\ell\right) : \na \frac{m_\ell}{\varrho} dx }_{=:\td \zeta_{31}'} 
    \underbrace{-  \dint_{\T^3}
    \varrho (R_{O1} + R_T + R_N)
    : \na \frac{m_\ell}{\varrho}
    dx }_{\td \zeta_{32}'}.
\end{align*}}

First of all, we have $\norm{\zeta_1'}+\norm{\td\zeta_{31}'}_0 \leq \e_0^2/(200\underline{M}(1+T+\tau_0)) \la_{q+1}^{-3\ga} \de_{q+2}^\frac32$, taking advantage of Lemma \ref{phase} and the representations \eqref{rep.W} and \eqref{alg.eq}.

To estimate $\td\zeta_{32}'$, we argue as we did in the previous section to estimate $\|   \ph_{H2}\|_0 $. More precisely, using the decomposition \eqref{dec.Gtri}, we have {
\begin{align*}
    \left\langle \varrho R_\tri : \na \frac {m_\ell}{\varrho}\right\rangle
    &=
    \left\langle\mathcal{R}\mathcal{P}_{\lec \la_{q+1}}G_\tri : \na \frac {m_\ell}{\varrho}\right\rangle
    +\left\langle \mathcal{R}\mathcal{P}_{\gtrsim \la_{q+1}}G_\tri : \na ( {m_\ell}P_{\leq \ell^{-1}}{\varrho}^{-1})\right\rangle\\
    &\quad+\left\langle \mathcal{R}\mathcal{P}_{\gtrsim \la_{q+1}}G_\tri : \na  ({m_\ell}P_{\geq \ell^{-1}}{\varrho}^{-1})\right\rangle.
\end{align*}
where $\varrho R_\tri$ represents either $\varrho R_{O1}$, $\varrho R_T$, or $\varrho R_N$ and can be written as $\mathcal{R}G_\tri$. Since the argument of $\langle{\cdot}\rangle$ in the second term has frequency localized to $\gtrsim \la_{q+1}$, it has zero-mean. The magnitude of }the first term can be estimated {by ${\e_0^2}/(400\underline M(1+T+\tau_0)) \la_{q+1}^{-3\ga}\de_{q+2}^\frac32$} as in \eqref{est.low.1} and \eqref{est.low.2}. The last term can be estimated using
\begin{align*}
   |\langle \mathcal{R}\mathcal{P}_{\gtrsim \la_{q+1}}G_\tri : \na  ({m_\ell}P_{\geq \ell^{-1}}{\varrho}^{-1})\rangle |
   &\lec 
 \norm{\mathcal{R}\mathcal{P}_{\gtrsim \la_{q+1}}G_\tri}_0 \norm{\na  ({m_\ell}P_{\geq \ell^{-1}}{\varrho}^{-1})}_0\\
   &\lec \la_{q+1}^{-1}\norm{G_\tri}_0 \ell^2\la_q\de_q^\frac12
   \leq {\e_0^2}/(400\underline M(1+T+\tau_0)) \la_{q+1}^{-3\ga}\de_{q+2}^\frac32,
\end{align*}
where the second inequality follows from \eqref{est.varrho_smooth} and the last one from \eqref{est.Gtr}. Combining the estimates, we get
\begin{align*}
    \norm{\zeta_1'}_0+\norm{\zeta_3'}_0
    \leq \frac{\e_0^2}{100\underline{M}(1+T+\tau_0)} \la_{q+1}^{-3\ga} \de_{q+2}^\frac32.
\end{align*}

\subsubsection{Estimates on $ \zeta_0$, $ \zeta_2$,  and $ \zeta_4$.} 

Decompose $\zeta_0'$ into $\zeta_{01}'$ and $\zeta_{02}'$ where
\begin{align*}
(2\pi)^3 \zeta_{01}' &= \int_{\T^3} \frac{\varrho}2 D_{t,\ell}\tr\left(  \frac{n_o\otimes n_o}{\varrho^2}- \de_{q+1} \I + R_\ell \right)   dx \\
(2\pi)^3\zeta_{02}'&= - \int_{\T^3}  \frac 12\tr\left(  \frac{n_o\otimes n_o}{\varrho^2} - \de_{q+1} \I + R_\ell \right)\div(m_q-m_\ell) dx
\end{align*}
Since the integrands can be written as
\begin{align*}
    (2\pi)^3  \zeta_{01}' &= \int \sum_{u} \sum_{k\in \Z^3\setminus \{0\}} \frac{ \de_{q+1}  }2\varrho\tr(D_{t,\ell} c_{u,k}) e^{ i\la_{q+1} k\cdot \xi_I}\, dx\\
(2\pi)^3  \zeta_{02}' 
&= {-}\int \sum_{u} \sum_{k\in \Z^3\setminus \{0\}} \frac{ \de_{q+1} }2\tr(c_{u,k}\div (m_q - m_\ell) ) e^{ i\la_{q+1} k\cdot \xi_I}\, dx\\
&={-}\int \sum_{u} \sum_{k\in \Z^3\setminus \{0\}} \frac{ \de_{q+1} }2\tr(c_{u,k}(-\pa_t \varrho + (\pa_t \varrho)_\ell) ) e^{ i\la_{q+1} k\cdot \xi_I}\, dx,
\end{align*}
it thus suffices to use Lemma \ref{phase} to estimate
\[
\norm{ \zeta_0'}_0 \leq 
{\frac{\e_0^2}{200\underline{M}(1+T+\tau_0)} }\la_{q+1}^{-3\ga} \de_{q+2}^\frac32.
\]

In a similar way, we write $\zeta_2'$ and $\zeta_4'$ as
\begin{align*}
(2\pi)^3  \zeta_{2}' 
&=\int \div (m_q-m_\ell) \frac{n\cdot m_\ell}{\varrho^2} dx
=\int \sum_{u} \sum_{k\in \Z^3\setminus \{0\}} \de_{q+1}^\frac12 \div (m_q - m_\ell) \frac{m_\ell}{\varrho^2} \cdot b_{u,k} e^{ i\la_{q+1} k\cdot \xi_I}\, dx\\
&=\int \sum_{u} \sum_{k\in \Z^3\setminus \{0\}} \de_{q+1}^\frac12 (-\pa_t \varrho + (\pa_t \varrho)_\ell) \frac{m_\ell}{\varrho^2} \cdot b_{u,k} e^{ i\la_{q+1} k\cdot \xi_I}\, dx
\end{align*}
and
\begin{align*}
(2\pi)^3  \zeta_{4}' 
&=
\int \frac{n}{\varrho} \cdot \na(p(\varrho) - p_\ell (\varrho))\, dx
= \int \sum_{u} \sum_{k\in \Z^3\setminus \{0\}} \de_{q+1}^\frac12 \frac1\varrho \na(p(\varrho) - p_\ell (\varrho)) \cdot b_{u,k} e^{ i\la_{q+1} k\cdot \xi_I}\, dx.
\end{align*}
Therefore, as before, we apply Lemma \ref{phase} to get
\[
\norm{ \zeta_2'}_0+\norm{ \zeta_4'}_0 \leq {\frac{\e_0^2}{200\underline{M}(1+T+\tau_0)} }\la_{q+1}^{-3\ga} \de_{q+2}^\frac32. \]

\medskip
Thus, we get the desired estimate \eqref{est.zeta'}.

\subsection{Transport current error}
We use the definition of $   \ph_T$ and recall its splitting into $  \ph_{T1}+  \ph_{T2}$. 
Since  we have $\norm{e^{ i\la_{q+1} k\cdot \xi_I}}_{N} \lec \la_{q+1}^N|k|^2$ for any $k\in \Z^3\setminus\{0\}$ and $N\leq2$, $D_{t,\ell} e^{i\la_{q+1}k\cdot \xi_I} =0$, and almost disjoint support of $c_{u,k}$, \eqref{est.k}, \eqref{alg.eq} and \eqref{est.c} imply
\begin{align}\label{est.sym.h1}
\Norm{\frac{n_o\otimes n_o}{\varrho^2} - \de_{q+1} \I + R_\ell}_N
&\lec \sum_{k\in \Z^3\setminus \{0\}}\Norm{\sum_{u\in \Z}   \de_{q+1} \frac{c_{u,k}}{\varrho^2} e^{ i\la_{q+1} k\cdot \xi_I}}_{N} 
\lec \la_{q+1}^N \de_{q+1}, 
\end{align}
 \begin{equation}\begin{split}
\label{est.sym.h2}
\Norm{   D_{t,q+1} (\frac{n_o\otimes n_o}{\varrho^2} - \de_{q+1} \I + R_\ell)}_{N-1}
&\lec \sum_{k\in \Z^3\setminus \{0\}}\Norm{\sum_{u\in \Z}   \de_{q+1} \left(D_{t,\ell}\frac{c_{u,k}}{\varrho^2}\right) e^{ i\la_{q+1} k\cdot \xi_I}}_{N-1}\\
&\quad +\Norm{(n+(m_q-m_\ell)\cdot \na )\left(\frac{n_o\otimes n_o}{\varrho^2} - \de_{q+1} \I + R_\ell\right)}_{N-1}\\
&\lec \la_{q+1}^N\de_{q+1}^\frac12 \cdot \de_{q+1}. 
\end{split}\end{equation}
Recall that
\begin{align*}
    \varrho\ph_{T1} &= -\ka_{q+1}  n+ \frac12 \tr\left(  \frac{n_o\otimes n_o}{\varrho^2} - \de_{q+1} \I + R_\ell \right)(m_q-m_\ell) - \frac{m_q\zeta}{\varrho}
\end{align*}
We then can use \eqref{est.asR}, \eqref{est.w}, \eqref{est.sym.h1}, \eqref{est.sym.h2}, \eqref{est.v.dif}, \eqref{est.zeta}, and $\ka_{q+1} = \frac12 \tr(R_{q+1})$ to estimate
\begin{align*}
\norm{ \varrho \ph_{T1}}_N 
&\lec  \sum_{N_1+N_2=N} \norm{R_{q+1} -{\textstyle{\frac 23}}  \zeta {/\varrho}\I +{\textstyle{\frac 23}}  \zeta{/\varrho}\I}_{N_1} \norm{n}_{N_2}\\
&\qquad+  \sum_{N_1+N_2=N}\Norm{ \frac{n_o\otimes n_o}{\varrho^2} - \de_{q+1} \I + R_\ell}_{N_1}\norm{m_q-m_\ell}_{N_2}
+\norm{m_q/\varrho}_{N} \norm{\zeta}_0\\
&\lec  \la_{q+1}^N\left( \la_q^\frac12\la_{q+1}^{-\frac12}\de_q^\frac14\de_{q+1}^\frac14 {+ \ell^2\la_q^2 \de_q^\frac12}\right) \de_{q+1}
\lec \la_{q+1}^N 
 \la_q^\frac12\la_{q+1}^{-\frac12}\de_q^\frac14\de_{q+1}^\frac54,
\end{align*}
and using additionally \eqref{nmtdv.ml} and \eqref{est.mtdv.rho}
\begin{align*}
&\norm{   D_{t,q+1} \ph_{T1}}_{N-1}\\
&\quad\lec
 \sum_{N_1+N_2=N-1} \norm{   D_{t,q+1} (R_{q+1} -{\textstyle{\frac 23} \frac{\zeta}{\varrho}} \I + {\textstyle{\frac 23} \frac{\zeta}{\varrho}} \I)}_{N_1}\norm{ n}_{N_2} + \norm{R_{q+1}-{\textstyle{\frac 23}\frac{\zeta}{\varrho} \I + {\textstyle{\frac 23}} \frac{\zeta}{\varrho} \I} }_{N_1}\norm{    D_{t,q+1} n}_{N_2} \\
&\qquad+ \sum_{N_1+N_2=N-1}  \Norm{   D_{t,q+1} \left(\frac{n_o\otimes n_o}{\varrho^2} -\de_{q+1}I+ R_\ell\right)}_{N_1}\norm{m_q-m_\ell}_{N_2}\\
&\qquad+ \sum_{N_1+N_2=N-1}   \norm{\frac{n_o\otimes n_o}{\varrho^2} -\de_{q+1}I+ R_\ell}_{N_1}\norm{   D_{t,q+1} (m_q-m_\ell)}_{N_2}
+ \norm{D_{t,q+1}(m_\ell\zeta /\varrho^2)}_N
\\
&\quad\lec \la_{q+1}^N \de_{q+1}^\frac12
 \la_q^\frac12\la_{q+1}^{-\frac12}\de_q^\frac14\de_{q+1}^\frac54.
 \end{align*}
As for $  \ph_{T2}$, observe that 
\begin{align*}
    \varrho\ph_{T2} &= \mathcal{R} \left(\frac{\varrho}2 D_{t,\ell}\tr\left(  \frac{n_o\otimes n_o}{\varrho^2} - \de_{q+1} \I + R_\ell  \right) 
\right)\\
&\qquad -
\mathcal{R}\left( \frac 12\tr\left(  \frac{n_o\otimes n_o}{\varrho^2} - \de_{q+1} \I + R_\ell \right)(\div m_q - \div m_\ell)\right)
\end{align*}
from \eqref{e:phi_T2} and since $\div m_q  = - \pa_t \varrho$. By \eqref{alg.eq},  
\begin{align*}
  \varrho  \ph_{T2} 
&=\frac 12\idv{\sum_{m} \sum_{k\in \Z^3\setminus \{0\}}  \de_{q+1} \tr(\varrho D_{t,\ell} c_{m,k}) e^{ i\la_{q+1} k\cdot \xi_I}}\\
& \qquad + \frac 12\idv{\sum_{m} \sum_{k\in \Z^3\setminus \{0\}}  \de_{q+1} \tr((\pa_t \varrho - (\pa_t \varrho)_\ell) c_{m,k}) e^{ i\la_{q+1} k\cdot \xi_I}}
\end{align*}

and estimate it using Corollary \ref{cor.mic2} with \eqref{est.c}
as follows
\begin{align*}
\norm{  \ph_{T2}}_N
\lec  \la_{q+1}^N\cdot
\frac {\de_{q+1}}{\la_{q+1} \tau_q},\quad
\norm{   D_{t,q+1}  \ph_{T2}}_{N-1} \lec \la_{q+1}^N\de_{q+1}^\frac12 \cdot
\frac {\de_{q+1}}{\la_{q+1} \tau_q}. 
\end{align*}
To summarize, we have
\begin{align*}
\norm{  \ph_{T}}_N\leq \frac 1{5}  \la_{q+1}^{N-3\ga}\de_{q+2}^\frac32,\quad
\norm{   D_{t,q+1}  \ph_{T}}_{N-1}\leq \frac 1{5}  \la_{q+1}^{N-3\ga}\de_{q+1}^\frac12\de_{q+2}^\frac32.
\end{align*}
for sufficiently small $b-1>0$ and large $\la_0$. 

\subsection{Oscillation current error} 
Recall that $ \varrho \ph_{O1}
=\mathcal{R}\left(\na\cdot \left(\frac {|n_o|^2n_o}{2\varrho^2} +\varrho \ph_\ell \right)\right)$. We remark that \eqref{e:rep3} gives
\begin{align*}
\div \left(\frac {|n_o|^2n_o}{2\varrho^2} +\varrho \varphi_\ell \right)
&= \div\left(\sum_{u\in \Z} \sum_{k\in \Z^3\setminus \{0\}} \de_{q+1}^\frac 32 \frac{d_{u,k}}{\varrho^2} e^{ i\la_{q+1} k\cdot \xi_I}\right)
= \sum_{u,k}  \de_{q+1}^\frac 32 \div\left(\frac{d_{u,k}}{\varrho^2}\right)e^{ i\la_{q+1} k\cdot \xi_I},
\end{align*}
because of $\dot{d}_{I,k} (f_I\cdot k) =0$. Also, we have
\begin{align*}
\Norm{D_{t,\ell} \div \frac{d_{u,k}}{\varrho^2}}_{\bar N}
\lec \Norm{D_{t,\ell} \frac{d_{u,k}}{\varrho^2}}_{\bar N+1} + \Norm{\na (m_\ell/\varrho)^{\top}:\na \frac{d_{u,k}}{\varrho^2}}_{\bar N }
{\lec_{M}\la_{q+1}^{\bar N} \frac{\la_{q+1}\de_{q+1}^\frac12}{\mu_q}}, \quad \forall \bar N\in[0,n_0{+1}].
\end{align*}
Therefore, using $\supp(d_{u,k}) \subset (t_u -\frac 12\tau_q, t_u + \frac 32\tau_q) \times \R^3$, it follows from Corollary \ref{cor.mic2} with \eqref{est.d} that
\begin{align*}
\norm{ \varrho \ph_{O1}}_N
\lec \la_{q+1}^N \cdot \frac {\de_{q+1}^\frac 32}{\la_{q+1}\mu_q}, \quad
\norm{   D_{t,q+1} (\varrho \ph_{O1})}_{N-1} 
\lec \la_{q+1}^N \de_{q+1}^\frac12\cdot \frac {\de_{q+1}^\frac 32}{\la_{q+1}\mu_q} .
\end{align*}
Next recall that $ \varrho \ph_{O2} = \frac {|n|^2n - |n_o|^2 n_o}{2\varrho^2}$. Then, \eqref{est.W}-\eqref{est.w} imply
\begin{align*}
\norm{ \varrho \ph_{O2}}_N
&\lec 
\Norm{\frac{(n_o\cdot n_c) n}{\varrho^2}}_N 
+ \Norm{\frac{|n_c|^2 n}{2\varrho^2}}_N + \Norm{\frac{|{n_o}|^2 n_c}{2\varrho^2}}_N
\lec \la_{q+1}^N \cdot \frac {\de_{q+1}^\frac32}{\la_{q+1}\mu_q}\\
\norm{   D_{t,q+1} (\varrho \ph_{O2})}_{N-1}
&\lec \Norm{   D_{t,q+1} \left(\frac{|{n_o}|^2 n_c}{2\varrho^2}\right)}_{N-1}
+ \Norm{   D_{t,q+1} \left(\frac{(n_o\cdot n_c)n}{\varrho^2} +\frac{ |n_c|^2n}{2\varrho^2}\right)}_{N-1}\\
& \lec \la_{q+1}^N\de_{q+1}^\frac12 \cdot \frac {\de_{q+1}^\frac32}{\la_{q+1}\mu_q}.
\end{align*}
Therefore, combining the estimates, we get
\begin{align*}
\norm{  \ph_{O}}_N\leq \frac 1{5} \la_{q+1}^{N-3\ga} \de_{q+2}^\frac32, \quad 
\norm{   D_{t,q+1}  \ph_{O}}_{N-1}\leq \frac 1{5} \la_{q+1}^{N-3\ga} \de_{q+1}^\frac12\de_{q+2}^\frac32,
\end{align*}
for sufficiently small $b-1>0$ and large $\la_0$. 

\subsection{Reynolds current error} 
Recall that $ \varrho \ph_R = (R_{q+1}-{\textstyle{\frac23}} \zeta{/\varrho} \I) n + {\textstyle{\frac23}} (\zeta{/\varrho})  n$. Similar to the estimate for $\ka_{q+1} n$ in $  \ph_{T1}$, we have
\begin{align*}
\norm{  \ph_R}_N 
&\lec \la_{q+1}^N 
\la_q^\frac12\la_{q+1}^{-\frac12} \de_q^\frac14\de_{q+1}^\frac54
\leq \frac 15 \la_{q+1}^{N-3\ga} \de_{q+2}^\frac32,\\
\norm{   D_{t,q+1}   \ph_R}_{N-1}
&\lec \la_{q+1}^N\de_{q+1}^\frac12 \la_q^\frac12\la_{q+1}^{-\frac12} \de_q^\frac14\de_{q+1}^\frac54
\leq \frac 15 \la_{q+1}^{N-3\ga}\de_{q+1}^\frac12 \de_{q+2}^\frac32
\end{align*}
for sufficiently small $b-1>0$ and large $\la_0$. 

\subsection{Mediation current error} 
Recall that $\varrho \ph_M = \varrho \ph_{M1} + \varrho \ph_{M2} + \varrho \ph_{M3} + \varrho \ph_{M4}$. We estimate each term separately. Recall now
\begin{align*}
  \varrho\varphi_{M1} =   \frac{|m_q-m_\ell|^2}{2\varrho}\frac n{\varrho}+ \varrho(\ph_q-\ph_\ell)
\end{align*}
\begin{align*}
  \varrho\varphi_{M2} = \left (\frac{n\otimes n}{\varrho}+\varrho R_q -\varrho R_{q+1} - \de_{q+1} \varrho \I \right) \frac{m_q-m_\ell}{\varrho}
\end{align*}
\begin{align*}
  \varrho\ph_{M3} = \mathcal{R} \left( 
  \div (m_q-m_\ell) \frac{n\cdot m_\ell}{\varrho^2} \right)
\end{align*}
\begin{align*}
  \varrho\ph_{M4} &:= \mathcal{R} \left( \frac{n}{\varrho} \cdot \na(p(\varrho) - p_\ell (\varrho))  \right)
\end{align*}
For $\ph_{M1}$, we use \eqref{est.v.dif}, \eqref{est.ph.dif}, and \eqref{est.w} to get 
\begin{align*}
\norm{  \ph_{M1}}_N \leq \frac 1{10} \la_{q+1}^{N-3\ga} \de_{q+2}^\frac32, \quad 
\norm{   D_{t,q+1}    \ph_{M1}}_{N-1} \leq \frac 1{10} \la_{q+1}^{N-3\ga}\de_{q+1}^\frac12 \de_{q+2}^\frac32.
\end{align*}
For $\ph_{M2}$, we use \eqref{rep:low.freq.app}, 
and we estimate $\ph_{M2}$ in a similar way as $  \ph_{T1}$ and $\ph_{H2}$,
\begin{align*}
\Norm{\left (\frac{n\otimes n}{\varrho}+\varrho R_q -\varrho R_{q+1} - \de_{q+1} \varrho \I \right) \frac{m_q-m_\ell}{\varrho}}_{N}
&\leq \la_{q+1}^N\la_q^\frac12\la_{q+1}^{-\frac12} \de_q^\frac14\de_{q+1}^\frac54,\\
\Norm{   D_{t,q+1} \left[\left (\frac{n\otimes n}{\varrho}+\varrho R_q -\varrho R_{q+1} - \de_{q+1} \varrho \I \right) \frac{m_q-m_\ell}{\varrho}\right]}_{N-1}
&\leq_M \la_{q+1}^N\de_{q+1}^\frac12\la_q^\frac12\la_{q+1}^{-\frac12} \de_q^\frac14\de_{q+1}^\frac54
\end{align*}
For $\ph_{M3}$ and $\ph_{M4}$, observe that the argument of $\mathcal{R}$ have the respective forms
\begin{align*}
     &\sum_{u} \sum_{k\in \Z^3\setminus \{0\}} \de_{q+1}^\frac12 \div (m_q - m_\ell) \frac{m_\ell}{\varrho^2} \cdot b_{u,k} e^{ i\la_{q+1} k\cdot \xi_I}\\
     &\sum_{u} \sum_{k\in \Z^3\setminus \{0\}} \de_{q+1}^\frac12 \na(p(\varrho) - p_\ell (\varrho)) \frac{1}{\varrho} \cdot b_{u,k} e^{ i\la_{q+1} k\cdot \xi_I}
\end{align*}
Therefore, using $\supp(b_{u,k}) \subset (t_u -\frac 12\tau_q, t_u + \frac 32\tau_q) \times \R^3$, the desired estimate follows from Corollary \ref{cor.mic2};
\begin{align*}
\norm{  \ph_{M}}_N \leq \frac 1{5} \la_{q+1}^{N-3\ga} \de_{q+2}^\frac32, \quad 
\norm{   D_{t,q+1}    \ph_{M}}_{N-1} \leq \frac 1{5} \la_{q+1}^{N-3\ga}\de_{q+1}^\frac12 \de_{q+2}^\frac32.
\end{align*}

\section{Proofs of the key inductive propositions}

\subsection{Proof of Proposition \ref{ind.hyp}}
For a given dissipative Euler-Reynolds flow $(m_q,c_q,R_q,\ph_q)$ on the time interval $[0,T]+\tau_{q-1}$, we recall the construction of the corrected one: $m_{q+1} = m_q + n_{q+1}$, where $n_{q+1}$ is defined by \eqref{def.w}. Furthermore, we find a new Reynolds stress $R_{q+1}$ and an unsolved flux current $\ph_{q+1}$ which solve \eqref{app.eq} together with $m_{q+1}$, and satisfy \eqref{est.asR}, \eqref{est.asph} and \eqref{est.zeta} for sufficiently small $b-1>0$ and large $\la_0$. In other words, $(m_{q+1},c_{q+1},R_{q+1},\ph_{q+1})$ is a dissipative Euler-Reynolds flow {for the energy loss $E$} and the error $(R_{q+1}, \ph_{q+1})$ satisfies \eqref{est.R}-\eqref{est.ph} for $q+1$ as desired. Now, denote the absolute implicit constant in the estimate {\eqref{est.w.indM}} for $n$ by $M_0$ and define $M = 2M_0$. Then, one can easily see that
\begin{align*}
\norm{m_{q+1}- m_q}_0
+ \frac 1{\la_{q+1}} \norm{m_{q+1} - m_q}_1 
= \norm{n_{q+1}}_0 
+\frac 1{\la_{q+1}} \norm{n_{q+1}}_1 
\leq 2M_0 \de_{q+1}^\frac12  = M \de_{q+1}^\frac12. 
\end{align*}
Also, using \eqref{est.vp} and \eqref{est.w}, we have
\begin{align*}
\norm{m_{q+1}}_0
&\leq \norm{m_q}_0 + \norm{n_{q+1}}_0
\leq \underline{M}-\de_q^\frac12 + M_0\de_{q+1}^\frac12 
\leq \underline{M}-\de_{q+1}^\frac12,\\
\norm{m_{q+1}}_N
&\leq \norm{m_q}_N + \norm{n_{q+1}}_N
\leq M \la_q^N\de_q^\frac12 + \frac12 M \la_{q+1}^N \de_{q+1}^\frac12
\leq M\la_{q+1}^N \de_{q+1}^\frac12,\\
\end{align*}
for $N=1,2$, provided that $\la_0$ is sufficiently large. Therefore, we construct a desired corrected flow $(m_{q+1},c_{q+1},R_{q+1},\ph_{q+1})$. 

\subsection{Proof of Proposition \ref{p:ind_technical}}
Consider a given time interval $\cal I \subset (0,T)$ with $|\cal I| \geq 3\tau_q$. Then, we can always find $u_0$ such that $\supp(\th_{u_0}(\tau_q^{-1}\cdot))\subset  \cal I$. Now, if $I = (u_0, v, f)$ belongs to $\mathscr{S}_R$, we replace $\ga_{I}$ in $n_{q+1, o}$ by $\td \ga_{I} =-\ga_{I}$. In other words, we replace $\Ga_{I}$ by $\td \Ga_{I} = -\Ga_{I}$. Otherwise, we keep the same $\ga_{I}$. Note that $\td\Ga_{I}$ still solves \eqref{eq.Ga1} and hence $\td \ga_{I}$ satisfies \eqref{eq.Ga}. Since the replacement does not change the estimates for $\Ga_{I}$ used in the proof of Lemma \ref{lem:est.coe}, the corresponding coefficients $\td b_{u,k}$, $\td c_{u,k}$, $\td d_{u,k}$, and $\td e_{u,k}$ satisfy \eqref{est.b}-\eqref{est.e}, and $\td n_o$, $\td n_c$, and $\td n_{q+1}$, generated by them, also fulfill \eqref{est.W}-\eqref{est.w}. As a result, the corrected dissipative Euler-Reynolds flow $(\td m_{q+1}, c_{q+1}, \td R_{q+1},  \td\ph_{q+1})$ satisfies \eqref{est.vp}-\eqref{est.ph} for $q+1$ and \eqref{cauchy} as desired. On the other hand, by the construction, the correction $\td n_{q+1}$ differs from $n_{q+1}$ on the support of $\th_{u_0}(\tau_q^{-1}\cdot)$. Therefore, we can easily see 
\[
\supp_t(m_{q+1} - \td m_{q+1})
=\supp_t(n_{q+1} - \td n_{q+1})
\subset \cal I. 
\] 
Furthermore, by \eqref{eq.Ga} and \eqref{eq.Ga1}, we have 
\begin{align*}
\sum_{I\in \mathscr{I}_{u,v,R}}\ga_I^2 |(\na \xi_I)^{-1}f_I|^2
&=\tr\left(
(\na \xi_I)^{-1}  \sum_{f\in \cF_{u,v,R}} \ga_I^2 f_I\otimes f_I [(\na \xi_I)^{-1}]^\top
\right)\\
&= \tr (\varrho^2(\de_{q+1}\I - R_\ell - \widetilde{M}_{I})
\end{align*} 
where $\widetilde{M}_I := \sum_{(u',v')\in I(u,v)}\th_{I'}^2\chi_{I'}^2(\xi_{I'})\sum_{f'\in \mathscr{I}_{u',v',\ph}}\ga_{I'}^2\left(\dint_{\T^3}\psi_{I'}^2dx \right) (\na \xi_{I'})^{-1} f' \otimes  (\na \xi_{I'})^{-1} f' $.
In particular 
\[
\norm{\widetilde{M}_I}_0 \lec \la_q^{-2\ga} \de_{q+1}
\] 
(see the proof in section \ref{subsec:R}). In this proof, $\norm{\cdot}_N$ denotes $\norm{\cdot}_{C([0,T];C^N(\T^3))}$.
Then, it follows that
\begin{align*}
|n_o - \td n_o |^2
&=  \sum_{I\in \mathscr{I}_R: u_I = u_0}
4\th_I^2(t) \chi_I^2(\xi_I) \ga_I^2 |(\na \xi_I)^{-1} f_I|^2 (1+(\psi_I^2(\la_{q+1} \xi_I)-1))\\
&= \sum_{I\in \mathscr{I}_R: u_I = u_0}
4\th_I^2\chi_I^2(\xi_I)  (\varrho^2(3\de_{q+1} - \tr(R_\ell)) - \tr(\widetilde{M}_I))\\&\quad+
\sum_{k\in \Z^3\setminus\{0\}}\sum_{I\in \mathscr{I}_R: u_I = u_0}
4\th_I^2\chi_I^2(\xi_I) \ga_I^2 |(\na \xi_I)^{-1} f_I|^2
\dot{c}_{I,k} e^{i\la_{q+1}k\cdot \xi_I}\\
&= 4\th_{u_0}^6(\tau_q^{-1}t) (\varrho^2(3\de_{q+1}- \tr R_\ell) - \tr {\widetilde{M} }) + \sum_{k\in \Z^3\setminus\{0\}} 4\de_{q+1}\tr(\td c_{u_0,k}^R)e^{i\la_{q+1}k\cdot \xi_I},
\end{align*}
where 
\[
\tr(\td c_{u_0,k}^R) = \sum_{I\in \mathscr{I}_R: u_I = u_0}
\th_I^2(t) \chi_I^2(\xi_I) \de_{q+1}^{-1}\ga_I^2\dot{c}_{I,k} |(\na \xi_I)^{-1} f_I|^2.
\]
Since we can obtain $\norm{\tr(\td c_{u_0,k}^R)}_N \lec \mu_q^{-N} |\dot{c}_{I,k}|$ for $N=0,1,2$ in the same way used to get the estimate \eqref{est.c} for $c_{u_0,k}$, we conclude
\begin{align*}
\norm{n_o - \td  n_o}_{C^0([0,T]; L^2(\T^3))}^2
&\geq 12\de_{q+1}\norm{\varrho}^2_{C^0({[t_{u_0}+\frac18\tau_q, t_{u_0}+\frac78\tau_q]}; L^2(\T^3))} - 4(2\pi)^3(\norm{\varrho^2 R_\ell}_0 + \norm{\tr(\widetilde{M})}_0) \\
&\quad- {\sup_{t\in [0,T]}}\sum_{k\in \Z^3\setminus\{0\}}4 
\left| 
\de_{q+1} \int\tr(\td c_{u_0,k}^R)e^{i\la_{q+1}k\cdot \xi_I} dx
\right|\\
&\geq 12\de_{q+1}\norm{\varrho}^2_{C^0([t_{u_0}+\frac18\tau_q, t_{u_0}+\frac78\tau_q]; L^2(\T^3))} - c_\varrho \de_{q+1}(\la_q^{-3\ga} + \la_q^{-2\ga} + (\la_{q+1} \mu_q)^{-2})\\
&\geq 4\de_{q+1}{\e_0}
\end{align*}
for sufficiently large $\la_0$. Indeed, in the second inequality, we used Lemma \ref{phase} to get
\begin{align*}
&\sum_{k\in \Z^3\setminus\{0\}}\left| \int \tr(\td c_{u_0,k}^R) e^{\la_{q+1} k\cdot \xi_I} dx\right|
\lec \sum_{k \in \Z^3\setminus\{0\}}\frac{\norm{\tr(\td c_{u_0,k}^R)}_2 + \norm{\tr(\td c_{u_0,k}^R)}_0\norm{{\na} \xi_I}_{C^0([t_{u_0}-\frac12\tau_q, t_{u_0}+\frac32\tau_q];C^2( \T^3))}}{\la_{q+1}^2 |k|^2}
\\
& \qquad \lec (\la_{q+1}\mu_q)^{-2} \sum_{k \in \Z^3\setminus\{0\}}\frac{|\dot{c}_{I,k}|}{|k|^2}
\lec (\la_{q+1}\mu_q)^{-2} 
\left(\sum_{k \in \Z^3\setminus\{0\}} |\dot{c}_{I,k}|^2\right)^\frac12
\left(\sum_{k \in \Z^3\setminus\{0\}} \frac{1}{|k|^4}\right)^\frac12.
\end{align*}
Therefore, we obtain
\begin{align*}
\norm{m_{q+1} - \td m_{q+1} }_{C^0([0,T]; L^2(\T^3))}
&=\norm{n_{q+1}- \td n_{q+1}}_{C^0([0,T]; L^2(\T^3))}\\
&\geq \norm{n_o - \td  n_o}_{C^0([0,T]; L^2(\T^3))}
-(2\pi)^\frac32(\norm{n_c}_0 + \norm{\td n_c}_0)\\
&\geq  2\de_{q+1}^\frac12 \e_0 -\frac{(2\pi)^\frac32 2M_0}{\la_{q+1}\mu_q} \de_{q+1}^\frac12 \geq \de_{q+1}^\frac12 \e_0
\end{align*}
for sufficiently large $\la_0$. 

Lastly, we suppose that a dissipative Euler-Reynolds flow $(\td m_q, c_q, \td R_q,  \td \ph_q)$ satisfies \eqref{est.vp}-\eqref{est.ph} and
\[
\supp_t(m_q-\td m_q, R_q-\td R_q,  \ph_q-\td \ph_q)\subset \cal{J} 
\]
for some time interval $\cal{J}$. Proceed to construct the regularized flow, $\td R_\ell$ and $\td \ph_\ell$ as we did for  $R_\ell$ and $\ph_\ell$ and note that they differ only in $\cal J + \ell_t\subset\cal J + (\la_q\de_q^\frac12)^{-1}$. Consequently, $n_{q+1}$ differ from $\td n_{q+1}$ only in $\cal J + (\la_q\de_q^\frac12)^{-1}$ and hence the corrected dissipative Euler-Reynolds flows $({m}_{q+1},  c_{q+1},  R_{q+1},  \ph_{q+1})$ and $(\td m_{q+1},  c_{q+1}, \td R_{q+1},  \td\ph_{q+1})$ satisfy 
\[
\supp_t(m_{q+1}-\td m_{q+1}, R_{q+1}-\td R_{q+1},  \ph_{q+1}-\td \ph_{q+1})\subset \cal{J} + (\la_q\de_q^\frac12)^{-1}.
\]

 \appendix
 \section{Some technical lemmas}
The proof of the following two lemmas can be found in \cite[Appendix]{BDLSV2020}.
 \begin{lem}[H\"{o}lder norm of compositions] \label{lem:est.com} Suppose $F:\Omega \to \R$ and $\Psi: \R^n \to \Omega$ are smooth functions for some $\Omega\subset \R^m$. Then, for each $N\in \N$, we have
\begin{align}
&\norm{\na^N (F\circ \Psi)}_0
\lec \norm{\na F}_0 \norm{\na\Psi}_{N-1} + \norm{\na F}_{N-1} \norm{\Psi}_0^{N-1}\norm{\Psi}_N \nonumber\\
&\norm{\na^N (F\circ \Psi)}_0
\lec \norm{\na F}_0 \norm{\na\Psi}_{N-1} + \norm{\na F}_{N-1} \norm{\na\Psi}_0^{N} \label{chain},
\end{align}
where the implicit constant in the inequalities depends only on $n$, $m$, and $N$.
 \end{lem}
 
 \begin{lem}\label{phase}
 Let $N\geq 1$. Suppose that $a\in C^\infty(\T^3)$ and $\xi\in C^\infty(\T^3;\R^3)$ satisfies
 \[
 \frac 1{C} \leq |\na \xi| \leq C
 \]
 for some constant $C>1$. Then, we have
 \[
 \left| \int_{\T^3} a(x) e^{ik\cdot \xi} dx \right|
 \lec \frac{\norm{a}_N + \norm{a}_0\norm{{\na} \xi}_{N}}{|k|^N} ,
 \]
 where the implicit constant in the inequality is depending on $C$ and $N$, but independent of $k$. 
 \end{lem}
 
The following lemmas contain various commutator estimates, used in the proof. 
\begin{lem} Let $f$ and $g$ be in $C^\infty([0,T]\times \T^3)$ and set $f_\ell = P_{\leq \ell^{-1} } f$, $g_\ell = P_{\leq \ell^{-1}} g$ and $(fg)_\ell =
 P_{\leq \ell^{-1}} (fg)$. Then, for each $N \geq 0$, the following holds,
\begin{align}
\norm{f_\ell g_\ell  - (fg)_\ell }_N \lec_N  \ell^{2-N} \norm{f}_1\norm{g}_1. \label{est.com} 
\end{align}
 \end{lem}
 \begin{proof}
 Since the expression that we need to estimate is localized in frequency, by Bernstein's inequality it suffices to prove the case $N=0$. Recall now the function $\phi$ used to define the Littlewood-Paley operators and the number $J$, which is the maximal natural number such that $2^J \leq \ell^{-1}$. Denoting by $\widecheck{\phi}$ the inverse Fourier transform and by $\widecheck{\phi}_\ell$ the function $\widecheck{\phi}_\ell (x) =
2^{3J} \widecheck{\phi} (2^{J} x)$,
 a simple computation (see for instance \cite{CoETi1994}) gives
\begin{align*}
(f_\ell g_\ell  - (fg)_\ell) (x) &= \frac{1}{2} \int\int (f(x)- f(x-y)) (g(x)-g(x-z)) \widecheck{\phi}_\ell (y) \widecheck{\phi}_\ell (z)\, dy\, dz\, 
\end{align*}
and the claim easily follows.
 \end{proof}

\begin{lem}\label{lem:com2} Let $f$ and $g$ be in $C^\infty([0,T]\times \T^3)$ and set $f_\ell = P_{\leq \ell^{-1} } f$ and $(fg)_\ell =
P_{\leq \ell^{-1}} (fg)$. Then, for each $N \geq 0$, the following holds,
\begin{equation}\label{est.com0}\begin{split}
\norm{[g, P_{\leq \ell^{-1}}]f}_0 
&\lec  \ell\norm{f}_0\norm{\na g}_0   \\
\norm{[g, P_{\leq \ell^{-1}}]f}_N 
&\lec_N  \ell^{1-N} \norm{f}_0\norm{g}_{\max \{1,N\}}.  
\end{split}\end{equation}
In particular, for any smooth function $v, F\in C^\infty(\T^3)$ and for each $N\geq 0$, we have 
\begin{align}
&\norm{ [v\cdot \na, {P}_{\le \ell^{-1}}]F}_0=\norm{[v\cdot\na, {P}_{> \ell^{-1}}] F}_{0} \lec \ell\norm{\na F}_0\norm{\na v}_0\label{est.com1}\\
&\norm{ [v \cdot \na, {P}_{\le \ell^{-1}}]F}_N =\norm{[v\cdot\na, {P}_{> \ell^{-1}}] F}_{N}\lec_N \ell^{1-N}\norm{\na F}_0\norm{v}_{\max \{1,N\}}.\label{est.com3}
\end{align} 
\end{lem}
\begin{rem}
When $v$ has the frequency localized to $\ell^{-1}$, using the Bernstein's inequality, \eqref{est.com3} can be improved to
$\norm{ [v\cdot \na, {P}_{\le \ell^{-1}}]F}_N\lec_N \ell^{1-N}\norm{\na v}_0\norm{\na F}_0$. 
\end{rem}
\begin{proof}
We first write
\begin{align*}
 f_\ell g  - (fg)_\ell = \int f(y) (g(x)-g(y)) \widecheck{\phi}_\ell (x-y) \, dy.
\end{align*}
Then, \eqref{est.com0} is obtained as $\norm{ f_\ell g  - (fg)_\ell}_0 
\leq \ell\norm{f}_0\norm{\na g}_0 $
and
\begin{align*}
|\na^N(f_\ell g  - (fg)_\ell)|
&\leq  \int |f(y)| |g(x)-g(y)| |\na^N \widecheck{\phi}_\ell (x-y) |\, dy \\& \quad+ \sum_{N_1=1}^N c_{N_1,N} \int |f(y)| |\na^{N_1}g(x)| |\na^{N-N_1}\widecheck{\phi}_\ell (x-y)| \, dy\\
&\lec \ell^{1-N}\norm{f}_0(\norm{\na g}_0 + \norm{g}_N)
 \end{align*}
for some constants $c_{N_1,N}$. Since we have
\begin{align*}
[v\cdot\na, {P}_{> \ell^{-1}}] F (x)
&= v\cdot \na  ({P}_{> \ell^{-1}}F -F)+(v\cdot \na F) -  {P}_{> \ell^{-1}}(v\cdot \na F)    \\
&= -v \cdot \na {P}_{\le \ell^{-1}} F + {P}_{\le \ell^{-1}} (v \cdot \na F) = - [v\cdot\na, {P}_{\le \ell^{-1}}] F,
\end{align*}
we apply \eqref{est.com0} to $g=v_i$ and $f=\pa_i F$, then \eqref{est.com1} and \eqref{est.com3} follow. 
 \end{proof}

\begin{lem}\label{lem.com4} For a fixed $\overline{N}\in \mathbb{N}$, if $v$ and $g$ satisfy
\[
\norm{v}_N \lec_{\overline N} \ell^{-N}v_F, \quad
\norm{g}_N \lec_{\overline N}  g_F
\]
for all integer $N\in [1, \overline{N}]$ and for some positive constants $v_F$ and $g_F$, then we have
\begin{align}
\norm{[v  \cdot \na, P_{\le \ell^{-1}}](fg)-([v \cdot \na, P_{\le \ell^{-1}}]f) g}_{N} 
\lesssim \ell^{1-N}\norm{\na f}_0 v_F g_F
+\ell^{-N}\norm{f}_0v_Fg_F, 
\label{est.com5}    
\end{align}
for any integer $N\in [0,\overline N]$.
\end{lem}
\begin{proof}
We first write
\begin{align*}
[v &\cdot \na, P_{\le  \ell^{-1}}](fg)-([v \cdot \na, P_{\le \ell^{-1}}]f) g\\
&=\int_{\R^3} (v  (x) - v  (y)) \cdot \na (fg)(y) \widecheck{\phi}_{\ell}(x-y) \, dy - \int_{\R^3} (v  (x) - v  (y)) \cdot \na f(y) g(x) \widecheck{\phi}_{\ell}(x-y) \, dy \\
&=\int_{\R^3} (v  (x) - v  (y)) \cdot \na f(y) (g(y)-g(x)) \widecheck{\phi}_{\ell}(x-y) \, dy \\
&\hspace{7cm} +\int_{\R^3} (v  (x) - v  (y)) \cdot \na g(y) f(y) \widecheck{\phi}_{\ell}(x-y) \, dy. 
\end{align*}
Then, \eqref{est.com5} follows from
\[
\norm{\na_x^N (v(x)-v(y))}_0
\lec \ell^{-N} v_F, \quad 
|g(x)-g(y)|
\lec {|x-y|} g_F, \quad 
\int |y||\widecheck{\phi}_{\ell}(y)| dy \lec \ell.
\]
\end{proof}

\begin{lem}\label{lem:com*}
For {vector-valued} functions $H$ and {$v$} in $C^\infty([0,T]\times \T^3)$, the following commutator estimate holds,
\begin{align*}
\norm{[P_{\lec \ell^{-1}}v \cdot\na, \cR]P_{\gtrsim \la_{q+1}} H}_{N-1} 
\lec \sum_{N_1+N_2=N-1} \ell \norm{\na v}_{N_1} \norm{H}_{N_2}. \end{align*}
for $N=1,2$, where $\cR= \De^{-1}\na$.
\end{lem}
\begin{proof}

For convenience, we write $v_\ell:= P_{\lec \ell^{-1}}v$ and $H_j = P_{2^j}H$ for for $2^j \gtrsim \la_{q+1}$. We first use the Fourier expansion and the Taylor's theorem to get
\begin{align}
&{-}[v_{\ell} \cdot\na, \cR]H_j \nonumber\\ 
&= \sum_{k, \eta \in \Z^3}  (\crF[\cR](k)-\crF[\cR](\eta)) i\eta \cdot \crF[ v_\ell](k-\eta) \crF[H_j] (\eta) e^{ik\cdot x}\nonumber\\
&= \sum_{k, \eta \in \Z^3} \sum_{l=1}^{l_0} \frac1{l!}[(k-\eta)\cdot \na]^l \crF[\cR](\eta) i\eta \cdot \crF[v_\ell](k-\eta) \crF[H_j] (\eta) e^{ik\cdot x}\nonumber\\
\begin{split}
&\quad + \frac1{l_0!} \sum_{k, \eta \in \Z^3} \int_0^1 [(k-\eta)\cdot \na]^{l_0+1} \crF[\cR](\eta + \sigma (k-\eta))(1-\sigma)^{l_0} d\sigma i\eta \\
&\hspace{7cm}\cdot \crF[ v_\ell](k-\eta) \crF[H_j] (\eta) e^{ik\cdot x},\label{com*}
\end{split}
\end{align}
{where $l_0>2$, independent of $q$, is chosen to satisfy $(\ell^{-1}\la_{q+1}^{-1})^{l_0} \la_{q+1}^3\lec 1$.} The first term can be written as $\sum_{l=1}^{l_0}\frac1{l!} \na^l v_\ell : \na \cR_l H_j$ where the operator $\cR_l$ has a Fourier multiplier defined by $\crF[\cR_l g](\eta) = { (-i)^l}\na_\eta^l \crF[\cR](\eta) \crF[g](\eta)$. 

Using this decomposition, we now estimate 
\begin{align*}
\Norm{\sum_{l=1}^{l_0}\frac1{l!} \na^l v_\ell : \na \cR_l P_{\gtrsim \la_{q+1}} H}_{N-1} &\lesssim \sum_{N_1+N_2=N-1} \sum_{l=1}^{l_0}\frac1{l!} \norm{ \na^l v_\ell }_{N_1}\norm{ \na \cR_l P_{\gtrsim \la_{q+1}} H}_{N_2}\\
&\lesssim \sum_{N_1+N_2=N-1} \sum_{l=1}^{l_0}\frac{\ell^{1-l}}{l!} \norm{ \na v }_{N_1} \sum_{2^j\gtrsim \la_{q+1}}\norm{ K_{l,j}}_{L^1} \norm{ H}_{N_2}\\
&\lesssim \sum_{N_1+N_2=N-1} \ell \norm{ \na v }_{N_1} \norm{ H}_{N_2}
\end{align*}
{where $K_{l,j}$ is the kernel of the operator $\na \cR_l P_{2^j}$ and the last estimate follows from 
\begin{align*}
K_{l,j} = 2^{j(-l+3)}K_{l,0}(2^j \cdot), \quad
\sum_{2^j\gtrsim\la_{q+1}}\norm{K_{l,j}}_{L^1(\R^3)}
\lec \la_{q+1}^{-l} \norm{K_{l,0}}_{L^1(\R^3)}.
\end{align*}

The remaining term can be estimated as follows. Since $|k-\eta|\leq \frac 12 |\eta|$ and hence $|k|\lec |\eta|$, we get
\begin{align*}
\Norm{\na^{N-1} \sum_{2^j\gtrsim \la_{q+1}} \eqref{com*}}_0 
&\lesssim {\sum_{2^j\gtrsim \la_{q+1}}}\sum_{k, \eta \in \Z^3} |k|^{N-1} \frac{|k-\eta|^{l_0+1}}{|\eta|^{l_0+1}} |\crF[v_\ell](k-\eta)| |\crF[H_j] (\eta)|\\
&\lec \sum_{2^j\gtrsim \la_{q+1}}\sum_{k, \eta \in \Z^3} \frac{|k-\eta|^{l_0}}{|\eta|^{l_0+1}}
|\crF[\na v_\ell](k-\eta)| |\eta|^{N-1}|\crF[H_j] (\eta)|\\
&\lec\sum_{2^j\gtrsim \la_{q+1}}
\sum_{|k|\lec 2^j}
(\ell^{-1}2^{-j})^{l_0} 2^{-j} \norm{\na v}_{L^2(\T^3)}\norm{{\na^{N-1}}H_j}_{L^2(\T^3)}\\
&\lec \sum_{2^j\gtrsim \la_{q+1}} 
(\ell^{-1}2^{-j})^{l_0} 2^{2j} \norm{\na v}_{L^2(\T^3)}\norm{\na^{N-1}H_j}_{L^2(\T^3)}\\
&\lec (\ell^{-1}\la_{q+1}^{-1})^{l_0} \la_{q+1}^2
\norm{\na v}_0 \norm{H}_{N-1}
\lec \la_{q+1}^{-1}\norm{\na v}_0 \norm{H}_{N-1}.
\end{align*}

}

\end{proof}

\subsection*{Acknowledgments}
The first author has been supported by the National Science Foundation under Grant No. DMS-FRG-1854344. The second author has been supported by the NSF under Grant No. DMS-1926686. {The authors are grateful to Camillo De Lellis for helpful discussions and his contribution to Section \ref{sec:time-dependent.density}.}


\begin{thebibliography}{10}

\bibitem{Bu2014}
T.~Buckmaster.
\newblock {\em {O}nsager's {C}onjecture}.
\newblock PhD thesis, Universit{\"a}t Leipzig, 2014.

\bibitem{Bu2015}
T.~Buckmaster.
\newblock Onsager's conjecture almost everywhere in time.
\newblock {\em Communications in Mathematical Physics}, 333(3):1175--1198,
  2015.

\bibitem{BuDLeIsSz2015}
T.~Buckmaster, C.~De~Lellis, P.~Isett, and L.~Sz{\'e}kelyhidi~Jr.
\newblock Anomalous dissipation for {1/5}-holder {E}uler flows.
\newblock {\em Annals of Mathematics}, 182(1):127--172, 2015.

\bibitem{BuDLeSz2013}
T.~Buckmaster, C.~De~Lellis, and L.~Sz{{\'e}}kelyhidi, Jr.
\newblock Transporting microstructure and dissipative {E}uler flows.
\newblock {\em arXiv:1302.2815}, 02 2013.

\bibitem{BuDLeSz2016}
T.~Buckmaster, C.~De~Lellis, and L.~Sz{{\'e}}kelyhidi, Jr.
\newblock Dissipative {E}uler flows with {O}nsager-critical spatial regularity.
\newblock {\em Comm. Pure Appl. Math.}, 69(9):1613--1670, 2016.

\bibitem{BDLSV2020}
T.~Buckmaster, C.~De~Lellis, L.~Sz{\'e}kelyhidi~Jr., and V.~Vicol.
\newblock Onsager's conjecture for admissible weak solutions.
\newblock {\em Communications on Pure and Applied Mathematics}, 72(2):229--274,
  2020/05/19 2019.

\bibitem{CG2019}
G.-Q.~G. Chen and J.~Glimm.
\newblock Kolmogorov-type theory of compressible turbulence and inviscid limit
  of the {N}avier–{S}tokes equations in $\mathbb{R}^3$.
\newblock {\em Physica D: Nonlinear Phenomena}, 400:132138, 2019.

\bibitem{CGL2020}
G.-Q.~G. Chen, J.~Glimm, and D.~Lazarev.
\newblock Maximum entropy production as a necessary admissibility condition for
  the fluid {N}avier–{S}tokes and {E}uler equations.
\newblock {\em SN Appl. Sci.}, 2(2160), 2020.

\bibitem{CVY21}
R.~M. Chen, A.~F. Vasseur, and C.~Yu.
\newblock Global ill-posedness for a dense set of initial data to the
  isentropic system of gas dynamics.
\newblock {\em arXiv:2103.04905v1}, 2021.

\bibitem{Ch2014}
E.~Chiodaroli.
\newblock A counterexample to well-posedness of entropy solutions to the
  compressible {E}uler system.
\newblock {\em Journal of Hyperbolic Differential Equations}, 11(03):493--519,
  2021/07/22 2014.

\bibitem{ChDLKr2015}
E.~Chiodaroli, C.~De~Lellis, and O.~Kreml.
\newblock Global ill-posedness of the isentropic system of gas dynamics.
\newblock {\em Communications on Pure and Applied Mathematics}, 58:1157–1190,
  2015.

\bibitem{ChKr2021}
E.~Chiodaroli, O.~Kreml, V.~Mácha, and S.~Schwarzacher.
\newblock Non–uniqueness of admissible weak solutions to the compressible
  {E}uler equations with smooth initial data.
\newblock {\em Transactions of the American Mathematical Society},
  374(4):2269--2295, 2021.

\bibitem{CoETi1994}
P.~Constantin, W.~E, and E.~Titi.
\newblock Onsager's conjecture on the energy conservation for solutions of
  {E}uler's equation.
\newblock {\em Comm. Math. Phys.}, 165(1):207--209, 1994.

\bibitem{CoFr}
R.~Courant and K.~O. Friedrichs.
\newblock {\em Supersonic Flow and Shock Waves}.
\newblock Springer-Verlag New York, 1976.

\bibitem{Dafermos}
C.~M. Dafermos.
\newblock {\em Hyperbolic Conservation Laws in Continuum Physics}.
\newblock Springer-Verlag Berlin Heidelberg, 3 edition, 2010.

\bibitem{DaSz2016}
S.~Daneri and L.~Sz{\'e}kelyhidi.
\newblock Non-uniqueness and h-principle for h{\"o}lder-continuous weak
  solutions of the {E}uler equations.
\newblock {\em Archive for Rational Mechanics and Analysis}, 224(2):471--514,
  2017.

\bibitem{DLK20}
C.~De~Lellis and H.~Kwon.
\newblock On non-uniqueness of {H}\"{o}lder continuous globally dissipative
  {E}uler flows.
\newblock {\em arXiv:2006.06482}, 2020.

\bibitem{DLeSz2010}
C.~De~Lellis and L.~Sz{\'e}kelyhidi.
\newblock On admissibility criteria for weak solutions of the {E}uler equations.
\newblock {\em Archive for Rational Mechanics and Analysis}, 195(1):225--260,
  2010.

\bibitem{DLSz2012}
C.~De~Lellis and L.~Sz{{\'e}}kelyhidi, Jr.
\newblock The {$h$}-principle and the equations of fluid dynamics.
\newblock {\em Bull. Amer. Math. Soc. (N.S.)}, 49(3):347--375, 2012.

\bibitem{DLSz2013}
C.~De~Lellis and L.~Sz{{\'e}}kelyhidi, Jr.
\newblock Dissipative continuous {E}uler flows.
\newblock {\em Invent. Math.}, 193(2):377--407, 2013.

\bibitem{DlSzJEMS}
C.~De~Lellis and L.~Sz\'ekelyhidi, Jr.
\newblock Dissipative {E}uler flows and {O}nsager's conjecture.
\newblock {\em J. Eur. Math. Soc. (JEMS)}, 16(7):1467--1505, 2014.

\bibitem{DE2018}
T.~Drivas and G.~Eyink.
\newblock An {O}nsager singularity theorem for turbulent solutions of
  compressible {E}uler equations.
\newblock {\em Commun. Math. Phys.}, 359:733–763, 2018.

\bibitem{Elling}
V.~Elling.
\newblock A possible counterexample to well posedness of entropy solutions and
  to godunov scheme convergence.
\newblock {\em Mathematics of Computation}, 75(256):1721–1733, 2006.

\bibitem{FGSGW2017}
E.~Feireisl, P.~Gwiazda, A.~Świerczewska Gwiazda, and E.~Wiedemann.
\newblock Regularity and energy conservation for the compressible {E}uler
  equations.
\newblock {\em Arch Rational Mech Anal}, 223:1375–1395, 2017.

\bibitem{GMSG2018}
P.~Gwiazda, M.~Michálek, and A.~Świerczewska Gwiazda.
\newblock A note on weak solutions of conservation laws and energy/entropy
  conservation.
\newblock {\em Arch Rational Mech Anal}, 229:1223–1238, 2018.

\bibitem{Is2013}
P.~Isett.
\newblock {\em Holder continuous {E}uler flows with compact support in time}.
\newblock ProQuest LLC, Ann Arbor, MI, 2013.
\newblock Thesis (Ph.D.)--Princeton University.

\bibitem{Is17}
P.~Isett.
\newblock Nonuniqueness and existence of continuous, globally dissipative euler
  flows.
\newblock {\em arXiv:1710.11186}, 2017.

\bibitem{Is2016}
P.~Isett.
\newblock A proof of {O}nsager's conjecture.
\newblock {\em Annals of Mathematics}, 188(3):871--963, 2018.

\bibitem{IsOh2016}
P.~Isett and S.-J. Oh.
\newblock On nonperiodic euler flows with h{\"o}lder regularity.
\newblock {\em Archive for Rational Mechanics and Analysis}, 221(2):725--804,
  2016.

\bibitem{IsVi2015}
P.~Isett and V.~Vicol.
\newblock H\"older continuous solutions of active scalar equations.
\newblock {\em Annals of PDE}, 1(1):1--77, 2015.

\bibitem{Kato}
T.~Kato.
\newblock The cauchy problem for quasi-linear symmetric hyperbolic systems.
\newblock {\em Archive for Rational Mechanics and Analysis}, 58(3):181--205,
  1975.

\bibitem{Lax}
P.~D. Lax.
\newblock {\em Hyperbolic Systems of Conservation Laws and the Mathematical
  Theory of Shock Waves}.
\newblock Society for Industrial and Applied Mathematics, 1973.

\bibitem{Majda}
A.~Majda.
\newblock {\em Compressible Fluid Flow and Systems of Conservation Laws in
  Several Space Variables}.
\newblock Applied Mathematical Sciences. Springer-Verlag New York, 1984.

\bibitem{Ma2020}
S.~Markfelder.
\newblock Convex integration applied to the multi-dimensional compressible
  {E}uler equations.
\newblock {\em PhD Thesis}, 2020.

\bibitem{MaKl2018}
S.~Markfelder and C.~Klingenberg.
\newblock The {R}iemann problem for the multidimensional isentropic system of gas
  dynamics is ill-posed if it contains a shock.
\newblock {\em Archive for Rational Mechanics and Analysis}, 227:967–994,
  2018.

\bibitem{Nash}
J.~Nash.
\newblock {$C^1$} isometric imbeddings.
\newblock {\em Ann. of Math. (2)}, 60:383--396, 1954.

\bibitem{On1949}
L.~Onsager.
\newblock {S}tatistical hydrodynamics.
\newblock {\em Il {N}uovo {C}imento (1943-1954)}, 6:279--287, 1949.

\bibitem{GJK20}
K.~K. Shyam Sundar~Ghoshal, Animesh~Jana.
\newblock On the uniqueness of solutions to hyperbolic systems of conservation
  laws.
\newblock {\em arXiv:2007.10923}, 2020.

\bibitem{Whitham}
G.~B. Whitham.
\newblock {\em Linear and Nonlinear Waves}.
\newblock Wiley-Intersciences: New York, 1974.

\bibitem{Weid2018}
E.~Wiedemann.
\newblock Weak-strong uniqueness in fluid dynamics.
\newblock {\em Partial Differential Equations in Fluid Mechanics (London
  Mathematical Society Lecture Note Series)}, pages 289--326, 2018.

\end{thebibliography}

\end{document}